\documentclass[aos]{imsart}

\RequirePackage{amsthm,amsmath,amsfonts,amssymb}
\RequirePackage[numbers]{natbib}
\RequirePackage[colorlinks,citecolor=blue,urlcolor=blue]{hyperref}
\RequirePackage{graphicx}
\usepackage{bm}

\setpkgattr{journal}{name}{}	

\usepackage{paralist}

\usepackage{enumitem}

\everymath{ 
	\fontdimen16\textfont2=1.0pt
	\fontdimen17\textfont2=1.0pt
}

\theoremstyle{plain}
\newtheorem{theorem}{Theorem}[section]
\newtheorem{proposition}[theorem]{Proposition}
\newtheorem{lemma}[theorem]{Lemma}
\newtheorem{corollary}[theorem]{Corollary}

\theoremstyle{remark}
\newtheorem{condition}[theorem]{Condition}
\newtheorem{remark}[theorem]{Remark}

\endlocaldefs

\newcommand{\dto}{\stackrel{d}{\longrightarrow}}

\newcommand{\R}{\mathbb{R}}
\newcommand{\Q}{\mathbb{Q}}
\newcommand{\Z}{\mathbb{Z}}
\newcommand{\N}{\mathbb{N}}
\newcommand\Prob{\mathbb{P}}    
\newcommand{\E}{\mathbb{E}}

\newcommand{\ip}[1]{\lfloor #1 \rfloor}

\newcommand{\CC}{\mathbb{C}}

\newcommand{\HH}{\mathbb{H}}
\newcommand{\GG}{\mathbb{G}}
\newcommand{\PP}{\mathbb{P}}

\newcommand{\XX}{\mathbb{X}}
\newcommand{\YY}{\mathbb{Y}}
\newcommand{\XXs}{\mathbb{X}'}

\newcommand{\Dc}{\mathcal{D}}

\newcommand{\Mcc}{\mathcal{M}}
\newcommand{\Nc}{\mathcal{N}}

\newcommand{\Gc}{\mathcal{G}}

\newcommand{\slb}{{\operatorname{sb}}} 		
\newcommand{\djb}{{\operatorname{db}}}		
\newcommand{\mb}{{\operatorname{mb}}}		

\newcommand{\se}{{\operatorname{S1}}}	    
\newcommand{\sz}{{\operatorname{S2}}}	    
\newcommand{\samp}{{\operatorname{S}}}

\newcommand{\ndb}{{m}}		
\newcommand{\nsb}{{n-r+1}}
\newcommand{\scs}{\scriptscriptstyle}

\newcommand{\eps}{\varepsilon}
\newcommand{\abs}[1]{\left\vert{#1}\right\vert}
\newcommand{\norm}[1]{\left\Vert{#1}\right\Vert}
\newcommand{\ind}{\operatorname{\bf{1}}}
\newcommand{\diff}{{\,\mathrm{d}}}

\newcommand{\Exp}{{\mathbb{E}}}
\newcommand{\Var}{\operatorname{Var}}
\newcommand{\Cov}{\operatorname{Cov}}

\newcommand{\emc}{\gamma_{\mathrm{EM}}} 

\newcommand{\RL}{\operatorname{RL}}

\newcommand{\Gmsgk}{G^k_{(\mu, \sigma, \gamma)}}

\newcommand{\dgam}{\frac{1}{\gamma}}

\newcommand{\dnsl}{{n-r+1}}

\renewcommand{\citealp}{\cite}

\begin{document}

\begin{frontmatter}
%
%
%
%

\title{On the Disjoint and Sliding Block Maxima method for piecewise stationary time series}
\runtitle{Block maxima for piecewise stationary time series}

\begin{aug}
	\author[A]{\fnms{Axel} \snm{Bücher}\corref{}\ead[label=e1]{axel.buecher@hhu.de}}
	\and
	\author[B]{\fnms{Leandra} \snm{Zanger}\corref{}\ead[label=e2]{Leandra.Zanger@hhu.de}}
	\address[A]{Heinrich-Heine-University, Düsseldorf, \printead{e1}}
	
	\address[B]{Heinrich-Heine-University, Düsseldorf, \printead{e2}}
\end{aug}

\begin{abstract}
	Modeling univariate block maxima by the generalized extreme value distribution constitutes one of the most widely applied approaches in extreme value statistics. It has recently been found that, for an underlying stationary time series,
 respective estimators may be improved by calculating block maxima in an overlapping way. A proof of concept is provided that the latter finding also holds in situations that involve certain piecewise stationarities. A weak convergence result for an empirical process of central interest is provided, and, as a case-in-point, further details are worked out explicitly for the probability weighted moment estimator.
Irrespective of the  serial dependence,  the estimation variance is shown to be smaller for the new estimator, while the bias was found to be the same or vary comparably little in extensive simulation experiments. The results are illustrated by Monte Carlo simulation experiments and are applied to a common situation involving temperature extremes in a changing climate.
\end{abstract}



\end{frontmatter}


\section{Introduction}

The annual or seasonal maximum of a certain variable of interest  is a common target distribution, in particular in environmental statistics \citep{Kat02, BeiGoeSegTeu04}. For instance, hydrologists are interested in maximal river discharges to facilitate flood protection, while meteorologists and climatologists study maximal temperatures, precipitation or wind speeds, collected over certain spatial or temporal regions. The latter comprises the emerging field of extreme event attribution studies, see \cite{Sto16}, which aim at exploring how the probability of certain extreme events evolve in the context of a changing climate due to anthropogenic activities.

The underlying statistical principle is known as the block maxima method and dates back to \cite{Gum58}, see also the monographs \cite{Col01, BeiGoeSegTeu04}. In its simplest form, it is postulated that a sample of successive (annual) block maxima constitutes an independent and identically distributed (i.i.d.)\ sample from the generalized extreme value (GEV) distribution, as suggested by the asymptotics formulated in the Fisher-Tippet-Gnedenko Theorem \citep{FisTip28}. The model may then be fitted by any method of choice, the most popular approaches being maximum likelihood \citep{PreWal80, BucSeg17} and the probability weighted moment (PWM) method \citep{HosWal85}. 

Considering the validation of statistical methodology (like proving consistency and asymptotic normality of estimators), it has long been assumed that the block maxima sample is a genuine independent sample from the GEV-distribution.  From a mathematical viewpoint, this assumption seems overly simplified: neither does it allow to quantify a possible bias due to the fact that block maxima are only asymptotically GEV-distributed, nor does it quantify to what extent possible temporal dependencies in the underlying sample are negligible. 
Notable exceptions are \cite{Dom15, FerDeh15, DomFer19}, who investigate respective methods under the assumption that block maxima of size $r=r_n\to\infty, r_n=o(n)$ are calculated based on an underlying i.i.d.\ series of length $n$ (corresponding to, say, daily  observations). The latter is however still not really fitting to typical applications of the block maxima method, where serial independence of a daily time series is rarely the case (another nuisance are potential seasonalities, which will be discussed below). Extensions to the case of a strictly stationary time series have been worked out in \cite{BucSeg14, BucSeg18a}, for the estimation of extreme value copulas in a multivariate context and estimation of Fr\'echet parameters in a univariate heavy tailed situation, respectively. The new viewpoint has also lead to methodological improvements, as it allows to study estimators which are based on block maxima calculated from  sliding (overlapping) blocks of observations. Perhaps surprisingly, respective estimators have been shown to be more efficient than their disjoint blocks counterparts in certain general situations \citep{BucSeg18b, ZouVolBuc21}. In the i.i.d\ case, \cite{OorZho20} recently provided a further methodological improvement based on what has been called the all-block maxima method; a method that is, however, not easily transferable to the time series case except the extremal index \citep{Lea83} is one. Furthermore, \cite{DreNeb21, CisKul21} study the use of sliding blocks with POT-type estimators.

This paper's main contribution is a surprising proof of concept that the sliding block maxima method may even yield more efficient estimators when applied to 
datasets that result in a non-stationary behavior of the sample of sliding block maxima. More precisely, suitable asymptotic theory is developed for a sampling scheme that involves an underlying triangular array consisting of independent and identically distributed stretches of observations extracted from a stationary time series. The framework is designed to asymptotically mimic the practically relevant situation encountered in environmental statistics where, due to seasonalities, stationarity can only be (approximately) guaranteed for, say, daily observations collected throughout the summer months.

Under the predescribed sampling scheme, as well as under a classical sampling scheme involving a plain stationary time series, asymptotic theory is developed for 
(1) an empirical process of pure theoretical interest as well as for (2) the PWM estimator of practical interest. It is worthwhile to mention that the restriction to PWM estimators is partly arbitrary, and that similar findings can be expected to hold for other estimators of practical interest. One of the reasons we opted for PWM is that we extend, as a by-product, results from \citep{FerDeh15} on the disjoint blocks maxima PWM estimator in an i.i.d.\ context. 

The asymptotic results are similar but not the same as in \citep{BucSeg18b, ZouVolBuc21}: it is found that, despite non-stationarity, the sliding blocks method \textit{works} and yields smaller asymptotic variances than the disjoint blocks method. 
However, the asymptotic bias is only guaranteed to be the same for stationary data. Within extensive simulation experiments on the PWM estimators, it is found that the overall improvement of the sliding blocks version over its disjoint blocks counterpart is remarkably large for negative shape parameters, while only small improvements are visible for positive shapes. The improvement for negative shape parameters is illustrated in Figure~\ref{fig:intro}, where we depict the mean squared estimation error for the estimation of the shape parameter $\gamma$ for a fixed block size $r=90$ (roughly corresponding to the length of a season) and increasing number of seasons.

\begin{figure}[t!]	
	\centering
	\makebox{\includegraphics[width=0.75\textwidth]{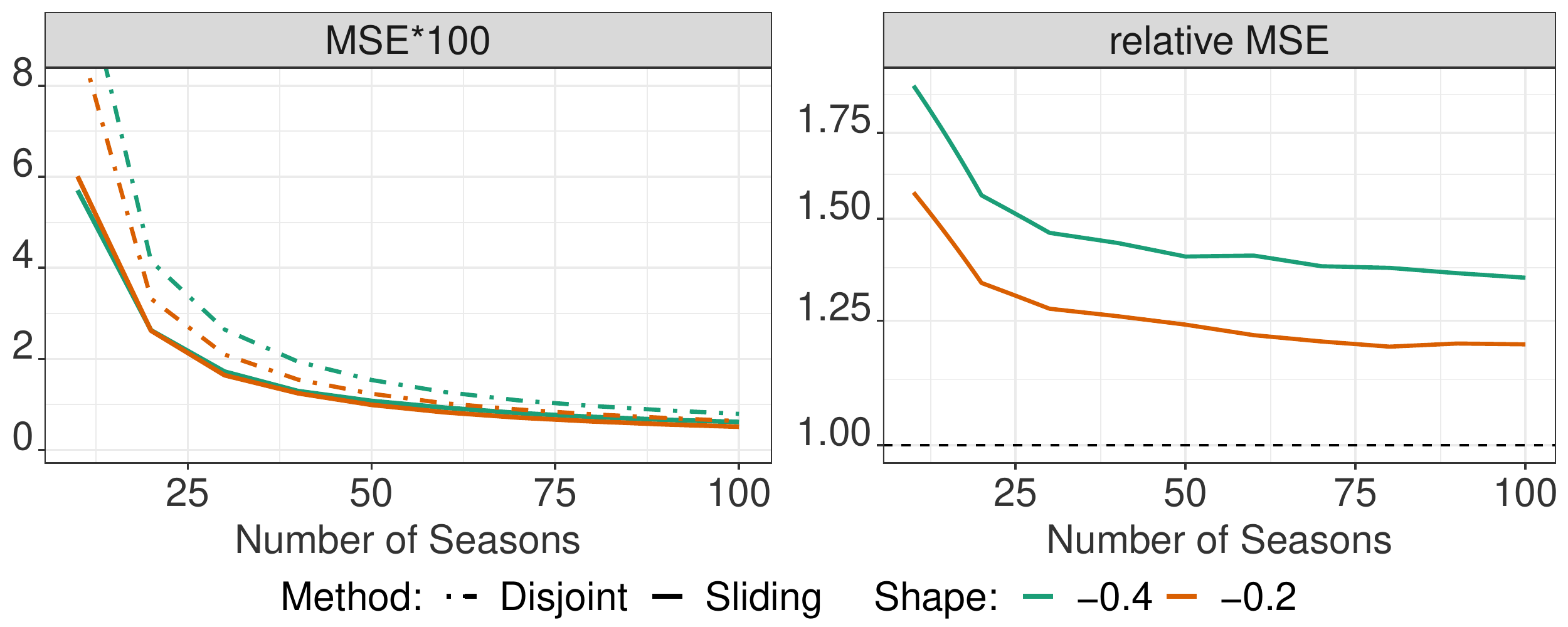}}\vspace{-.3cm}
	\caption{\label{fig:intro}
		Mean squared error for the estimation of $\gamma$ for the disjoint and sliding blocks estimator (left), and the ratio $\mathrm{MSE(disjoint)} / \mathrm{MSE(sliding)}$ (right). The data generating process is an AR(0.5)-GPD($\gamma$)-model, sampling scheme (S2) with fixed block size $r=90$, as 
		described in Section~\ref{sec:sim}.}
\end{figure}

Negative shape parameters are often found when analyzing temperature extremes, for which shapes are typically within the range $-0.4$ to $-0.2$. A respective case study is worked out, where we also deal with non-stationarities in the location parameter of the GEV-model. The considered model is commonly employed in extreme event attribution studies, see \cite{philip2020protocol}. A parametric bootstrap device is proposed to assess estimation uncertainties. 

Last but not least, we would like to make aware of the fact that the (sliding) blocks method exhibits an important methodological advantage over the competing peaks-over-threshold (POT) approach \citep{DehFer06} when the ultimate goal consists of assessing  return levels or return periods. Indeed, for the latter purpose, methods based on the POT approach typically require an application of a declustering approach (or an estimator of the extremal index) to take care of the time series' serial dependence. This is not necessary for the (sliding) blocks method, where the serial dependence only shows up in the scaling sequences associated with the max-domain of attraction condition (see Condition~\ref{cond:mda} below for details), both of which are automatically estimated by the method. Respective details are worked out in Section~\ref{sec:rl}.

This paper is organized as follows. Section~\ref{sec:ss} contains details on the basic model assumptions and a weak convergence result on an empirical process of central interest. 
Details on the PWM estimator are worked out in Section~\ref{sec:pwm}. 
A large scale Monte Carlo simulation study is presented in Section~\ref{sec:sim}. The case study on temperature extremes can be found in Section~\ref{sec:case}, followed by a conclusion in Section~\ref{sec:con}. The most important proofs are worked out in Sections~\ref{sec:proofs} and \ref{sec:proofsaux}, with some less central parts postponed to a supplementary material. Theoretical results from the supplement are numbered by capital letters; e.g., Lemma~\ref{lem:jointsl}.
All convergences are for $n\to\infty$, if not mentioned otherwise. The generalized (left-continuous) inverse of a cumulative distribution function (c.d.f.) $F$ is denoted by $F^\leftarrow$.

\section{A new sampling scheme and some general theoretical results} \label{sec:ss}

Recall the Generalized Extreme Value (GEV) distribution with parameters $\mu$ (location), $\sigma$ (scale) and $\gamma$ (shape), defined by its cumulative distribution 
function 
\[
G_{(\mu, \sigma, \gamma)}(x) := \exp\left[-\left\{1+\gamma\left(\frac{x-\mu}{\sigma}\right)\right\}^{-\frac{1}{\gamma}}\right] , \qquad 1+\gamma\,  \frac{x-\mu}{\sigma} >0.
\]
If $\theta=(\mu,\sigma,\gamma)' = (0,1,\gamma)'$, we will use the abbreviation $G_{(0,1,\gamma)} = G_\gamma$. The support of $G_\gamma$ is denoted by $S_\gamma=\{x \in \R: 1+\gamma x>0\}$.

An extension of the classical extremal types theorem to strictly stationary time series \citep{Lea83} implies that, under suitable conditions, affinely standardized maxima extracted from a stationary time series converge to the GEV-distribution. We make this an assumption, and additionally require the scaling sequences to exhibit some common regularity inspired by the max-domain of attraction condition in the i.i.d.\ case \citep{DehFer06}.

\begin{condition}[Max-domain of attraction] \label{cond:mda}
	Let $(X_t)_{t\in\Z}$ denote a stationary time series. There exist sequences $(a_r)_{r} \subset \R_{+}, (b_r)_r \subset \R$  and $\gamma\in\R$, such that, for any $s>0$,
	\begin{align}\label{eq:rvscale}
	\lim\limits_{r \to \infty}\frac{a_{\lfloor rs \rfloor}}{a_r}  = s^\gamma, 
	\qquad 
	\lim\limits_{r \to \infty} \frac{b_{\lfloor rs \rfloor} -b_{r}}{a_{r}} = \frac{s^{\gamma} -1}{\gamma}, 
	\end{align} 
	where the second limit  is interpreted as $\log(s)$ if $\gamma=0$. Moreover,   	
	for $r \to \infty$, 			
	\begin{align}\label{eq:firstorder} 
	Z_r = \frac{ \max(X_1, \ldots, X_r) - b_r}{a_r} \dto Z \sim G_{\gamma}.
	\end{align}
\end{condition} 

Note that (1) and (2) may typically be deduced from existence and positivity of the extremal index, see, e.g.,  Theorem 10.4 in \cite{BeiGoeSegTeu04}, and that its existence is known for many common time series models.

The max-domain of attraction condition allows to formulate two sampling mechanisms used throughout this paper.

\begin{condition}[Observation scheme] \label{cond:obs}
	For sample size $n\in\N$, we have observations $X_{1,n}, \dots,  X_{n,n}$ that do not contain ties with probability one, such that either (S1) or (S2) holds, where:
	\begin{compactenum}
		\item[(S1)] $(X_{1,n}, \dots, X_{n,n})=(X_1, \dots, X_n)$ is an excerpt from a strictly stationary time series satisfying Condition~\ref{cond:mda} with continuous marginal c.d.f $F$.
		\item[(S2)] For some block length sequence $(r_n)_n \subset \N$ diverging to infinity such that $r_n=o(n)$, we have 
		\begin{multline*}
		(X_{1,n}, \dots, X_{n,n}) = (Y_{1}^{(1)}, \dots, Y_{r_n}^{(1)}, 
		Y^{(2)}_{1}, \dots, Y^{(2)}_{r_n}, \dots \\
		\dots, 
		Y_{1}^{(\ndb )}, \dots, Y_{r_n}^{(\ndb )}, Y_{1}^{(\ndb +1)}, \dots, Y_{n-\ndb r_n}^{(\ndb +1)}),
		\end{multline*}
		where $\ndb = \ndb_n =\lfloor n/r_n \rfloor$ and where $(Y_t^{(1)})_t, (Y_t^{(2)})_t, \dots$ denote i.i.d.\ copies from a stationary time series satisfying Condition~\ref{cond:mda} with continuous marginal c.d.f $F$. Note that $Y^{\scs (j)}_t$ should be regarded as the $t$th observation in the $j$th season.	\end{compactenum}
\end{condition}

Sampling scheme (S2) shall represent typical environmental applications which are subject to seasonalities. The parameter $r_n$ may correspond to, say, the number of daily observations within the summer months. For such a situation, it appears reasonable to assume strict stationarity within a particular summer, and stochastic independence and distributional equality between multiple summers. In order to obtain meaningful asymptotic results, which in particular cover a sliding blocks version, $r_n$ must be assumed to go to infinity.

\begin{remark}[Possible relaxations of Condition~\ref{cond:obs}]
It is worthwhile to mention that sampling scheme (S2) has been chosen as a starting point for this paper because it is, on the one hand, reasonably general to capture typical real data situations and, on the other hand, simple enough to allow for accessible proofs.  It may be extended in various ways:
first of all, different `seasons' may be assumed to be serially dependent and to satisfy certain mixing conditions; the necessary changes in the proofs would mostly require bringing together arguments from the (S1) and (S2) case. Next, sampling schemes (S1) and (S2) may be subsumed under a more general condition: denoting by $\text{S}(s_n)$ the sampling scheme that consists of concatenating independent `seasonal blocks' of size $s_n$, we observe that (S1) is the same as $\text{S}(n)$, while (S2) is the same as $\text{S}(r_n)$. At the cost of more sophisticated conditions and proofs (in particular,  a more complex version of Condition~\ref{cond:bias2} below would be needed), one may extend the results in this paper to the case where $r_n / s_n \to c \in (0,1]$ (which, in practice, may represent monthly maxima, $r_n=30$, after concatenating seasons, $s_n=90$). The higher level of complexity needed for handling this situation results from the fact that even the disjoint blocks maxima sample may not be stationary anymore (for instance, if $s_n = 1.5 r_n$). Finally, the no-tie assumption in Condition~\ref{cond:obs} is merely made for convenience. At the cost of more sophisticated proofs and conditions on the serial dependence, it can possibly be dispensed with. Similar arguments apply to the continuity assumption on $F$.
\end{remark}

\subsection{Two approximate block maxima samples from the GEV distribution}
\label{sec:sample}

Subsequently, we write $X_{i} = X_{i,n}$ for simplicity. The block maxima method with block size parameter $r=r_n \in \{1, \dots, n\}$ is traditionally based on the sample of disjoint block  maxima $\mathcal M^{\scs (\djb)}_{r}  =\mathcal M^{\scs (\djb)}_{r,n}  = \{M_{r,1}^{\scs (\djb)}, \dots, M_{r,\ndb}^{\scs (\djb)}\}$, where the block maxima are defined by
\[
M_{r,j}^{(\djb)} :=  \max(X_{(j-1)r +1}, \ldots, X_{jr}), \qquad j \in\{1, \ldots, \ndb  \},
\]
and where $\ndb=\ndb_n =\lfloor n/r_n\rfloor$ denotes the number of disjoint blocks that fit into $\{1, \dots, n\}$. For data arising from one of the sampling schemes in Condition~\ref{cond:obs}, it follows that  the sample $\mathcal M^{\scs (\djb)}_{r,n}$ is stationary with marginal c.d.f.\ denoted by 
\begin{align} 
\label{eq:fr}
F_r(x) = \Prob( \max(X_1, \dots, X_r) \le x), \quad x \in \R.
\end{align}
As a consequence of Condition \ref{cond:mda}, we have 
$F_r(x) \approx G_{(b_r, a_r, \gamma)},$
whence the parameters $(b_r, a_r,\gamma)$ may be estimated by any method of choice for fitting the GEV-distribution. 

As mentioned in the introduction, the sample of sliding block maxima $\mathcal M^{\scs (\slb)}_{r}  = \mathcal M^{\scs (\slb)}_{r,n}  = \{M_{r,1}^{\scs (\slb)}, \dots, M_{r,n-r+1}^{\scs (\slb)}\}$ defined by 
\[
M_{r,j}^{(\slb)} :=  \max(X_{j}, \ldots, X_{j+r-1}), \qquad j \in\{1, \ldots, n-r+1 \}
\]
provides an attractive alternative to the sample $\mathcal M^{\scs (\djb)}_{r,n}$. In fact, under sampling scheme (S1), we have $M_{r,j}^{\scs (\slb)} \sim F_r$ for all $j \in\{1, \ldots, n-r+1 \}$ as well, whence respective estimators can be expected to work, in particular when based on the method of moments. Note that the asymptotic analysis becomes substantially more difficult due to the strong serial dependence between the sliding block maxima.

In this paper, we also advocate the use of sliding block maxima under the possibly more realistic sampling scheme~(S2). Compared to (S1), an obstacle occurs:
the c.d.f.\ of $M_{r,j}^{\scs (\slb)}$, 
\begin{align}\label{eq:def:Frj}
F_{r,j}(x) = \PP(\max(X_{j}, \ldots, X_{j+r-1}) \leq x), \qquad x\in \R,
\end{align}
is in general no longer independent of $j$. Perhaps surprisingly, it can be shown  that $F_{r,j} \approx G_{(b_r, a_r,\gamma)}$ for all $j$ and sufficiently large $r$: 
	\begin{lemma}[Asymptotic stationarity of sliding block maxima]
		\label{lem:weakdf_S2} 
		Suppose one of the sampling schemes from Condition \ref{cond:obs} is met. Then, for every $\xi \in [0,1]$ and $z\in\R$,  
		\[
		\lim_{r \to \infty} 
		F_{r,1+\ip{r\xi}}\left( a_r z + b_r \right) 
		=
		G_\gamma(z).
		\]
	\end{lemma}

As a consequence of this lemma, estimators based on $\mathcal M^{\scs (\slb)}_{r,n}$ can still be expected to work under (S2). Further extensions to joint convergence of two block maxima are provided in Lemmas~\ref{lem:jointsl}  and~\ref{lem:jointsl2}.

\subsection{An empirical process associated with rescaled block maxima}

A central theoretical ingredient for all subsequent results  (and, presumably, for possible future results on other estimators involving the sliding block maxima method) is weak convergence of the centered empirical process associated with the empirical distribution function of the (unobservable) rescaled block maxima samples  $Z_{r,1}^{\scs (\djb)}, \dots, Z_{r,\ndb }^{\scs (\djb)}$ and  $Z_{r,1}^{\scs (\slb)}, \dots, Z_{r,n-r+1}^{\scs (\slb)}$, where 
		\begin{align} \label{eq:zrjm}
		Z_{r, j}^{(\djb)} &= \frac{ M_{r,j}^{(\djb)} - b_r}{a_r}, 
		\qquad
		Z_{r, j}^{(\slb)} = \frac{ M_{r,j}^{(\slb)} - b_r}{a_r}.
		\end{align}	
Throughout its proof, we are going to apply common blocking techniques, whence the block length $r$ must be well-adapted to the serial dependence of the time series. Suitable control may be provided by   mixing conditions that were also imposed in related situations in \cite{BucSeg14,BucSeg18a, BucSeg18b}.

\begin{condition}\label{cond:rl}
	For the block size sequence $(r_n)_n$ it holds that
	\begin{enumerate}[label=(i)]\itemsep2pt
		\item[(i)] $r_n \rightarrow \infty  $ and $r_n =o(n)$.
		\item[(ii)] $\bigl(\frac{n}{r_n}\bigl)^{1/2} \beta(r_n) = o(1)$ and $\bigl(\frac{n}{r_n}\bigr)^{1+\omega} \alpha(r_n) = o(1) $ for some $\omega>0$.
		\item[(iii)] There exists a sequence $(\ell_n)_n \subset \N$ such that $\ell_n \rightarrow \infty $, $\ell_n = o(r_n),  \frac{n}{r_n} \alpha(\ell_n) = o(1)$ and $ \frac{r_n}{\ell_n}  \alpha(\ell_n) = o(1)$.
	\end{enumerate}
\end{condition}  
 
Here, $\alpha$ and $\beta$ denote the $\alpha$- and $\beta$-mixing coefficients of the time series  $(X_t)_t$ that was introduced in Condition~\ref{cond:mda} (see \citep{Bra05} for a precise definition).
Note that Conditions (ii) and (iii) imply that the block length sequence $r_n$ must not be too small. 

Subsequently, for $z\in\R$,  let
	\begin{align*}
		\hat H_r^{(\djb)}(z) 
		&=
		\frac{1}{\ndb} \sum_{j=1}^{\ndb} \ind(Z_{r,j}^{(\djb)} \le z), \quad
		\hat H_r^{(\slb)}(z) 
		=
		\frac{1}{\nsb} \sum_{j=1}^{\nsb} \ind(Z_{r,j}^{(\slb)} \le z)
		\end{align*}
		and
			\begin{align} \label{eq:barhr}
	\bar{H}_r(z) = \frac{1}{r} \sum_{j=1}^r H_{r,j}(z), \quad  H_{r,j}(z) = \PP( Z_{r,j}^{(\slb)} \le z).
	\end{align} 
Note that $\Exp[\hat H_r^{(\mb)}(z) ]=\bar H_r(z)$, unless $\mb =\slb$ and sampling scheme $(\sz)$ is met, in which case we have equality up to negligible terms. The following central result is similar to Theorem~2.10 in \cite{ZouVolBuc21}, despite under different assumptions (in particular sampling scheme $(\sz)$).

\begin{theorem} \label{theo:weakh}
	Consider one of the sampling schemes from Condition~\ref{cond:obs}. Under Condition~\ref{cond:rl}, we have for $\mb \in \{ \djb, \slb\}$
	\begin{align*}
	\HH_r^{(\mb)} = \sqrt{\frac{n}r} \left( \hat{H}_{r}^{(\mb)} -  \bar H_r\right) \dto \HH^{(\mb)} = \CC^{(\mb)} \circ G_\gamma 
	\end{align*}
	in $\ell^\infty(\R)$ equipped with the supremum metric, where $\CC^{(\djb)}$ is a standard Brownian bridge on $[0,1]$ and where $\CC^{(\slb)}$ is a centered Gaussian process with covariance function 
	\begin{align}\label{eq:covCCsl}
	\Cov(\CC^{(\slb)}(u), \CC^{(\slb)}(v)) 
	= 2\left(\frac{uv - u\wedge v}{\ln(u\vee v)} - uv\right), \qquad u,v \in (0,1).
	\end{align}
	Moreover, the limit processes $\HH^{(\mb)}$ are almost surely contained in $C_b(\R)$ (the space of continuous and bounded real-valued functions on $\R$) and satisfy
	\begin{align} \label{eq:varl11}
	 \Cov \big(\CC^{(\slb)}(u_1), \dots, \CC^{(\slb)} (u_d) \big) 
	 \le_L 
	 	\Cov \big(\CC^{(\djb)}(u_1), \dots, \CC^{(\djb)} (u_d) \big) 
	\end{align}
	for all $u_1, \dots, u_d \in (0,1)$ and $d\in\N$, 
	where $\le_L$ denotes the Loewner-ordering between symmetric matrices.
\end{theorem}

The convergence of the finite-dimensional distributions in Theorem~\ref{theo:weakh} is a consequence of a more general multivariate central limit theorem, see Theorem~\ref{theo:normallbl}, which might be of independent interest.

\section{PWM estimators based on the block maxima method} \label{sec:pwm}

Throughout this section, let $M$ denote a GEV-distributed random variable with parameter $\theta=(\mu, \sigma, \gamma)'$. For $\gamma<1$, the first three probability weighted moments of $M$ are given by 
\begin{align}\label{eq:pwm}
\beta_{\theta, k} := \E[ M\Gmsgk(M)] = \frac{1}{k+1}\Big[ \mu - \frac{\sigma}{\gamma} \left\{ 1-(k+1)^\gamma \Gamma(1-\gamma)\right\} \Big],
\end{align}
where $ k\in\{ 0,1,2\}$. As shown by \cite{HosWal85}, we obtain the following equation system between $\theta$ and $(\beta_{\theta,0}, \beta_{\theta,1}, \beta_{\theta,2})$:
\begin{align} \label{eq:eqs}
\begin{cases}
\ \gamma = g_1^{-1} \Big( \frac{3\beta_{\theta, 2}- \beta_{\theta, 0}}{2\beta_{\theta, 1} - \beta_{\theta, 0}} \Big) \\
\  	\sigma = g_2(\gamma)\left(2\beta_{\theta, 1} - \beta_{\theta,0}\right) \\
\	 \mu  = \beta_{\theta,0}+ \sigma g_3( \gamma)  
\end{cases}
\end{align}
where
\begin{align*}
g_1(\gamma) = \frac{3^\gamma -1 }{2^\gamma -1},
\qquad
g_2(\gamma) =  \frac{\gamma}{\Gamma(1-\gamma)(2^\gamma -1) },
\qquad
g_3(\gamma) = \frac{1-\Gamma(1-\gamma)}{\gamma}
\end{align*}
with $g_1(0) = {\log 3}/{\log2}, g_2(0)=1/{\log2}$ and $g_3(0)=-\emc$ defined by continuity. Here, $\Gamma$ denotes the Gamma function and $\emc$ is the Euler-Mascheroni constant. The PWM estimator is then defined by replacing the respective moments on the right-hand side of \eqref{eq:eqs} by empirical versions and successively solving for $\gamma, \sigma$ and $\mu$. Several (asymptotic equivalent) empirical versions suggest itself, and throughout this paper we opt for the version proposed in \cite{Lan79}, that is, 
\begin{align*} 
\hat \beta_{0}(\Mcc) = \frac{1}n \sum_{i=1}^n M_i, 
\quad
\hat \beta_{1}(\Mcc) = \frac{1}n \sum_{i=1}^n \frac{i-1}{n-1} M_{(i)},
\quad
\hat \beta_{2}(\Mcc) = \frac{1}n \sum_{i=1}^n \frac{(i-1)(i-2)}{(n-1)(n-2)} M_{(i)},
\end{align*}
where $M_{(1)} \le \dots \le M_{(n)}$ is the order statistic of a sample $\Mcc=\{M_1, \dots, M_n\}$ which is to be fitted to the GEV-distribution. It is worthwhile to mention that these estimators are unbiased in case $\mathcal M$ is an i.i.d.\ sample. Indeed, we may rewrite
$
\hat \beta_1(\mathcal M) = 
\{n(n-1)\}^{-1}\sum_{i \ne j} M_i \ind(M_j \le M_i), 
$
whence $\Exp[\hat \beta_1(\mathcal M)] = \Exp[M_i \ind (M_j \le M_i)] = \beta_{\theta,1}$, and a similar calculation can be made for $\hat \beta_2(\mathcal M)$. The resulting estimator for $\theta$ based on solving \eqref{eq:eqs} will be denoted by $\hat \theta(\Mcc)$. The estimators of ultimate interest in this paper are
\begin{align}\label{eq:pwmest}
\hat \theta^{(\djb)}_{r} = \hat \theta(\mathcal M^{(\djb)}_{r,n}),
\qquad
\hat \theta^{(\slb)}_{r} = \hat \theta(\mathcal M^{(\slb)}_{r,n}),
\end{align}
which are derived from the empirical weighted moments 
$\hat \beta_{r,k}^{\scs (\mb)}=\hat \beta_k(\mathcal M^{\scs (\mb)}_{r,n})$ for $\mb\in\{\djb, \slb\}$
and are to be considered as estimators for $\theta_r=(b_r,a_r,\gamma)'$.

\begin{remark}[Bias-reduced sliding blocks estimator]\label{rem:biasred} 
	Well-known heuristics suggest that block maxima are asymptotically  independent when calculated based on non-overlapping time periods, and that they are asymptotically dependent otherwise (see Lem\-ma~\ref{lem:jointsl} for a rigorous result). As a consequence, the sliding blocks empirical PWMs from \eqref{eq:pwmest} may exhibit a certain `dependency' bias. To remove this bias, one may alternatively consider the estimators $\tilde \beta_{r,0}^{(\slb)} = \hat \beta_{r,0}^{(\slb)}$,
	\begin{align*}
	\tilde \beta_{r,1}^{(\slb)} 
	&=  
	\frac{1}{|D_n(2)|}\sum_{(i,j) \in D_n(2)} M_{r,i}^{(\slb)} \ind(M_{r,j}^{(\slb)} \le M_{r,i}^{(\slb)}) \\
	\tilde \beta_{r,2}^{(\djb)} 
	&=  
	\frac1{|D_n(3)|} \sum_{(i,j,j') \in D_n(3)} M_{r,i}^{(\slb)}  \ind(M_{r,j}^{(\slb)} \le M_{r,i}^{(\slb)}) \ind(M_{r,j'}^{(\slb)} \le M_{r,i}^{(\slb)})
	\end{align*}
	where $D_n(2)$ denotes the  set of all pairs $(i,j)\in\{1, \dots, \nsb\}^2$ such that $I_i \cap I_j= \varnothing$ and where $D_n(3)$ is the set of all triples $(i,j,j')\in\{1, \dots, \nsb\}^3$ such that $I_i\cap I_{j}=I_i \cap I_{j'}= I_{j} \cap I_{j'}=\varnothing$, with $I_i=\{i, \dots, i+r-1\}$. 
	Obviously, the larger the block size $r$, the more $\tilde \beta_{\scs r,k}^{\scs (\slb)}$ deviates from $\hat \beta_{\scs r,k}^{\scs (\slb)}$. The difference between the two estimators is asymptotically negligible though, while the computational cost is substantially higher  for the tilde version.
\end{remark}

\subsection{Asymptotic normality of PWM estimators}	 \label{sec:main}
Before formulating explicit results, it is worthwhile to mention that asymptotic theory involving PWM estimators has hitherto been mostly worked out under the simplifying assumption that the (disjoint) block maxima provide a genuine i.i.d.\ sample from the GEV distribution (as noted in the introduction, \citealp{FerDeh15} is a notable exception). The alternative viewpoint based on Condition~\ref{cond:obs} has at least three important advantages: it allows to explicitly describe potential bias terms, it does not neglect serial dependence between successive blocks (sampling scheme (S1)), and, perhaps most importantly, it makes possible the treatment of the more efficient sliding blocks version.

A number of regularity conditions is needed to derive consistency and asymptotic normality of the estimators in \eqref{eq:pwmest}.

\begin{condition}[Bias]\label{cond:bias2}
	For $k\in\{0,1,2\}, \mb \in \{\slb, \djb\}$ and $\samp \in \{\se,\sz\}$, the limit 	
	\begin{align*}
	B_{k}^{(\mb, \samp)} = \lim\limits_{n\to \infty} B_{n,k}^{(\mb, \samp)}, 
	\end{align*}
	exists, where
	\begin{align*}
	B_{n,k}^{(\mb, \samp)} 
	= \begin{cases}  
	\displaystyle 
	\sqrt{\frac{n}r} \Big\{\Exp[Z_rH_r^k(Z_r)] - \Exp[Z G_\gamma^k(Z)] \Big\}
	, & (\mb, \samp)\ne(\slb, \sz), \\
	\displaystyle 
	\sqrt{\frac{n}r} \frac{1}r \sum_{j=1}^r  \Big\{\Exp\Big[Z_{r,j}^{(\slb)} \bar{H}_r^k(Z_{r,j}^{(\slb)}) \Big)\Big] - \Exp\Bigl[Z G_\gamma^k(Z)\Bigr] \Big\}, & (\mb, \samp)=(\slb, \sz),
	\end{cases}
	\end{align*}	
	where $Z_r$ from \eqref{eq:firstorder} has c.d.f.\ $H_r$, where $Z\sim G_\gamma$ and where $Z_{r,j}^{(\slb)}$ and $\bar H_r$ are defined in \eqref{eq:zrjm} and \eqref{eq:barhr}, respectively.
\end{condition}

It is worthwhile to mention that $B_{n,k}^{(\slb, \se)}=B_{n,k}^{(\slb, \sz)}$ provided the underlying time series is serially independent; in fact, the entire sampling schemes coincide in this case. In typical cases of serial dependence, the simulation experiments in Section~\ref{sec:sim} suggest that the difference between the two limits is  small.

\begin{condition}[Uniform integrability]\label{cond:unifint}
There exists some $ \nu >\frac{2}{\omega}$ with $\omega$ from Condition~\ref{cond:rl}(ii)
such that 
\[
	\limsup\limits_{r\to\infty} \E\left[ \vert Z_{r}\vert^{2+ \nu} \right] < \infty.
\]
\end{condition}

The condition is used to deduce convergence
	of moments from convergence in distribution, which, for certain moments, is needed in view of
	the fact that the PWM estimators are based on the method of moments.
	
\begin{remark} In case $\gamma>0$, Condition~\ref{cond:unifint} together with Condition~\ref{cond:obs} and~\ref{cond:rl} implies additional constraints on $\gamma$ and $\omega$. 
Indeed, observing that $\Exp[|Z|^{2+\nu}]< \infty$ iff $\nu < 1/\gamma-2$, Condition~\ref{cond:unifint} can only be satisfied if $\gamma<1/2$. Further, since $\nu>2/\omega$, we must have $2/\omega<1/\gamma-2$, which is equivalent to $\omega>(2\gamma)/(1-2\gamma)$.
\end{remark}

Our main result will be a corollary of the following theorem on the joint asymptotic properties of the empirical probability weighted moments. The following notations are needed for its formulation:  let $f_0(x)=x$ and 
\begin{align} \label{eq:fk}
\begin{split}
&f_1(x) = xG_\gamma(x) + \E[Z\ind(Z>x)],  \qquad f_2(x) =  xG^2_\gamma(x) + 2\E[ZG_\gamma(Z)\ind(Z>x)],
\end{split}
\end{align}
where $Z\sim G_\gamma$ (note that the dependence on $\gamma$ is suppressed in the notation $f_k$). Moreover, let $G_{\gamma, \xi}(x,y) = G_{\gamma}(x)G_\gamma(y)$
for $\xi>1 $ and 
\begin{align}\label{def:Ggamxi}
G_{\gamma, \xi }(x,y) 
= 
\exp\left[ - \left\{ \xi(1+\gamma x)^{-\dgam} + \xi(1+\gamma y)^{-\dgam} + (1-\xi)(1+\gamma (x \wedge y))^{-\dgam}\right\}\right],
\end{align}
for $\xi\in[0,1]$, 
where $(x,y)$ is such that $1+\gamma x>0$ and $1+\gamma y>0$. Note that $G_{\gamma, \xi }$ defines a bivariate extreme value distribution with marginal c.d.f.s $G_{\gamma}$, irrespective of $\xi$, and with Pickands dependence function $ A_\xi(w) = (1\wedge \xi) + \{1-(1\wedge \xi)\} \{ w \vee (1-w) \}$.

\begin{theorem} \label{theo:pwm1}
	Suppose one of the sampling schemes from Condition~\ref{cond:obs} is met with $\gamma<1/2$. Further, assume that Conditions ~\ref{cond:rl}, \ref{cond:bias2} and \ref{cond:unifint} are met, and write $\theta_r=(b_r, a_r, \gamma)'$ with respective  PWMs $\beta_{\theta_r, k}$.  Then, for $\mb \in \{ \djb, \slb\}$ and $\samp \in (\se,\sz)$,
	\begin{align} \label{eq:norm}
	\bigg(\sqrt{\frac{n}{r}}  \bigg(\frac{\hat \beta_{r,k}^{(\mb)} - \beta_{\theta_r, k}}{a_r}\bigg)\bigg)_{k=0,1,2} \dto 
	\Nc_3(\bm B^{(\mb, \samp)}, \bm\Omega^{(\mb)}),
	\end{align}
	where $\bm{B}^{(\mb, \samp)} = (B_k^{(\mb, \samp)})_{k=0,1,2}, \bm \Omega^{(\mb)} = (\bm {\Omega }^{(\mb)}_{k,k'})_{k,k'=0,1,2}$ and where, with $Z\sim G_{\gamma}$ and $(Z_{1\xi}, Z_{2\xi})\sim G_{\gamma, \xi}$,
	\[
	\bm {\Omega }^{(\djb)}_{k,k'} = \Cov(f_k(Z), f_{k'}(Z)), 
	\qquad 
	\bm {\Omega }^{(\slb)}_{k,k'} = 2 \int_0^1 \Cov\left(f_k(Z_{1\xi}), f_{k'}(Z_{2\xi})\right) \diff\xi
	\] 
	Moreover, with $\le_L$ denoting the Loewner-ordering between symmetric matrices, we have
	\begin{align} \label{eq:varl1}
	\bm \Omega^{(\slb)} \le_L \bm \Omega^{(\djb)}.
	\end{align}
\end{theorem} 

Recall that the asymptotic bias is always the same, except under sampling scheme (S2) and for sliding blocks.
It is worthwhile to mention that the theorem may be extended to arbitrary $k \ge 3$; in that case, $f_k$ is given by $f_k=f_{k,1}+f_{k,2}$ with $f_{k,1}, f_{k,2}$ from \eqref{eq:fkdec}. Further, 
note that the integral in $\bm \Omega^{\scs(\slb)}_{k,k'}$ corresponds to the contribution introduced by the strong serial dependence between the sliding block maxima.
More explicit expressions for the asymptotic covariances can be found in Appendix~\ref{sec:var} in the supplementary material, see Lemma~\ref{lem:formelcov}. The graphs of the ratio of the variance curves $\gamma \mapsto \bm \Omega_{k,k}^{\scs (\djb)} / \bm \Omega_{k,k}^{\scs (\slb)}$ are depicted in Figure~\ref{fig:asym_var_ratio},  for  $\gamma \in(-1,1/2)$. As can be seen, the sliding blocks variances are universally smaller than the disjoint blocks counterparts, with a substantial improvement for $k=0$ and negative  $\gamma$.

\begin{figure}
	\centering
	\makebox{\includegraphics[width = 0.85\textwidth]{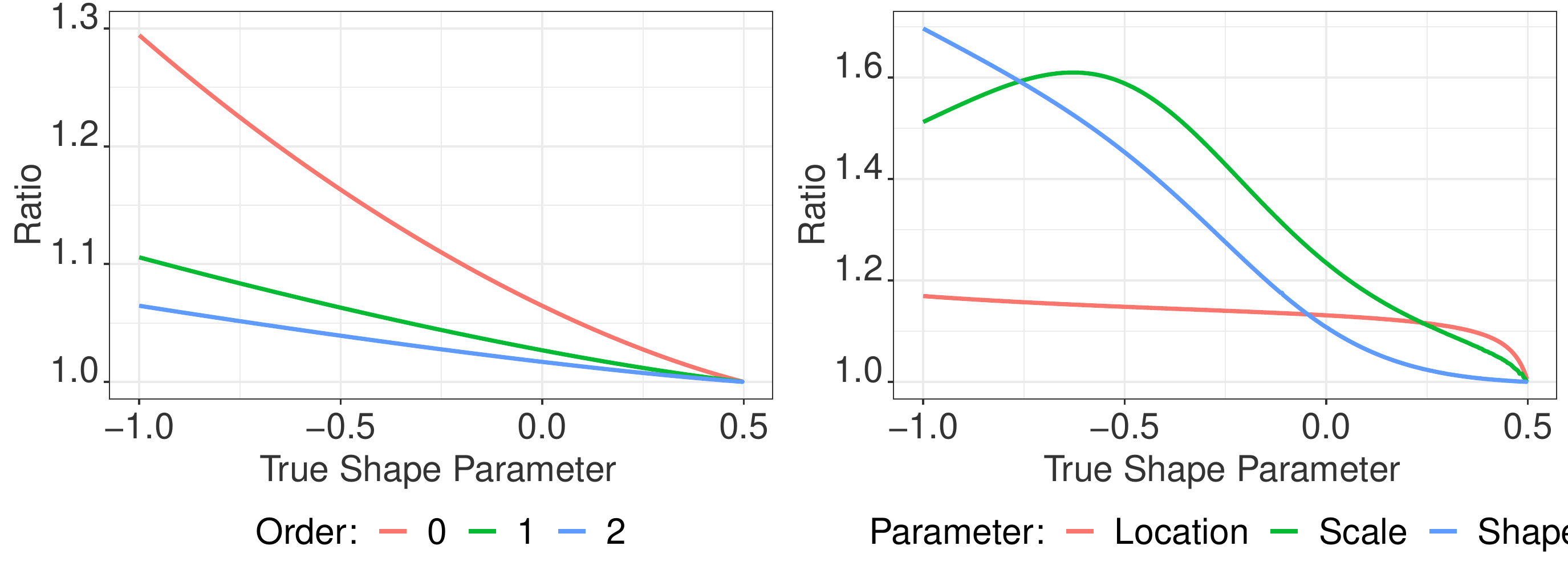}}
	\caption{\label{fig:asym_var_ratio}	Left: Graph of $ \gamma \mapsto \bm \Omega_{k,k}^{\scs (\djb)} / \bm \Omega_{k,k}^{\scs (\slb)}$ for $k\in\{0,1,2\}$ and with $\bm \Omega^{\scs(\mb)}$ as in Theorem~\ref{theo:pwm1}. Right: Graph of $ \gamma \mapsto \bm \Sigma_{\ell,\ell}^{\scs (\djb)} / \bm \Sigma_{\ell,\ell}^{\scs (\slb)}$ for $\ell=1$ (shape), $\ell=2$ (scale) and $\ell=3$ (location) with $\bm \Sigma^{\scs (\mb)} = (\bm \Sigma_{\ell,\ell'}^{\scs (\mb)})_{\ell, \ell'=1,2,3}$ from Corollary \ref{cor:pwm}.}
\end{figure}

Asymptotic normality of the PWM estimator for $(b_r, a_r, \gamma)$ essentially follows from the above theorem and the delta method. Let 
\begin{align} \label{eq:phi}
\renewcommand{\arraystretch}{1.3}
\phi: \Dc_{\phi} \to \R^3, \quad  \bm \beta:= (\beta_0,\beta_1,\beta_2)' \mapsto 
\left( \begin{array}{c} \phi_1(\bm \beta) \\  \phi_2(\bm \beta) \\ \phi_3(\bm \beta) \end{array}\right) = 
\left( \begin{array}{l}
g_1^{-1} \left(\frac{3\beta_2 -\beta_0}{2\beta_1-\beta_0} \right)  \\
g_2(\phi_1(\bm \beta))(2\beta_1-\beta_0)  \\ 
\beta_0 +  \phi_2(\bm \beta) g_3(\phi_1(\bm\beta))
\end{array}\right),
\end{align}
where $\Dc_{\phi}=\{ \bm \beta\in \R^3: 2 \beta_1 - \beta_0>0, 3 \beta_2 - 2 \beta_1>0, - \beta_0+4 \beta_1 - 3 \beta_2>0\}$. 
Recall that $\theta_r=(b_r, a_r, \gamma)'=\phi(\beta_{\theta_r,0}, \beta_{\theta_r,1},  \beta_{\theta_r,2})$ for $\gamma<1$ by \eqref{eq:eqs}, and that $\hat \theta_{r}^{\scs(\mb)} = \phi(\hat{\bm \beta}_r^{\scs (\mb)})$, where $\hat{\bm \beta}_r^{\scs(\mb)}=(\hat \beta_{r,0}^{\scs (\mb)}, \hat \beta_{r,1}^{\scs (\mb)}, \hat \beta_{r,2}^{\scs (\mb)})'$ and $\mb\in\{\djb, \slb\}$. Further, as shown in Proposition~2.1 in \cite{KojNav17}, we necessarily have $(\beta_{\theta_r,0}, \beta_{\theta_r,1},  \beta_{\theta_r,2})' \in \Dc_\phi$ for $\gamma<1$.  Theorem~\ref{theo:pwm1} then implies $\lim_{n\to\infty} \Prob(\hat{\bm \beta}_r^{\scs(\mb)} \in \Dc_{\phi})=1$ after a simple calculation.

\begin{corollary} \label{cor:pwm}
	Write $\hat{\theta}_{r}^{\scs(\mb)} = (\hat b_{r}^{\scs (\mb)}, \hat a_{r}^{\scs (\mb)}, \hat \gamma_{r}^{\scs (\mb)})'$. 
	Under the conditions of Theorem~\ref{theo:pwm1}, we have
	\begin{equation} \label{eq:norm3}
	\renewcommand{\arraystretch}{1.3}
	\sqrt{\frac{n}{r}} \left(\begin{array}{c}  \hat \gamma_{r}^{\scs (\mb)} - \gamma \\ (\hat a_{r}^{\scs (\mb)} -a_r)/{a_r}  \\  (\hat b_{r}^{\scs (\mb)} - b_r)/{a_r}\end{array}  \right)
	\dto
	\Nc_3(\bm C\bm B^{(\mb,S)}, \bm \Sigma^{(\mb)}),
	\end{equation}
	where $\bm{\Sigma}^{(\mb)} = \bm C \bm \Omega ^{(\mb)} \bm C'$ with $\bm C= (D\phi)(\bm \beta_\gamma)$ the Jacobian of $\phi$ evaluated at the true PWMs $\bm \beta_{\gamma}=(\beta_{\gamma,0}, \beta_{\gamma,1}, \beta_{\gamma,2})'$ of $G_\gamma$. Moreover, we have
	\begin{align} \label{eq:varl2}
	\bm \Sigma^{(\slb)} \le_L \bm \Sigma^{(\djb)}.
	\end{align}
\end{corollary}

Precise formulas for the matrix $\bm C$ can be found in Lemma~\ref{lem:JacobiC}.
A careful calculation reveals that we retrieve the asymptotic variance from \cite{FerDeh15} in the disjoint blocks i.i.d.\ case, despite by a completely different proof and under slightly different assumptions on the tail of the c.d.f.\ of $X$. 

\subsection{Application: return level estimation} \label{sec:rl}
A typical quantity of interest in environmental statistics is the return level (RL) of an extreme event. Formally, for a block size~$r$ (often a year or a season) and a target number of (disjoint) blocks $T$, the $(T,r)$-return level of the distribution $F_r$ defined in \eqref{eq:fr} is defined as
\[
\RL(T,r) = F_r^{\leftarrow} (1-1/T) = \inf\{ x \in \R: F_r(x) \ge 1-1/T\}.
\]
Note that it will take on average $T$ independent disjoint blocks of size $r$ until the first such block whose maximum  exceeds $\RL(T,r)$. Now, by Condition~\ref{cond:mda}, we have $F_r \approx G_{(b_r, a_r, \gamma)}$, whence $\RL(T,r) \approx \RL^\circ(T,r)$, where
\[
\RL^\circ(T,r) = G_{(b_r, a_r, \gamma)}^\leftarrow(1-1/T) = a_r \frac{ c_T^{-\gamma}-1}{\gamma} + b_r,
\]
and where $c_T=- \log(1-1/T)$. We therefore obtain the estimators
\[
\widehat \RL{}^{(\mb)}(T,r)= \hat a_r^{(\mb)} \frac{ c_T^{-\hat \gamma^{(\mb)}}-1}{\hat \gamma^{(\mb)}} + \hat b_r^{(\mb)}, \qquad \mb \in \{\djb, \slb\}. 
\]

\begin{corollary}\label{cor:rl}
	Under the conditions of Theorem~\ref{theo:pwm1}, we have
	\[
	\sqrt{n/r} \left( \frac{\widehat \RL{}^{(\mb)}(T,r) - \RL^\circ(T,r)}{a_r} \right)
	\dto 
	\Nc\left(q_T' \bm{C}\bm{B}^{(\mb, \samp)}, q_T' \bm{\Sigma}^{(\mb)}q_T\right),
	\]
	where  $q_T=q_T(\gamma)$ is defined as $q_T(0) = ( \log^2(c_T)/2, - \log(c_T) ,1)'$ and
	\[
	q_T(\gamma)=\left( \frac{1-c_T^{-\gamma}(\gamma \ln(c_T)+1) }{\gamma^2} , \frac{c_T^{-\gamma} -1}{\gamma}, 1\right)', \qquad \gamma \ne 0.
	\]	
\end{corollary}

The asymptotic variance in Corollary~\ref{cor:rl} being an explicit function of $\gamma$, it may easily be estimated by the plug-in principle; we denote the respective estimator  by $\hat \sigma^{\scs 2,(\mb)}_T$.
Corollary~\ref{cor:rl} then allows to construct asymptotic confidence intervals for $\RL(T,r)$. Indeed, assuming that the block size $r$ is chosen sufficiently large to guarantee that $\bm B^{(\mb, \samp)}=0$ and that $\RL^\circ(T,r) = \RL(T,r) + o(\sqrt{r/n} a_r)$, we obtain that
\[
\RL(T,r) \in \Big[ \widehat \RL{}^{(\mb)} (T,r) \mp \hat a_r^{(\mb)} \sqrt{\frac{r}n \hat \sigma^{2,(\mb)}_T} u_{1-\alpha/2} \Big]
\]
with asymptotic probability $\alpha$, where $u_{1-\alpha/2}$ is the $(1-\alpha/2)$-quantile of the standard normal distribution. It follows from the bounds on the asymptotic variances in Corollary~\ref{cor:pwm} that the confidence intervals are asymptotically more narrow for the sliding blocks method; an observation that will be confirmed by the case study in Section~\ref{sec:case}.

\section{Simulation study} \label{sec:sim}

The finite-sample properties of the proposed estimators have been evaluated in a large scale Monte Carlo simulation study. Three target variables have been selected: the shape parameter $\gamma$, and two return levels, $\RL(50,r)$ and $\RL(100,r)$.
The following central aspects have been investigated: 
\begin{enumerate}[label=(\roman*)]
	\item Performance of the disjoint and sliding blocks based PWM estimator when sub-asymptotic versions of sampling schemes (S1) and (S2) from Condition~\ref{cond:obs}   are met 
	\begin{itemize}
	\item for fixed blocksize $r$ (Section \ref{sec:fixr}) 
	\item for fixed samplesize $n$ (Section \ref{sec:fixn})
	\end{itemize}

	\item Performance of the PWM estimator when the seasonal stationarity from Condition \ref{cond:obs} is violated (Section \ref{sec:seasComp}).
	\item Comparison of the PWM estimator to Maximum Likelihood estimators based on sliding blocks (Section \ref{supp:sec:ml} of the supplementary material,  summarized in Section~\ref{sec:simsum}).
\end{enumerate}
The data-generating processes for the models that were used for (i) and (iii) are as follows:
\begin{asparaenum}
	\item[(a)] \textbf{Stationary distribution of $X_t$.} We opted for a model that allows for both positive and negative shape parameters in a continuous way, and hence chose five distributions from the generalized Pareto family, namely GPD$(0, 1, \gamma)$ with shape parameter $\gamma$ in $\{-0.4, -0.2, 0, 0.2, 0.4\}$ with corresponding c.d.f.
	\begin{align*} 
	F_\gamma(x)=
	\begin{cases}
	\left(1-(1+\gamma x)^{-\frac{1}{\gamma} }\right)\ind(x \geq 0), & \gamma > 0, \\
	\left( 1-(1+\gamma x)^{-\frac{1}{\gamma} }\right)	\ind(0 \leq x \leq \abs{\gamma}^{-1}), & \gamma <0, \\
	\left(1-\exp(-x)\right)\ind(x \geq 0 ), & \gamma = 0.
	\end{cases}
	\end{align*} 
	Note that an i.i.d.\ series from $F_\gamma$ satisfies Condition~\ref{cond:mda} with shape parameter $\gamma$ and scaling sequences $a_r=r^\gamma$ and $b_r=(r^\gamma -1)/{\gamma}$, to be interpreted as $\log r$ for $\gamma=0$. 

Experiments involving a different family of distributions (where weak convergence of block maxima to the GEV is slower tha for the GPD) have also been performed; the qualitatively similar results can be found in Section \ref{supp:sec:hw} in the supplementary material.
	\item[(b)] \textbf{Time series model.} Next to the i.i.d. case, we considered quantile transformed versions of the Gaussian AR(1) model (with extremal index $1$), of an AR(1) process with heavy tailed Cauchy(1) innovations and of the Fr\'echet ARMAX(1) model (the latter two having extremal index smaller than 1). Recall that the extremal index is a measure for the tendency of extreme observations to occur in clusters (the smaller $\theta$, the larger that tendency), see Section 10.2.3 in \cite{BeiGoeSegTeu04}  for a gentle introduction.
	
	The transformed Gaussian AR-model is defined as follows: for given AR-parameter $|\phi| < 1$ (we chose $\phi \in \{-0.75, -0.5, -0.25, 0, 0.25, 0.5, 0.75 \}$; note that $\phi=0$ corresponds to the i.i.d.\ case), consider the stationary solution $(Y_t)_t$ of the classical AR(1) recursion
	\begin{align}\label{eq:ARrec}
	Y_t = \phi Y_{t-1} + \epsilon_{t}, \quad \ t\in \Z, \quad (\epsilon_{t})_t \stackrel{\mathrm{i.i.d}}{\sim} \Nc(0,1).
	\end{align}
	The marginal distribution, say $F_Y$, is known to be centred normal with variance $1/(1-\phi^2)$  \citep{BroDav87} and the extremal index of $(Y_t)_t$ is known to be 1 \citep{embrechts1997modelling}. As a consequence,  $X_t = F_\gamma^{\leftarrow} (U_t)$ with $U_t = F_Y(Y_t)$ satisfies Condition~\ref{cond:mda} with shape parameter $\gamma$ and extremal index 1. 
	
	For the Cauchy AR (CAR) model, the Gaussian innovations in \eqref{eq:ARrec} are replaced by i.i.d.\ Cauchy(1)-innovations. Proposition 13.3.2 in \cite{BroDav87} yields the representation 
$
		Y_t = \sum_{j = 1}^{\infty} \phi^j \epsilon_{t-j}.
$
	 For $\phi \in (0,1)$ (we chose $\phi \in \{ 0.25, 0.5, 0.75\}$), Example 8.1.1 d) in \cite{embrechts1997modelling} then implies that the extremal index exists and is given by $\theta = 1- \phi$. Moreover, a simple calculation based on characteristic functions shows that the marginal distribution $F_Y$ of $Y_t$ is Cauchy as well, with scale parameter $1/(1-\phi)$. We may thus transform to uniform margins by letting $U_t = F_Y(Y_t)$ and may generate $ X_t = F_\gamma^{\leftarrow}(U_t)$, which satisfies Condition \ref{cond:mda} with shape parameter $\gamma$ and extremal index $\theta = 1- \phi$. 
	
	The transformed ARMAX-model is defined as follows: for given $b \in [0,1)$ (we chose $b \in \{0.25, 0.5, 0.75\}$), consider the stationary solution $(Y_t)_t$ of the ARMAX(1) recursion
	\[ 
	Y_t := \max( b Y_{t-1}, (1-b ) \epsilon_{t}), \quad  t\in \Z, \quad
	(\epsilon_{t})_t \stackrel{\mathrm{i.i.d}}{\sim} \text{Fr\'echet}(1).
	\]
	The marginal distribution $F_Y$ is known to be Fréchet(1) as well, and the extremal index is equal to $\theta=1-b$ (Section 10 in \citealp{BeiGoeSegTeu04}). As a consequence,  $X_t = F_\gamma^{\leftarrow} (U_t)$ with $U_t = F_Y(Y_t)$ satisfies Condition~\ref{cond:mda} with shape parameter $\gamma$ and extremal index $\theta=1-b$.
\end{asparaenum}  

\subsection{Fixed block length $r$}\label{sec:fixr}
In a first experiment, we considered each combination of the described time series model and the marginal distribution function in a situation where the block size is fixed and the overall sample size is increasing.
We fixed $r=90$, which could be interpreted as the number of daily observations within a three-month season; a common situation encountered in environmental applications. The number of seasons was chosen to vary between 10 and 100, 
yielding overall sample sizes of the underlying time series between 900 and 9000. 
We computed the PWM estimators based on disjoint and sliding block maxima, and the respective estimators for $\RL(50,r)$ and $\RL(100,r)$ from Section~\ref{sec:rl}. The estimators have been evaluated in terms of their relative efficiency based on $N=5000$ simulation repetitions, i.e., we divided the MSE of the disjoint blocks estimator by the MSE of the sliding blocks counterparts.  For the sake of brevity, we only report the results for the transformed AR-model; other results can be found in the supplementary material, Section \ref{supp:sec:fixr}.

Results for the estimation of $\gamma$ are presented in Figure~\ref{fig:sim_releff_ar} (see also Figure~\ref{fig:intro} from the introduction), with remarkably similar results for the two sampling schemes (S1) and (S2). Note that, for i.i.d. observations (i.e., AR(0)), sampling schemes (S1) and (S2) coincide, so that there is only one line for each shape in the corresponding panel. The results reveal that the sliding blocks method is universally better than the disjoint blocks method for non-positive shape parameters, with large improvements for small sample sizes (note that situations of less than 50 seasons are not uncommon in environmental applications; in particular when restricting attention to stationary time periods). On the other hand, for positive shape parameters, the disjoint blocks method may outperform the sliding blocks method for small sample sizes. This effect can mostly be resolved by considering the bias-reduced sliding blocks estimator from Remark~\ref{rem:biasred}, which, however, is computationally costly for situations involving overall sample sizes of up to $n=9000$.
 A discussion of the latter estimator is postponed to Section \ref{supp:sec:slbr} of the supplementary material. Finally, it is worthwhile to mention that the time series model does not have a huge impact on the qualitative results.

\begin{figure}[t!]	
	\centering
	\makebox{\includegraphics[width=0.99\textwidth]{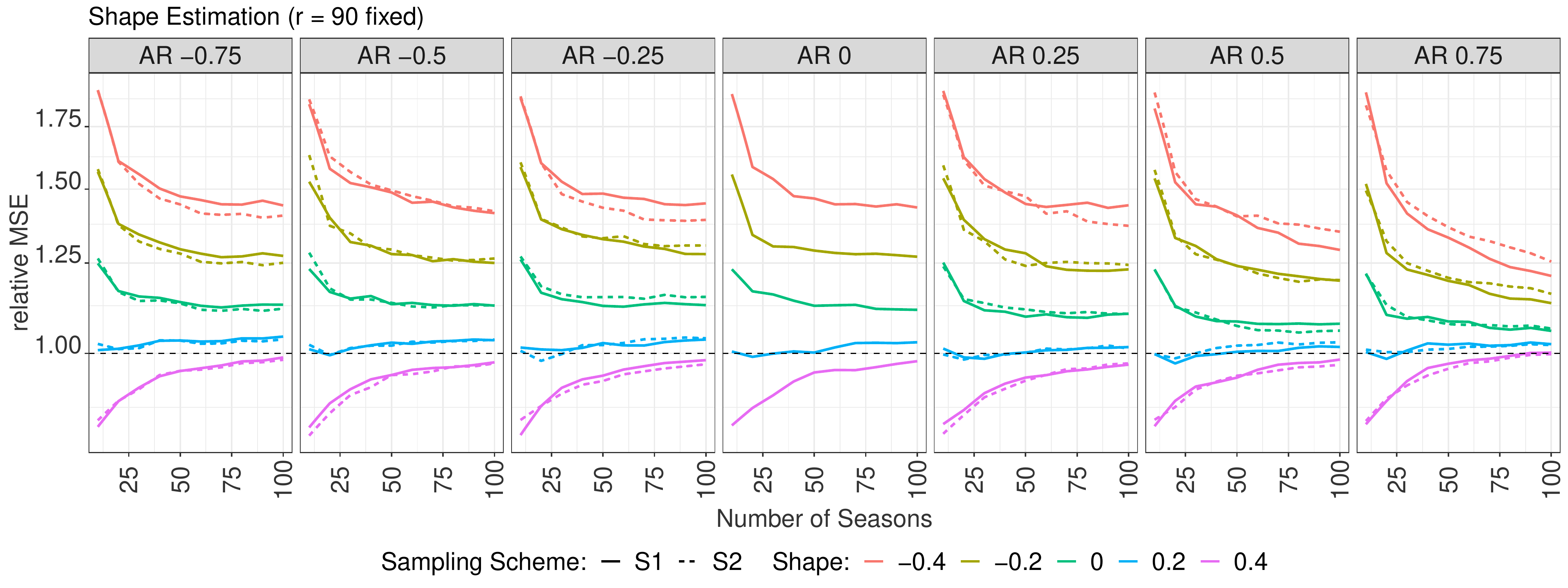}}
	\caption{ \label{fig:sim_releff_ar} Relative Efficiency (MSE of disjoint blocks estimator divided by MSE of sliding blocks estimator)  in a transformed AR(1) model with GPD-margins  under sampling schemes (S1)  and (S2), for fixed block size $r=90$.
	}
\end{figure}

We next consider the estimation of return levels. For the evaluation of the respective estimators,   (`true') population values for the return levels are needed. Since these are not known explicitly, they have been obtained by a preliminary simulation: after simulating $10^6$ independent blocks of length $r$, we calculated the empirical $(1-1/T)$-quantile of the obtained sample to obtain an accurate approximation for $\RL(T,r)$. The respective values for  block size $r=90$ can be found in Table~\ref{tab:emp_rl}; note that the little variation within columns may be explained by the fact that the extremal index of the AR-model is 1 irrespective of the AR-parameter.
The results from the simulation experiment are presented in a similar way as for the shape estimation and can be found in Figure~\ref{fig:sim_rl}. For the sake of a clear presentation, we only consider sampling scheme (S1); the results for sampling scheme (S2) are very similar and can be found in the supplementary material. Overall, the findings are quite similar to those for the estimation of $\gamma$. Compared to the latter target variable, slight advantages  for the sliding blocks method are also visible for $\gamma=0.2$, while we still observe a disadvantage for $\gamma=0.4$. Finally, it is worthwhile to mention that the relative MSE is  increasing in $T$ for all considered situations.

\begin{table}
	\caption{\label{tab:emp_rl}Population return levels $\RL(T, 90)$ for $T = 50$ ($T=100$)} 
	\centering%
	\begin{tabular}{r ccccc}
		\hline
		AR & $\gamma=-0.4$ & $\gamma=-0.2$ & $\gamma=0$ & $\gamma=0.2$ & $\gamma=0.4$ \\ \hline
		-0.75 & 2.41 (2.43) & 4.06 (4.18) & 8.36 (9.05) & 21.40 (25.50) & 67.60 (91.20)  \\
		-0.50 & 2.41 (2.43) & 4.07 (4.19) & 8.39 (9.08) & 21.83 (25.76) & 69.58 (92.83) \\ 
		-0.25 & 2.41 (2.43) & 4.07 (4.19) & 8.39 (9.09) & 21.85 (25.89) & 69.36 (92.53) \\ 
		0 & 2.41 (2.43) & 4.07 (4.19) & 8.40 (9.09) & 21.85 (25.84) & 69.18 (92.18) \\ 
		0.25 & 2.41 (2.43) & 4.07 (4.19) & 8.40 (9.09) & 21.84 (25.88) & 69.45 (92.87) \\ 
		0.50 & 2.41 (2.43) & 4.06 (4.18) & 8.37 (9.06) & 21.74 (25.77) & 68.62 (92.01) \\ 
		0.75 & 2.41 (2.43) & 4.03 (4.16) & 8.20 (8.94) & 20.70 (24.80) & 63.30 (86.00) \\
		\hline
	\end{tabular}
\end{table}

\begin{figure}[t]	
	\centering
	\makebox{	\includegraphics[width=0.95\textwidth]{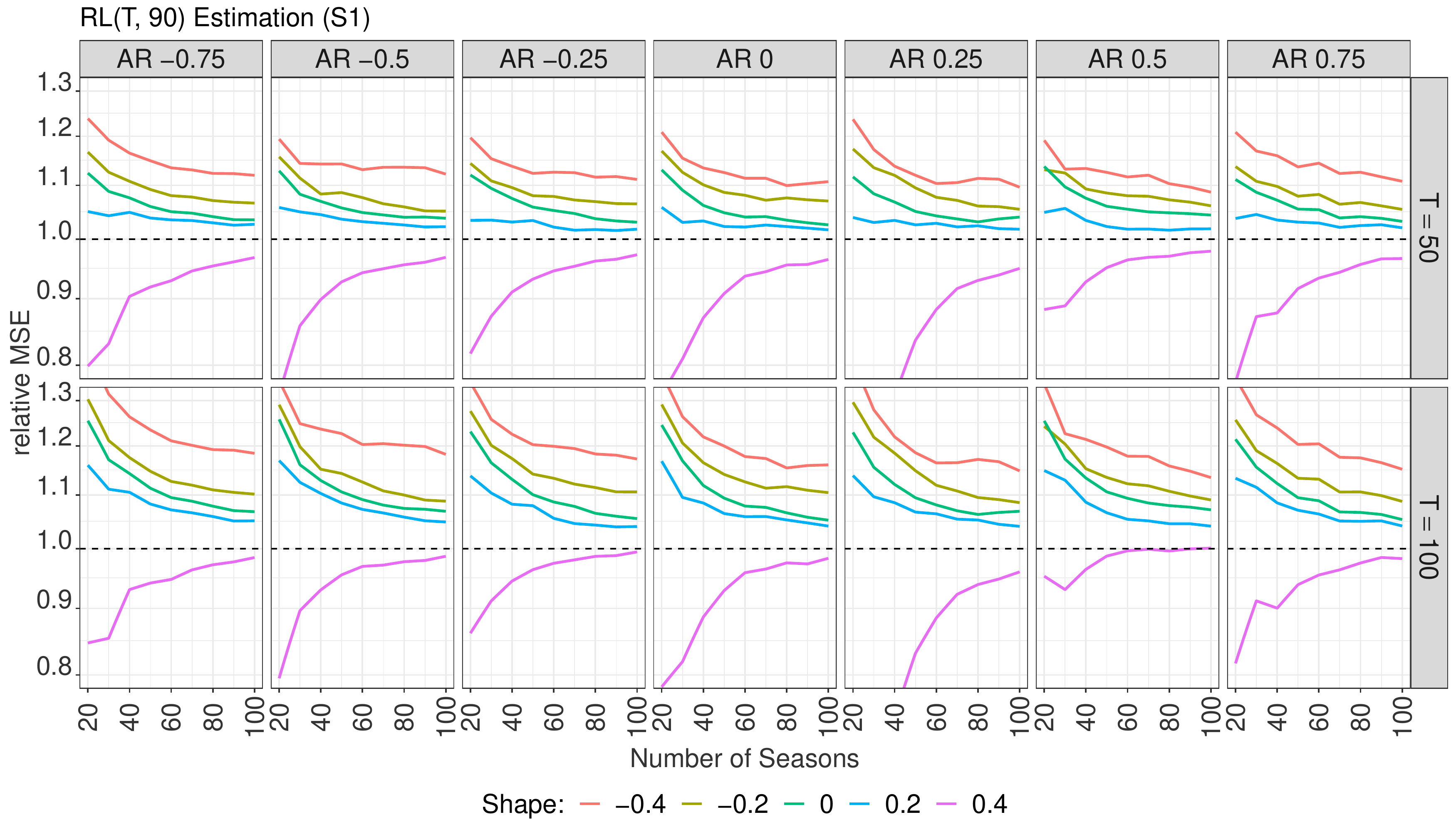}} \vspace{-.4cm}
	\caption{ \label{fig:sim_rl} 
		Relative Efficiency (MSE of disjoint blocks estimator divided by MSE of sliding blocks estimator)  in a transformed AR(1) model with GPD-margins  under sampling schemes (S1)  for fixed block size $r=90$.}
\end{figure}

\subsection{Fixed sample size $n$} \label{sec:fixn}
In a second experiment, we considered each combination of the described time series model and the marginal distribution function for fixed sample length $n=1000$ and sampling scheme (S1). The setting aims at evaluating the common bias-variance tradeoff in extreme value statistics, which becomes visible when treating the block length as a hyperparameter to be chosen by the statistician with the ultimate goal of maximizing the estimation accuracy (which is comparable to the choice of the number of upper order statistics in the peaks-over-threshold approach). Note that treating the blocksize as a hyperparameter is only valid for sampling scheme (S1) (it is given when considering sampling scheme (S2)) and for estimating the shape parameter (as return levels depend on the blocksize). For the experiment, the block length  
has been chosen as      
\[
r \in \{4,5,6,7,8,9,10,12,14,16,18,20,25,30, 40\},
\]
yielding between 25 and 250 disjoint blocks.  
All estimators (disjoint, sliding, and sliding bias reduced) have been evaluated in terms of their empirical MSE, variance and squared bias based on $N=1000$ simulation repetitions. For the sake of brevity, we only report the results for the transformed AR-model and for the comparison of the plain disjoint and sliding blocks estimator; the respective results for the CAR and ARMAX (which are qualitatively similar) and for the bias-reduced sliding blocks estimator can be found in the supplementary material (Sections \ref{supp:sec:fixn} and \ref{supp:sec:slbr}).

\begin{figure}[t!]	
	\centering
	\makebox{	\includegraphics[width=0.9\textwidth]{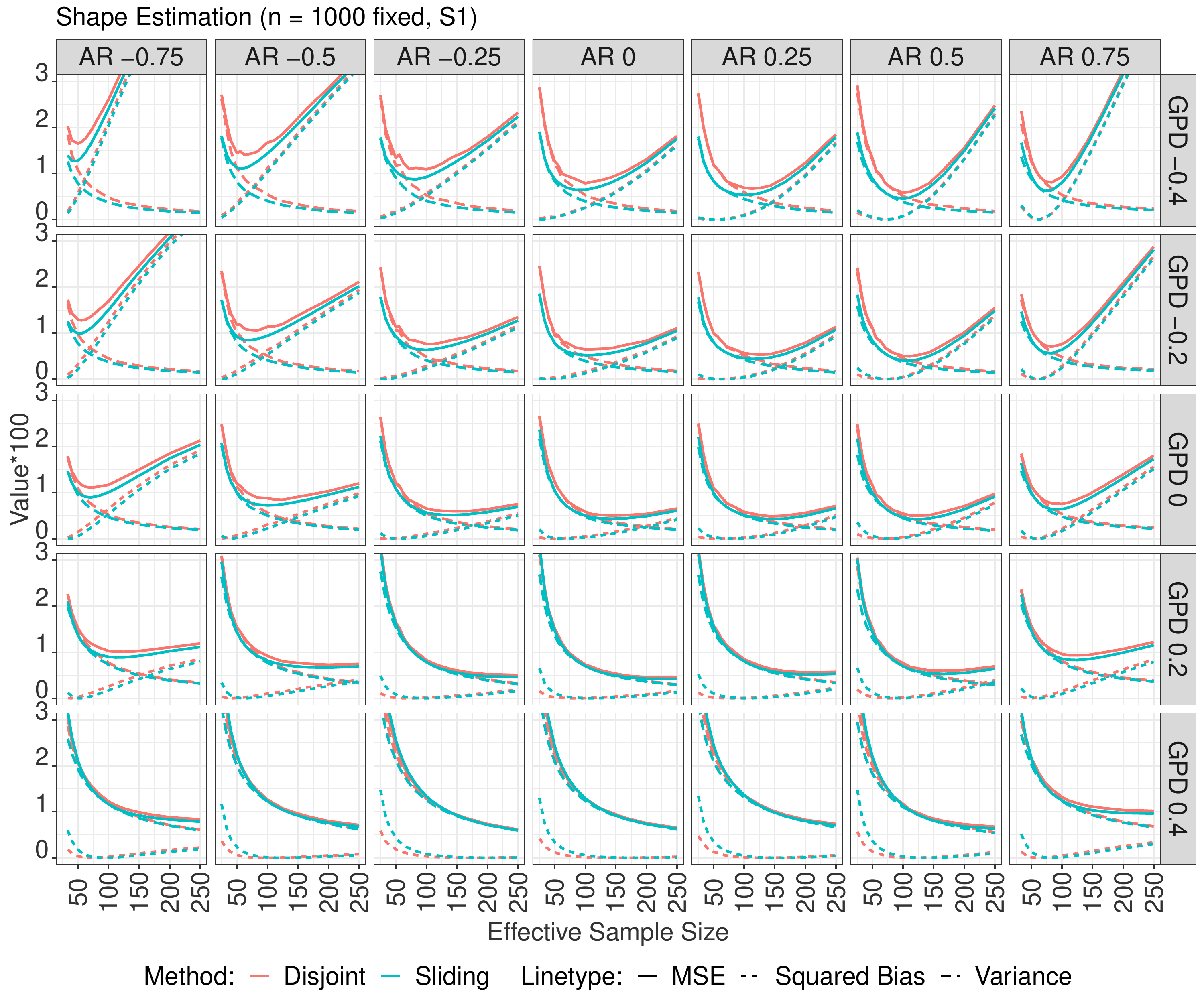}} \vspace{-.3cm}
	\caption{\label{fig:sim_ar_all} MSE, squared bias and variance for the estimation of the shape parameter $\gamma$ in a transformed AR(1) model with GPD-margins under sampling scheme (S1) for fixed sample size $n=1000$.}
\end{figure}

The results are summarized in  Figure~\ref{fig:sim_ar_all}.
 The x-axis corresponds to the effective sample size, defined as the number of disjoint blocks $\lfloor n/r \rfloor$. The general shape of the curves is mostly (with the exception of $\gamma=0.4$) as follows: we observe a decreasing variance curve that is universally smaller for the sliding blocks method (as expected from the theoretic results) and an (eventually) increasing bias curve that is mostly comparable between the two methods. As a result, the MSE curve is mostly u-shaped, representing the typical bias-variance tradeoff. The improvement of the sliding blocks method over the disjoint blocks method is largest for negative shape parameters, while no significant improvement is visible for positive shape parameters.  The time series model does not have a significant effect on the qualitative performance. For small effective sample sizes (i.e., large block sizes), we observe a significantly higher bias for the sliding blocks method, which may be explained by the dependency bias discussed in Remark~\ref{rem:biasred}; see also the results in the supplementary material for further discussions.

\subsection{Deviation of the piecewise stationary setting} \label{sec:seasComp}
In a third experiment, we investigate the performance within a situation that deviates from the piecewise stationary setting postulated in Condition~\ref{cond:obs}.
Since the previous simulation results suggest that the efficiency gain of using sliding blocks is largest for non-positive shape parameters, we aim for a model describing temperature extremes, since shape parameters of seasonal maxima are well-known to be negative for this kind of data. 
We may then rely on \cite{stein2017}, where the asymptotic distribution of block maxima was investigated in a framework where the finite upper bound of the `daily observations' was allowed to depend smoothly on (rescaled) inner-seasonal time. In the case of serially independent observations, the limiting distribution was found to be GEV again, despite with an unexpected shape parameter; see Theorem 1. Extensions to serial dependence were not worked out explicitly, but it was conjectured that similar phenomena arise.

We employ the marginal model described in the third paragraph on page 5 in  \cite{stein2017}: for the $i$th day of the year (restricting attention to the first 90 days of the summer season corresponds to $i \in \{152, \dots, 241\}$), we denote by $F_i$ the cdf of the 
\[
\text{GPD}(u_i - (7\cdot10^7)^\frac{1}{5} , ((7\cdot10^7)^\frac{1}{5} )/5, -0.2 )
\]
distribution, where $u_i = 111 - (i - 200)^2/400$.		

We then apply the quantile transformation technique again: starting from one of the serial dependence structures of interest, we transform the marginals to the time dependent GPD $F_i$. We restrict attention to sampling scheme (S2), since this seems to be the natural choice here. 
Last but not least, note that the above model is in Fahrenheit, so we transform the simulated data to ${}^\circ C$ by multiplying by 5/9 after subtracting 32.

 We restrict attention to return level estimation (note that the true limiting shape parameter is only known for the i.i.d.\ case: it is $-2/11$ by Theorem 1 in \citealp{stein2017}). Since `true' return levels are not known explicitly either, they are approximated based on a preliminary Monte Carlo simulation involving $N=10^6$ block maxima of size $r=90$, from which the empirical $99\%$-quantile (i.e., the $100$-season return level) is determined.

The results are compared to a situation without innerseasonal non-stationarities. To obtain observations of the same magnitude, we generate data with margins corresponding to GPD$(72.21,  ((7\cdot10^7)^\frac{1}{5} )/5, -0.2 )$, since $\frac{1}{90}\sum_{i=1}^{90}(u_i - (7\cdot 10^7)^{1/5}) = 72.21 $ is the mean of the location parameters of the non-stationary counterparts.
MSEs and relative MSEs observed in a selection of models with different dependence structures are shown in Figure~\ref{fig:seascomp:relMSE_RL}. We observe that the innerseasonal non-stationarity does not have a significant influence on the estimation performance, and that the advantage of sliding blocks over disjoint blocks remains. The sliding blocks method may hence be regarded as robust to certain deviations from the piecewise stationary setup. 

\begin{figure}[tbh!]
	\centering	
	\includegraphics[width=0.7\textwidth]{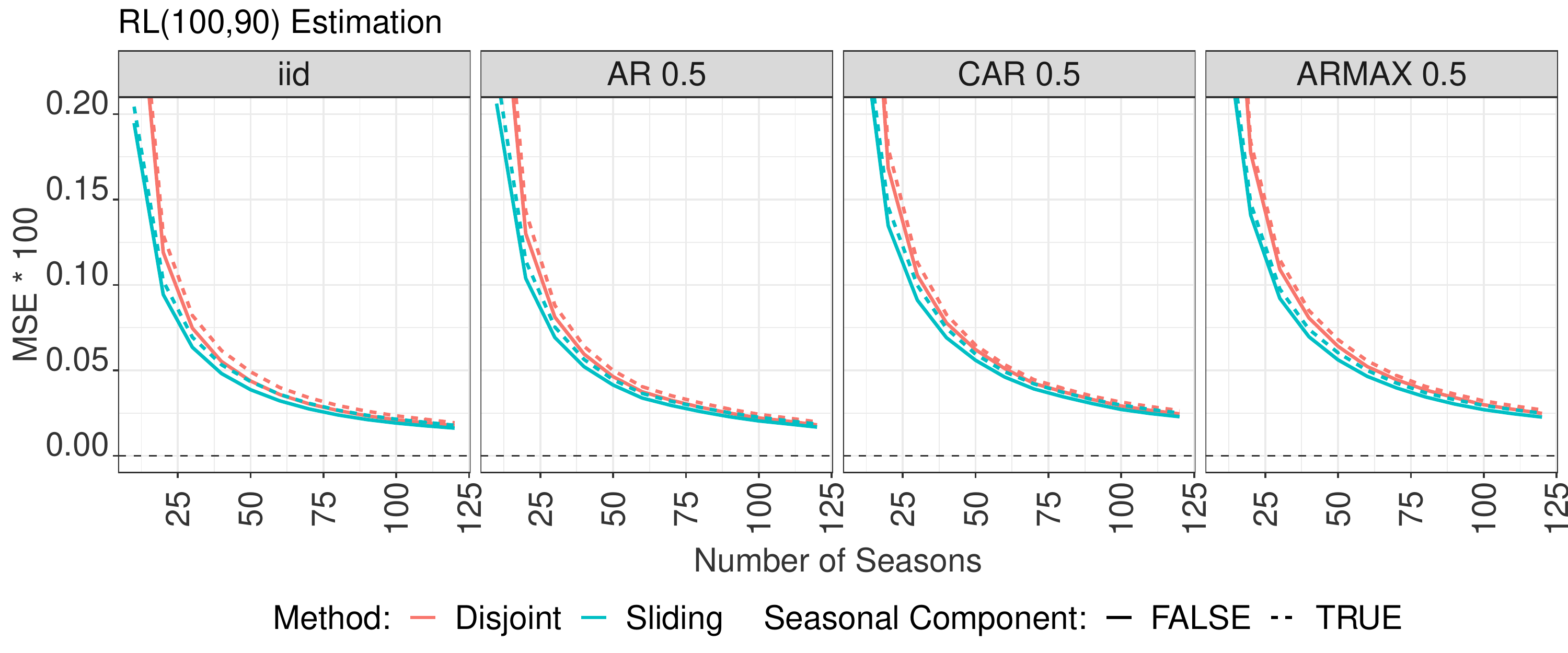}
	\includegraphics[width=0.7\textwidth]{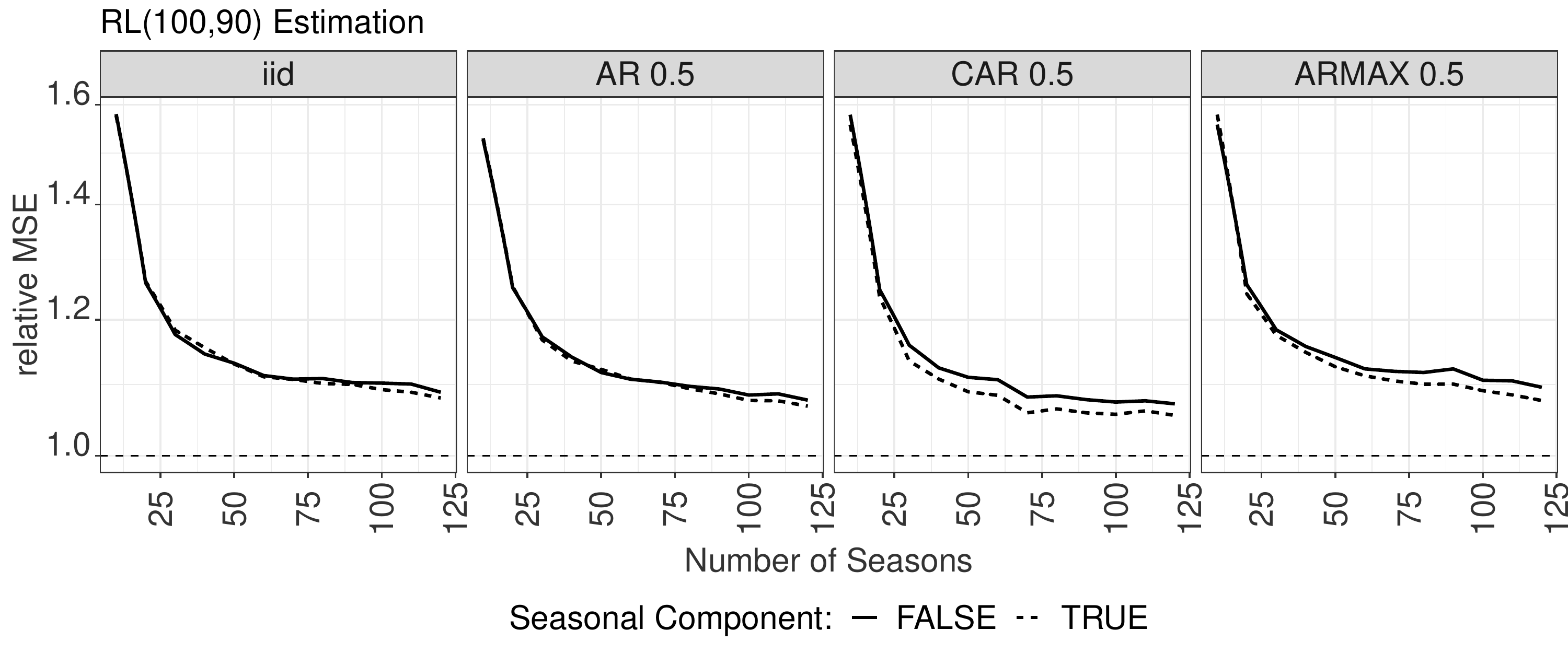}

	\caption{ \label{fig:seascomp:relMSE_RL} MSE (top) and relative MSE (bottom; MSE of disjoint block estimation divided by MSE of sliding block estimation) of $\RL(100,90)$ estimation without (solid line) and with (dashed line) innerseasonal non-stationarity as a function of the observed number of seasons. 
}
\end{figure}

\subsection{Comparison with Maximum Likelihood Estimation} \label{sec:simsum} The sliding blocks PWM estimator has been compared to its counterpart based on (pseudo) maximum likelihood estimation, which is defined by maximizing the  GEV likelihood function that arises from treating all sliding blocks as independent. For the sake of brevity, the results are illustrated in the supplementary material only. They can be summarized as follows: the PWM estimator has a tendency to be superior for small sample sizes while the maximum likelihood estimator is superior for large sample sizes; to the best of our knowledge this is a  usual view of the two estimators among applied statisticians. For shape estimation, smaller shapes yield better results for the PWM estimator, while for return level estimation, the picture is almost reversed.

\section{Case study} \label{sec:case}

Estimating return levels of the distribution of annual or seasonal maxima (of some meteorologic variable of interest) based on GEV-models constitutes one of the cornerstones of extreme weather event attribution studies \citep{Sto16}. Since the sliding blocks PWM estimator has been seen to provide the largest improvement over its disjoint counterpart for negative shape parameters, our case study concentrates on maximal air temperature data, for which shapes are usually within the range  $-0.4$ to $-0.2$. The data set to be analyzed consists of daily observations throughout the summer months (June, July, August) at four selected weather stations  in Germany (Aachen, Eseen-Bredeney, Frankfurt/Main, Kahler Asten), provided by the DWD (Deutsche Wetterdienst). The respective sample lengths (years of observations) are  $66, 72, 71, 65$, and the block size is equal to $r=92$. The target variable is the $100$-year return level, i.e.,  $\RL(100, 92)$, at each station.

Maximal temperature data are known to be non-stationary due to climate change (the average global surface temperature has roughly increased by about 1 degree celsius compared to pre-industrial times), whence a realistic model for maximal temperature must involve non-stationarities as well. Subsequently, let $T_1, \dots, T_{92}, T_{93}, \dots, T_{184}, \dots$ denote the concatenated sequence of daily temperatures throughout the summer months at a specific station, where the first observation corresponds to June 1 in a certain year. 
A standard GEV-model that is commonly applied within the context of extreme event attribution studies for maximal temperature data $M_t = \max(T_{92(t-1)+1}, \dots, T_{92t})$ in season $t$ consists of imposing a simple linear model for the location parameter in terms of the 4-year smoothed global mean surface temperature (sGMST) anomaly, see, e.g., \cite{philip2020protocol} and the references therein. More precisely, 
\begin{align} \label{eq:mmt}
M_t= cx_t + Z_t \sim  \mathrm{GEV}(b + c x_t, a, \gamma)
\end{align}
where $(x_t)_t$ denotes the yearly sequence of  sGMST (see Figure~\ref{fig:gmst}), where $b, c, \gamma \in \R$ and $a>0$ are the free parameters of the model and where $Z_t \sim \mathrm{GEV}(b, a, \gamma)$ is stationary.  

After subtracting the global trend, it appears heuristically reasonable to assume that the (unobservable) detrended time series defined by concatenating the blocks
\begin{align}\label{eq:detrsample}
(Y_1^{(t)}, \dots, Y_{92}^{(t)}) =  (T_{92(t-1)+1} -c x_t, \dots, T_{92t}-cx_t)
\end{align}
consists of independent and identically distributed blocks that are close to being stationary, with possibly some small deterministic seasonal component as investigated in Section \ref{sec:seasComp}. Recall that the latter inner-seasonal non-stationarity was found to have no big impact on estimation performance.
 An observable counterpart of \eqref{eq:detrsample} may be obtained by estimating the slope parameter $c$ in model \eqref{eq:mmt}, for which we employ the widely used and robust method from Sen \citep{sen1968estimates}. The respective parameter estimates $\hat c=\hat c(M_1, \dots, M_m)$ for the four stations of interest 
are stated in the third column of Table~\ref{tab:slss}. The resulting sample
	\begin{align}\label{eq:detrsample2}
(\hat T_1, \dots, \hat T_{92}, \hat T_{93}, \dots ) = (T_{1} - \hat c x_1, \dots, T_{92}- \hat cx_1, T_{93} - \hat c x_2, \ldots )
\end{align}
	will be referred to as the sample of detrended daily observations.

\begin{figure}
	\centering
	\makebox{	\includegraphics[width = 0.62\textwidth]{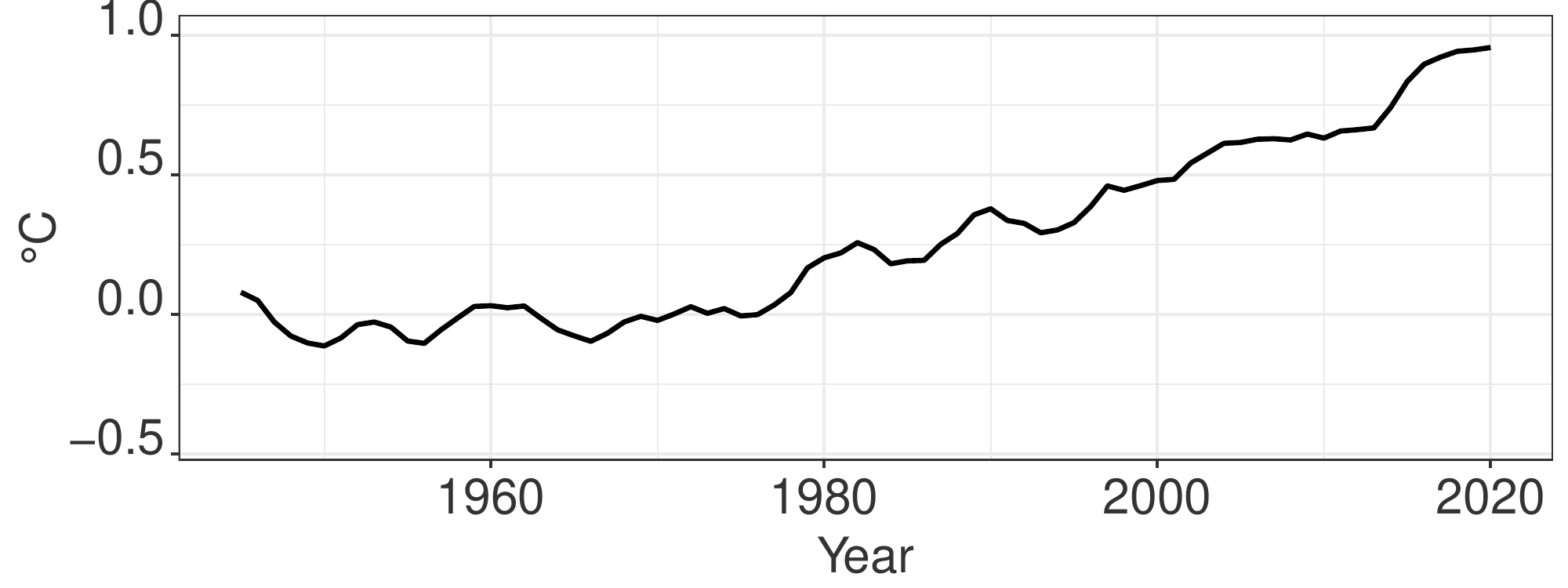}} \vspace{-.3cm}
	\caption{ \label{fig:gmst}  4-year smoothed global mean surface temperature (sGMST) anomaly, with reference value being the average of GMST from 1951-1980.  
	}
	\vspace{-.3cm}
\end{figure}

For illustrative purposes, we proceed the analysis by ignoring any potential estimation error within $\hat c$ (a more rigorous treatment can be found in the next paragraph), and implicitly assume that the sample defined in \eqref{eq:detrsample2} meets the assumption of sampling scheme (S2). 
Note that the respective block maxima

	\begin{align} \label{eq:zst}
	 Z_{t}^{(\djb)} = \max(\hat T_{92(t-1)+1}, \dots, \hat T_{92t}), \qquad
	Z_{92(t-1)+j}^{(\slb)}  = \max(\hat T_{92(t-1)+j}, \dots, \hat T_{92t + j})	
	\end{align}

	satisfiy $Z_{92(t-1)+1}^{(\slb)} = Z_t^{(\djb)} = M_{t}^{(\djb)} - \hat c  x_t $.

	In view of the fact that the (detrended daily) observations from the first and last disjoint block have a reduced chance of appearing multiple times within the sliding blocks sample (for instance, if the sample maximum is the very last observation, it only appears once in the sliding blocks sample, while it would appear $r$ times if it was observed in the second to last season), we chose to tweak the underlying daily sample by attaching the first block to the last one (which is akin to the circular block bootstrap in time series analysis). The resulting sliding blocks sample has then a sample size of exactly $92$ times the number of seasons.

	The disjoint and sliding block maxima can then be fitted to the GEV distribution based on the PWM methods. Estimated parameters are collected in Table~\ref{tab:slss}, and a graphical check of the fit of the resulting distributions can be found in the supplementary material, Section \ref{supp:sec:cs}.

	 Respective estimates for the 100-season return level can be obtained as described in Section~\ref{sec:rl}, including (asymptotic) confidence bounds. The results are summarized in Table~\ref{tab:app1} (the results in brackets will be explained below and can be skipped for the moment). As was to be  expected from both the theoretical results and the simulation study, the confidence intervals based on sliding blocks method are always smaller than their disjoint counterparts, with a substantial margin between 0.23 and 0.78.

\begin{table}
	\caption{\label{tab:app1}Theoretical 95\% confidence intervals for the 100-year RL of the series of detrended summer maxima. Bootstrapped confidence interval bounds (10\;000 repetitions) are shown in round brackets. }
	\centering
\begin{tabular}{llllll}
	\hline
	Station & Method & RL  & CI lower Bound & CI upper Bound & CI Width \\ 
	\hline
	Aachen & Disjoint & 36.66 & 35.10 (35.08) & 38.22 (38.24) & 3.12 (3.16) \\ 
	& Sliding & 36.67 & 35.22 (35.24) & 38.11 (38.09) & 2.88 (2.84) \\ 
	Essen & Disjoint & 35.43 & 34.02 (34.00) & 36.84 (36.85) & 2.82 (2.85) \\ 
	& Sliding & 35.24 & 34.05 (34.07) & 36.43 (36.41) & 2.38 (2.34) \\ 
	Frankfurt & Disjoint & 38.18 & 36.51 (36.47) & 39.85 (39.88) & 3.34 (3.41) \\ 
	& Sliding & 37.80 & 36.52 (36.54) & 39.09 (39.07) & 2.56 (2.52) \\ 
	Kahler Asten & Disjoint & 30.84 & 29.43 (29.42) & 32.24 (32.25) & 2.81 (2.83) \\ 
	& Sliding & 30.75 & 29.46 (29.48) & 32.04 (32.03) & 2.58 (2.55) \\ 
	\hline
\end{tabular}
\end{table}

Note that point estimates for the return level in the climate of season $t$, say $\RL_t(100, 92)$, may be obtained by simply  adding $\hat c  x_t$ to the values in the third column of Table~\ref{tab:app1}.  
However, simply adding $\hat c x_t$ to the confidence bounds in Table~\ref{tab:app1} does not provide valid confidence sets for $\RL_t(100,92)$, as the estimation error of $\hat c$ has not been captured. The latter may be captured by suitable bootstrap devices, as by the following parametric bootstrap scheme.

Given estimates $(\hat c, \hat b, \hat a, \hat \gamma)$ (where the last three components may either be based on disjoint or sliding block maxima samples), we may generate, for each season $t$, an i.i.d.\ sample $T_{92(t-1)+1}^*, \dots, T_{92t}^*$ of size $r=92$ from the GEV distribution with parameter
\begin{align*}
\tilde b_t = \hat{b} + \hat c x_t   - \frac{\hat a (92^{\hat\gamma} -1)}{92^{\hat\gamma}\hat\gamma}, \quad
\tilde{a} = \frac{\hat a}{92^{\hat\gamma}}, \quad 
\tilde{\gamma} = \hat\gamma.
\end{align*}
A simple calculation shows that the $t$th disjoint block maximum from the bootstrap sample is GEV-distributed with parameter $(\hat b+ \hat cx_t, \hat a, \hat \gamma)$.
The fact that $T_{92(t-1)+1}^*, \dots, T_{92t}^*$ may be simulated serially independent can be explained by the fact that the asymptotic distribution of the PWM estimator does not depend on the serial dependence of the underlying time series (except through the parameter sequences $b_r$ and $a_r$ and~$\gamma$, see Corollary~\ref{cor:pwm}; and under the assumption that the block length is sufficiently large to guarantee that the bias is negligible).
Now we apply the same procedure as for the original observations $T_1, T_2, \dots$: first, we build disjoint block maxima and estimate the trend $\hat{c}^*$, then we use this estimate to detrend $T_1^*, T_2^*, \dots$ as in \eqref{eq:detrsample2}, and we finally caluculate the respective disjoint and sliding block maxima as in \eqref{eq:zst}, based on which we ultimately obtain bootstrap estimates $(\hat b^*, \hat a^*, \hat \gamma^*)$ and  $\widehat {\RL_t}^*(100, 92)$.
Repeating the bootstrap procedure $B=10000$ times, we may obtain estimates of the standard error of $(\hat c, \hat b, \hat a, \hat \gamma)$ by calculating the empirical standard deviation of the sample of bootstrap estimates. Likewise, we may obtain bootstrap confidence intervals for any parameter of interest based on the percentile method \citep{DavHin97}. To obtain symmetric 95\% confidence intervals with respect to the estimated return level, we rather solve $ \hat F^* (\widehat{\RL} + \epsilon) -  \hat F^* (\widehat{\RL} - \epsilon) =  0.95 $ for $\epsilon$, where $\hat F^*$ is the empirical distribution function of the bootstrap estimates of return levels, and use $(\widehat{\RL} - \epsilon, \widehat{\RL} + \epsilon)$ as a confidence interval.

The bootstrap scheme has been applied to each station, both for the disjoint and the sliding blocks method. The results are summarized in Table~\ref{tab:slss} (standard deviation of the estimation of $c,b,a$ and $\gamma$) and in Figure~\ref{fig:RL_casestudy} (pointwise confidence intervals for the estimation of $\RL_t(100, 92)$). Remarkably, at each station, the sliding blocks estimator yields slightly smaller estimates for the shape parameter and slightly larger estimates for the scale parameter. The resulting estimates for the 100-year return level are mostly similar, with visible differences only for station `Frankfurt'. By definition, the slope estimates are the same at each station, which explains the fact that the difference between the sliding and disjoint blocks curves in Figure~\ref{fig:RL_casestudy} is constant. In all cases, the confidence bands are smaller for the sliding blocks version, as expected by the large-sample theory and the simulation experiments.

As promised above, we finally explain the values in brackets in Table~\ref{tab:app1}, and return to the  data sets from \eqref{eq:zst} which were considered to be arising from an underlying stationary time series. The values in brackets correspond to bootstrap confidence bounds based on a parametric bootstrap adapted to this simple stationary situation. More precisely, the bootstrap scheme is carried out as before, but with setting $\hat c=0$ when generating the bootstrap samples, and therefore also omitting the detrending step. As can be seen from the results in Table~\ref{tab:app1}, the bootstrap confidence bounds are very similar to the bounds obtained by the normal approximation and estimation of the theoretical asymptotic variance (see Section~\ref{sec:rl}). These findings indicate that the bootstrap scheme is working well, and support its application to the non-stationary situation described above, where the normal approximation cannot be applied without major additional calculations regarding the propagation of uncertainty due to the initial estimation of the slope parameter.

\begin{figure}[thb!]
	\centering
	\makebox{\includegraphics[width=0.99\textwidth]{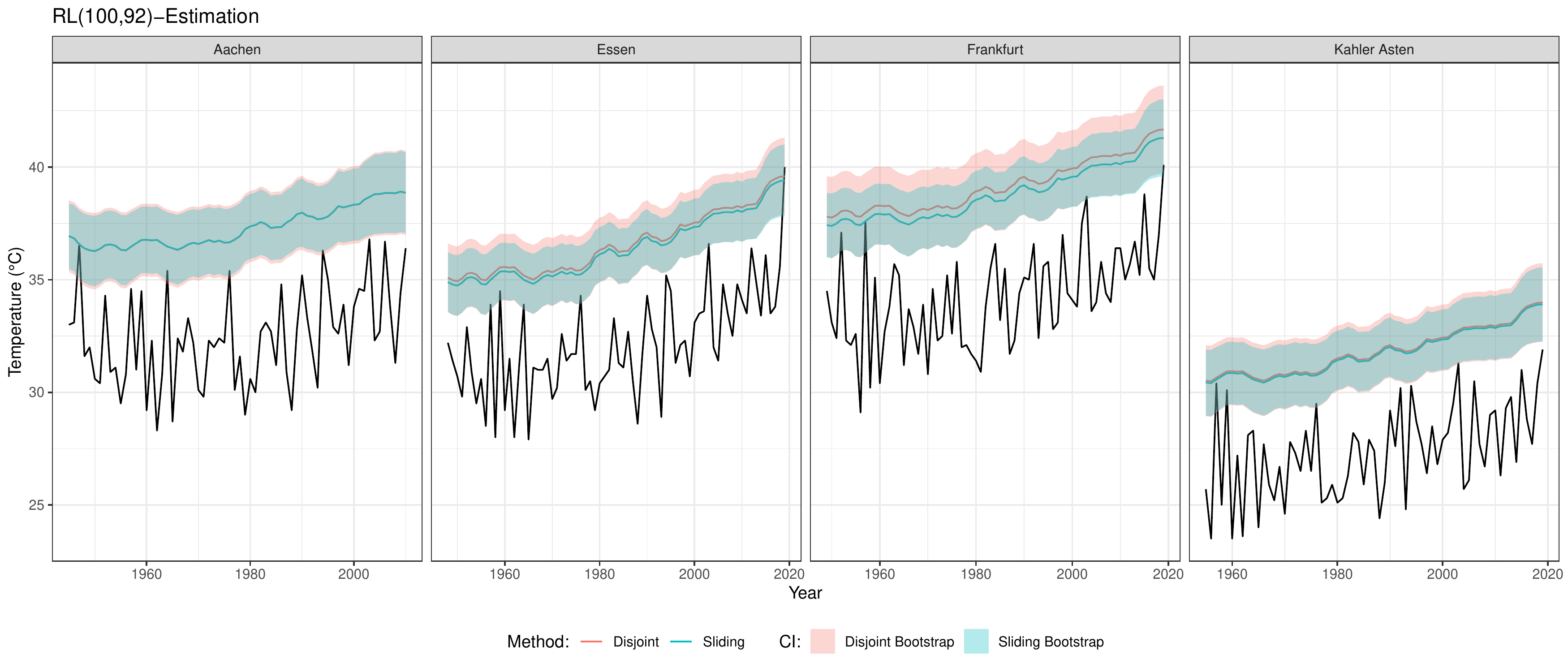}}\vspace{-.2cm}
	\caption{ \label{fig:RL_casestudy}  Estimated 100-year RL of summer months' maximal air temperate along with  95\% confidence regions as obtained from the  parametric bootstrap (10\,000 bootstrap repetitions) at several stations in the western part of Germany.}
	\vspace{-.4cm}
\end{figure}

\begin{table}
	\caption{\label{tab:slss} Estimated Parameters along with the standard deviation based on the bootstrap. }
\centering
\begin{tabular}{llllll}
	\hline
	Station & Method & Slope & Location & Scale & Shape \\ 
	\hline
	Aachen & Disjoint & 3.48 (1.07) & 31.04 (0.32) & 1.85 (0.18) & -0.19 (0.092) \\ 
	& Sliding & 3.48 (1.12) & 31.05 (0.32) & 1.96 (0.16) & -0.22 (0.081) \\ 
	Essen & Disjoint & 4.40 (0.71) & 30.26 (0.29) & 1.65 (0.16) & -0.18 (0.088) \\ 
	& Sliding & 4.40 (0.77) & 30.19 (0.31) & 1.84 (0.14) & -0.25 (0.076) \\ 
	Frankfurt & Disjoint & 3.69 (0.76) & 32.42 (0.31) & 1.71 (0.17) & -0.14 (0.089) \\ 
	& Sliding & 3.69 (0.82) & 32.43 (0.33) & 1.95 (0.15) & -0.25 (0.078) \\ 
	Kahler Asten & Disjoint & 3.33 (0.79) & 25.68 (0.34) & 1.75 (0.17) & -0.21 (0.093) \\ 
	& Sliding & 3.33 (0.83) & 25.59 (0.34) & 1.87 (0.15) & -0.25 (0.082) \\ 
	\hline
\end{tabular}
\end{table}

\section{Conclusion and discussion} \label{sec:con}

Large sample theory for univariate extreme value statistics based on the block maxima method is classically developed under the assumption that the block maxima  constitute a genuine i.i.d.\ sample from the GEV distribution. Two more realistic sampling schemes were considered in this paper: either an underlying stationary time series, or a triangular array consisting of independent blocks extracted from a stationary time series model.  The latter shall represent a typical situation encountered in environmental statistics, where stationarity can only be (approximately) guaranteed within seasons rather than years. 
Under certain additional regularity conditions, it was found that more efficient estimators can be obtained by taking into account all successive, overlapping block maxima. The results are remarkable in view of the fact that the time series of sliding block maxima is non-stationary under the second sampling scheme.  When restricted to the PWM estimator, the improvement was found to be substantial for negative shape parameters, both in large-sample theory and in finite-sample simulations. As a consequence, whenever computationally feasible, the sliding blocks estimator should be preferred over its disjoint blocks version. A possible approach to deal with non-stationarities between seasons was worked out in a case study involving temperature extremes, including a bootstrap approach to assess estimation uncertainty. 
The paper suggests several important topics for future research:

\begin{compactitem}
	\item[(a)] The new sampling scheme may be a worthwhile starting point for developing large-sample theory for other estimators commonly applied in extreme value statistics. Furthermore, in view of the simulation results in Section~\ref{sec:seasComp} and the theoretical results in \cite{stein2017}, the sampling scheme may be generalized to certain forms of inner-seasonal non-stationarity.
	\item[(b)] The developed theory shows that the sliding block maxima method can be applied in situations where the respective sample is non-stationary (with constant GEV parameters). This suggests that the general method may also be applicable in situations involving non-stationary models for the GEV parameters, possibly to be estimated by maximum likelihood then. 
	\item[(c)] The proof of Lemma~\ref{lem:weakdf_S2} suggests that the sliding blocks method may be generalized to some method involving an even larger subset of the set of all block maxima à la \cite{OorZho20}.
	\item[(d)] The parameteric bootstrap approach  has not been studied theoretically. Likewise,  possible alternative (nonparametric, block) bootstrap schemes could be investigated.
	\item[(e)] The asymptotic results may be used to derive more powerful formal tests for homogeneities within multivariate models, for instance involving a scaling model as imposed with the index flood assumption \citep{Hosking1997}.
\end{compactitem}

\section{Proofs}
\label{sec:proofs}

The main paper contains five theoretical results:  Lemma~\ref{lem:weakdf_S2} on weak convergence of sliding block maxima, Theorem~\ref{theo:weakh} on weak convergence of the empirical process of rescaled block maxima and Theorem~\ref{theo:pwm1}, Corollary~\ref{cor:pwm} and Corollary~\ref{cor:rl} on asymptotic normality of empirical PWMs, of derived parameter estimators and of return level estimators, respectively. 
Within this section, we provide proofs for the first three results,
with some intermediate steps postponed to Section~\ref{sec:proofsaux} or the supplementary material. The proof of Corollaries~\ref{cor:pwm}  and~\ref{cor:rl} are also provided in the supplement.

	\begin{proof}[Proof of Lemma~\ref{lem:weakdf_S2}] 
	Recalling the definition of $Z_{r,j}^{(\slb)}$ in \eqref{eq:zrjm}, the assertion to be shown in equivalent to convergence in distribution of
$
		Z_{r, 1+\ip{r\xi}}^{\scs (\slb)}
$
to $Z$, where $Z \sim G_\gamma$. 
		We omit the upper index $\slb$. Under sampling scheme (S1) the assertion holds by stationarity of the sliding block maxima. Consider sampling scheme (S2).
		For $\xi \in \{0,1\}$, the assertion holds by assumption. Let $\xi \in(0,1)$. By independence between and stationarity within blocks we get, for  any $x \in S_\gamma$,
		\begin{align*}
		&\phantom{{}={}}  
		\PP(Z_{r, 1+\ip{r\xi}} \le x)  \\
		&=
		\PP(\max(X_{1+ \ip{r\xi}}, \ldots, X_{r+\ip{r\xi}}) \leq a_r x+b_r ) \\
		&=
		\PP(\max(X_{1+ \ip{r\xi}}, \ldots, X_{r}) \leq a_r x+b_r) \PP( \max( X_{r+1 }, \ldots, X_{r+\ip{r\xi}}) \leq a_r x +b_r) \\
		&=  
		\PP\left( Z_{r -\ip{r\xi}, 1} \le \frac{a_r}{a_{r- \ip{r\xi}}} x + \frac{b_r- b_{r -\ip{r\xi}} }{a_{r- \ip{r\xi}}}   \right) 
		\PP\left( Z_{\ip{r\xi}, 1} \le \frac{a_r}{a_{\ip{r\xi}}} x + \frac{b_r- b_{\ip{r\xi}} }{a_{\ip{r\xi}}}   \right).
		\end{align*}
		Condition \ref{cond:mda} implies that the expression in the previous display converges to
		\begin{align} \label{eq:ggg}
		G_\gamma\Big( \frac{x}{(1-\xi)^\gamma} + \frac{(1-\xi)^{-\gamma}-1}{\gamma} \Big) G_\gamma\Big( \frac{x}{\xi^\gamma} + \frac{\xi^{-\gamma}-1}{\gamma} \Big)
		=
		G_\gamma(x),
		\end{align}
		where the last equation follows from a straightforward calculation.
	\end{proof}

\begin{proof}[Proof of Theorem~\ref{theo:weakh}]
	First, consider the case where $\mb = \djb$. Then we can write
	$
	\HH_r^{(\djb)}  =  \CC_{n,r}\circ H_r
	$
	almost surely, 
	where
	\begin{align*} 
	\CC_{n,r}(u) 
	= 
	\sqrt{\frac{n}r}  \frac{1}{\ndb }\sum_{j=1}^{\ndb } \left\{ \ind( H_{r}(Z_{r,j}^{(\djb)}) \leq u) -u \right\}, \qquad u \in [0,1].
	\end{align*}
	Under sampling scheme (S1), we have $\CC_{n,r} \dto \CC^{(\djb)} \ \text{in} \ \ell^\infty([0,1])$ by Theorem 3.1 in \cite{BucSeg14}, and similar (but simpler) arguments as in that proof show that the same convergence is met under sampling scheme (S2). Hence, by asymptotic equicontinuity, we obtain, 
	\[
	\CC_{n,r} \circ H_r = \CC_{n,r} \circ G_\gamma +o_\Prob(1)  \dto \CC^{(\djb)}  \circ G_\gamma
	\]
	as asserted. 

	Since the Brownian bridge $\CC^{(\djb)}$ has continuous trajectories almost surely, so does the limit process $\HH^{(\djb)}$.

	Now, let $\mb = \slb$, omit the upper index $\slb$, and note that we may redefine 
	\[
	\HH_r^{(\slb)}(x) = \frac{1}{\sqrt{nr}} \sum_{j=1}^{n-r}  \{ \ind(Z_{r,j} \le x) - \bar H_{r}(x) \}.
	\]
	First, we are going to show asymptotic tightness.
	For simplicity, assume $r/n\in 3 \N$. We may then write $\HH_r(x) =\HH_{r1}(x) + \HH_{r2}(x) + \HH_{r3}$, where
	\begin{align*}	
	\HH_{r\ell}(x)& = \frac{1}{\sqrt{nr}} \sum_{j \in J_{r}(\ell)}  \sum_{i \in I_j} \{ \ind(Z_{r,i} \le x) -  \bar H_{r}(x) \}
	\end{align*}
	with $J_r(\ell) = \{j \in \{1, \dots, \ndb -1\}: j \in 3\N_0 + \ell\}$ for $\ell= 1,2,3$ and $I_{j}= \{ (j-1)r +1 , \ldots, jr\}$ denoting the indices making up the $j$-th disjoint block of observations. It is sufficient to show asymptotic tightness  of each $\HH_{r\ell}$, and since they all have the same distribution we only consider the case $\ell = 1$.
	
	After successively applying Berbee's coupling lemma (\citealp{Ber79}, see also Lemma~4.1 in \citealp{DehPhi02}), we can construct a triangular array $\{\tilde{Z}_{r,i}\}_{i\in I_{1} \cup I_{4} \cup \dots} $ for which the following hold: 
	\begin{enumerate}[label=(\roman*)]
		\item For any $j\in J_r(1)$, we have $(\tilde{Z}_{r,i})_{i \in I_{j}} \stackrel{D}{=}({Z}_{r,i})_{i \in I_{j}}$.
		\item For any $j\in J_r(1)$, we have $\PP((\tilde{Z}_{r,i})_{i \in I_{j}} \ne ({Z}_{r,i})_{i \in I_{j}}) \leq \beta(r)$.     	
		\item $(\tilde{Z}_{r,i})_{i \in I_{1}}, (\tilde{Z}_{r,i})_{i \in I_{4}}, (\tilde{Z}_{r,i})_{i \in I_{7}} \dots$ is independent and identically distributed.
	\end{enumerate}
	Let $\tilde{\HH}_{r1}$ be defined in the same way  as $\HH_{r1}$, but in terms of $\{\tilde{Z}_{r,i}\}_{i\in I_{1} \cup I_{4} \cup \dots} $ instead of $\{Z_{r,i}\}_{i\in I_{1} \cup I_{4} \cup \dots} $. Asymptotic tightness of $\HH_{r1}$ follows once we show that
	\begin{align} \label{eq:hhhh}
	\|\HH_{r1}-\tilde \HH_{r1}\|_\infty=o_\PP(1)
	\end{align} 
	(where $\|H\|_\infty=\sup_{x\in\R} |H(x)|$)
	and that $\tilde \HH_{r1}$ is asymptotically tight. 
	
	Regarding the latter assertion, note that 
	\begin{align*}
	\tilde{\HH}_{r1}(x) &= \sum_{j \in J_{r}(1)}   \{ f_{r,j}(x) - \E[f_{r,j}(x)]  \},
	\end{align*}
	where
	\begin{align*}
	f_{r,j}(x) =\frac{1}{\sqrt{nr}} \sum_{l=1}^r \ind(\tilde{Z}_{r, (j-1)r +l} \leq x)	
	\end{align*} 
	Since the summands $f_{r,j}(x)$ making up $ \tilde\HH_{r1}$ are independent, we may apply classical results from empirical process theory for independent sequences. More precisely, asymptotic tightness follows from Theorem 11.16 in \cite{Kos08}, once we show that $\{f_{r,j}: j\in J_r(1)\}$ is almost measurable Suslin (AMS) and that Conditions (A)-(E) from that Theorem are met. The AMS property  follows from Lemma 11.15 in \citealp{Kos08}; use $T_n = \Q$ as the a countable subset to deduce separability.  The remaining items can be seen as follows:
	\begin{enumerate}[label=(\alph*), leftmargin=0.8cm]
		\item Since $x \mapsto f_{r,j}(x)$ is monotone  increasing, the discussion on p.~213 of \cite{Kos08} yields the manageability. The envelope functions can be chosen as 
		\[ 
		E_{r,j}(x) := \sqrt{r/n}, \quad j \in J_r(1),
		\]
		which are trivially independent.
		\item The limit $\lim_{n \to \infty} \E[  \tilde\HH_r(x) \tilde\HH_r(y)] $ exists for all $x,y \in \R$. Indeed,  since $f_{r,j}$ is independent of $f_{r,j'}$ when $ j \neq j'$, we have
		\begin{align*}
		\E[ \tilde\HH_{r1}(x) \tilde\HH_{r1}(y)]
		&= \frac{n}{3r} \Cov(f_{r,1}(x), f_{r,1}(y)) \\
		&= \frac{1}{3r}\sum_{l=1}^r \frac{1}{r} \sum_{m=1}^r\Cov( \ind(\tilde Z_{r,l} \leq x), \ind(\tilde Z_{r,m} \leq y)) \\
		&= \frac{1}{3} \int_0^{1} \int_{0}^{1} g_{r,x,y}(\xi, \xi') \diff \xi\diff\xi',
		\end{align*}
		where $g_{r,x,y}(\xi, \xi') = \Cov(g_x(\tilde Z_{r, 1+\ip{r\xi}}), g_y(\tilde Z_{r, 1+\ip{r\xi'}}))$ for $g_x(z) = \ind(z \leq x)$. Lemma~\ref{lem:jointsl} (sampling  scheme (S1)) and Lemma~\ref{lem:jointsl2} (sampling scheme (S2)) from the supplement yields
		\begin{align*}
		\lim_{r \to \infty } g_{r,x,y}(\xi, \xi') = G_{\gamma,\abs{\xi - \xi'}}(x,y) - G_\gamma(x)G_\gamma(y)
		\end{align*}
		and by dominated convergence, 
		we get 
		\begin{align*}
		\lim_{n\to\infty} \E[ \tilde \HH_{r1}(x) \tilde \HH_{r2}(y) ] = \frac13 \int_{0}^{1}\int_0^1 G_{\gamma , \abs{\xi - \xi'} }(x,y) \diff \xi \diff \xi' - G_\gamma(x)G_\gamma(y).
		\end{align*} 
		\item Since
		$
		\sum_{j \in J_r(1)} \E[E_{r,j}^2] = \frac{1}{3},
		$
		the sum of second moments of  the  envelopes is finite.
		\item For every $\varepsilon >0 $, we have
		\begin{align*}
		\limsup_{n\to\infty} \sum_{j\in J_r(1)} \E[ E_{r,j}^2\ind(E_{r,j} > \varepsilon)] 
		=  \limsup_{n\to\infty} \frac13 \ind(\sqrt{r/n}>\eps) =0.
		\end{align*}
		\item For $x,y \in \R$, let
		\[
		\rho_n(x,y) = \Big\{ \sum\nolimits_{j \in J_r(1)} \E\left[ \abs{f_{r,j}(x) - f_{r,j}(y) }^2 \right] \Big\}^{1/2}.
		\] 
		We  have to show that the pointwise limit of $\rho_n(x,y)$, say $\rho(x,y)$, exists and that, if $\lim_{n\to\infty}\rho(x_n, y_n) = 0$, then $\lim_{n\to\infty} \rho_n(x_n,y_n)  =0$. Without loss of generality assume $x \le y$.
		Then 
		\begin{align*}
		\qquad \rho_n(x,y)^2 &= \frac{1}{3r^2} \E\Big[\Big(\sum_{l=1}^r \ind(x < \tilde Z_{r,l} \leq y) \Big)^2 \Big] \\
		&= \frac{1}{2r^2} \sum_{l=1}^r \PP(x < \tilde Z_{r,l} \leq y) 
		+ \frac{1}{r^2} \sum_{ l=1}^r \sum_{h=l+1}^r \PP( x < \tilde Z_{r,l} \leq y, x < \tilde Z_{r,h} \leq y ).
		\end{align*}
		The first term is of order $1/r$ and thus converges to 0. The second one  equals 
		\begin{align*}
		\int_0^{1} \int_{\xi}^{1} \PP(x < \tilde Z_{1+ \ip{r\xi}} \leq y , x < \tilde Z_{1+\ip{r\xi'}} \leq y ) \diff \xi' \diff \xi .
		\end{align*}
		Due to Lemma \ref{lem:jointsl2} and dominated convergence, this converges to 
		\begin{align*}
		\rho(x,y)^2 &=\int_{0}^{1}\int_\xi^1 G_{\gamma, \abs{\xi - \xi'}}(x,x) +  G_{\gamma, \abs{\xi - \xi'}}(y,y) - 2 G_{\gamma, \abs{\xi - \xi'}}(x,y) \diff \xi'\diff\xi.
		\end{align*}
		This double integral can be calculated explicitly, where some care has to be taken on whether both, one, or none of the arguments $x, y$ fall into the support of $G_{\gamma, \abs{\xi -\xi'}}$. Since the first case is the most involved, we restrict to that case.   For $x \le y$ in such a way that  $1+\gamma x>0$ and  $1+\gamma y>0$, a straightforward calculation implies that
		\begin{align*}
		\rho(x,y)^2 &= \frac{e^{-\tilde{x}}}{\tilde{x}} \left\{ 1+ \frac{e^{-\tilde{x}} -1}{\tilde{x}} \right\}   
		+ \frac{e^{-\tilde{y}}}{\tilde{y}} \left\{ 1+ \frac{e^{-\tilde{y}}-1}{\tilde{y}} \right\}   
		-2 \frac{e^{-\tilde{x}}}{\tilde{y}} \left\{ 1+ \frac{e^{-\tilde{y}}-1}{\tilde{y}} \right\}, 
		\end{align*}
		where $\tilde{x} := (1+\gamma x )^{-1/\gamma} \ge \tilde{y} := (1+\gamma y )^{-1/\gamma}$.
		Obviously, $\rho(x,y) = 0$ for $x=y$. Write $g(s) := \{1+ (e^{-s}-1)/s \}/s$. Observing that $g(\tilde x) \le g(\tilde y)$, a careful calculation of derivatives shows that the function
		$
		[\tilde y, \infty) \to \R,
		\tilde{x} \mapsto g(\tilde x) e^{-\tilde{x}} -2g(\tilde y) e^{-\tilde{x}}$
		is strictly increasing. As a consequence,  $\rho(x, y)$ is strictly decreasing in $x$ (for $x\le y$) and can therefore have only one root which must be at $x=y$. Altogether, $\rho(x,y) = 0$ iff $x=y$. 
		But then, $\lim_{n\to \infty}\rho(x_n, y_n) = 0$ iff either $x = \lim_{n \to \infty}x_n = \lim_{n\to\infty} y_n$, the limit possibly being $\pm \infty$, or if $x_n=y_n$, eventually. In the latter case, $\rho_n(x_n,y_n)=0$ eventually, while in the former case, we have
		\begin{align*}
		\rho_n(x_n,y_n)  
		&\leq \frac{1}{r} \sum_{l=1}^r \norm{ \ind( x_n <\tilde Z_{r,l} \leq y_n) \ind(x_n < y_n)  +  \ind( y_n < \tilde Z_{r,l} \leq x_n) \ind(y_n < x_n)}_2 \\
		&= \frac{1}{r}\sum_{ l=1}^r \abs{H_{r,l}(x_n) - H_{r,l}(y_n)}  \\
		& \leq \abs{G_\gamma(x_n) -G_\gamma(y_n)}  + \frac{2}{r}\sum_{l=1}^r \norm{ H_{r, l} - G_\gamma}_\infty,
		\end{align*}
		which converges to 0 by continuity of $G_\gamma$ and Lemma \ref{lem:unicdf_s2} (sampling scheme (S2)) or Condition~\ref{cond:mda} (sampling scheme (S1), which implies $H_{r,\ell}=H_r$). 
	\end{enumerate} 		
		Finally, \eqref{eq:hhhh}  follows from
	\begin{align*}
	|\HH_{r1} (x) -  \tilde \HH_{r1} (x) |
	& \le
	\frac{1}{\sqrt{nr}}  \sum_{j \in J_{r}(1)}  \sum_{i \in I_j} | \{ \ind(Z_{r,i} \le x) -  \ind(\tilde Z_{r,i} \le x) | \\
	&\le 
	\frac{1}{\sqrt{nr}}  \sum_{j \in J_{r}(1)}  \sum_{i \in I_j} \{ \ind(Z_{r,i} \ne \tilde Z_{r,i} ) 
	\end{align*}
	for any $x\in\R$, which implies
	\[ 
	\Prob(\|  \HH_{r1} - \tilde \HH_{r1} \|_\infty > \eps)
	\lesssim \sqrt{\frac{n}r} \beta(r) = o(1)
	\]
	for any $\eps>0$ by  Markov's inequality and Condition~\ref{cond:rl}(ii). The proof of asymptotic tightness of $\HH_{r1}$ and hence of $\HH_r$ is finished.

  For the convergence of the finite dimensional distributions, note that indicator functions $g_x(z) = \ind(z \leq x)$ are elements of the set $\Gc'$ defined in Theorem~\ref{theo:normallbl}
	and that we may write
	\[ \HH_{r}^{(\slb)}(x) = \GG_{n}^{(\slb)}g_x. \]
	Therefore, Theorem~\ref{theo:normallbl} yields convergence to a centered normal distribution. Further, a simple calculation shows that 
	\[
	\lim_{n \to \infty}\Cov(\GG_{n}^{(\slb)}g_x, \GG_{n}^{(\slb)}g_y) = \Cov(\CC^{(\slb)}(G_\gamma(x)), \CC^{(\slb)}(G_\gamma(y)) )
	\]
	with covariance function of $\CC^{(\slb)}$ as defined in \eqref{eq:covCCsl}. 
	
	The inequality in \eqref{eq:varl11} is a special case of the result in Theorem 2.12 in \cite{ZouVolBuc21}, and may also be deduced from Lemma~\ref{lemma:Loewner_Cov} in the supplement.
	\end{proof}

	\begin{proof}[Proof of Theorem~\ref{theo:pwm1}]
Recall the definition of $Z_{r, j}^{(\djb)}$ and $Z_{r, j}^{(\slb)}$ from \eqref{eq:zrjm}
		and note that, under sampling scheme (S1), both are approximately $G_\gamma$-distributed with PWMs $\beta_{\gamma,k} = \beta_{(0,1,\gamma),k}$ for large~$r$ (in particular, unlike $M_{r,j}^{\scs (\cdot)}$,  the variables $Z_{r, j}^{\scs(\cdot)}$ are stochastically bounded for $r\to\infty$). The  same in fact holds under sampling scheme $(\sz)$, see Lemma \ref{lem:weakdf_S2}. For $\mb\in\{\djb, \slb\}$, $k\in\{0,1,2\}$, let
		\begin{align} \label{eq:tildebeta}
		\tilde \beta_{r,k}^{(\djb)}  = \hat  \beta_k( Z_{r,1}^{(\djb)}, \dots, Z_{r,\ndb }^{(\djb)}), 
		\qquad
		\tilde \beta_{r,k}^{(\slb)}  = \hat  \beta_k( Z_{r,1}^{(\slb)}, \dots, Z_{r,n-r+1}^{(\slb)}).
		\end{align} 
		Further, for  $f:\R \to \R$ integrable with respect to~$G_{\gamma}$,  let
		\begin{align}
		\label{eq:ggdjb}
		\GG_n^{(\djb)} f 
		&:= 
		\sqrt{\frac{n}{r}}\bigg( \frac{1}{\ndb } \sum_{j=1}^{\ndb } f(Z_{r,j}^{(\djb)}) - \E\left[ f(Z_{r,j}^{(\djb)}) \right]\bigg), \\
		\label{eq:ggslb}
		\GG_n^{(\slb)} f 
		&:= 
		\sqrt{\frac{n}{r}}\bigg( \frac{1}{n-r+1} \sum_{j=1}^{n-r+1} f(Z_{r,j}^{(\slb)}) - \E\left[ f(Z_{r,j}^{(\slb)}) \right]\bigg),
		\end{align}
		where $Z\sim G_\gamma$. In view of Condition~\ref{cond:bias2}, the proof of \eqref{eq:norm} is finished once we show that
		\begin{align}
		(1) & \qquad \label{eq:hattilde}
		\sqrt{\frac{n}{r}}  \bigg(\frac{\hat \beta_{r,k}^{(\mb)} - \beta_{\theta_r, k}}{a_r}\bigg)
		= 
		\sqrt{\frac{n}{r}}  \Big( \tilde \beta_{r,k}^{(\mb)} - \beta_{\gamma, k} \Big), \qquad\\
		(2) & \qquad \label{eq:exp}
		\sqrt{\frac{n}{r}}  \Big( \tilde \beta_{r,k}^{(\mb)} - \beta_{\gamma, k} \Big)
		= 
		\GG_n^{(\mb)} f_k + B_{n,k}^{(\mb, \samp)} + o_\Prob(1), \\
		(3) & \qquad \label{eq:norm2}
		\left(\GG_{n}^{(\mb)} f_k\right)_{k=0,1,2} \dto \Nc_3(0, \bm\Omega^{(\mb)}). \phantom{\sqrt{\frac{n}{r}}}
		\end{align}
		
		Subsequently, we write $M_{r,j}$ and $Z_{r,j}$ instead of $M_{r,j}^{\scs (\mb)}$ and $Z_{r,j}^{\scs (\mb)}$, respectively, whenever an equation is correct both for the disjoint and the sliding blocks version. 
		
		We begin by proving \eqref{eq:hattilde}, which holds irrespective of the sampling scheme. Note that $M_{r,(j)} =a_r Z_{r,(j)}+b_r$. Hence, since $\sum_{i=1}^{n} (i-1) = n(n-1)/2$ and  $\sum_{i=1}^{n} (i-1)(i-2)= n(n-1)(n-2)/3$,
		\begin{align} \label{eq:hatinv}
		\hat \beta_{r,k}^{(\mb)} = a_r \tilde \beta_{r,k}^{(\mb)} + \frac{b_r}{k+1}, \qquad k\in\{0,1,2\}.
		\end{align} 
		Likewise, recalling the notation $\beta_{\gamma,k} = \beta_{(0,1,\gamma),k}$, a simple calculation shows that
		\begin{align} \label{eq:betainv}
		\beta_{\theta_r,k} =  a_r \beta_{\gamma,k} + \frac{b_r}{k+1}, 
		\end{align}
		This implies \eqref{eq:hattilde}.
		The assertion in \eqref{eq:exp} is a consequence of Proposition~\ref{prop:expallbl}, and the weak convergence result in \eqref{eq:norm2} follows from Theorem~\ref{theo:normallbl}.
		Finally, the assertion in \eqref{eq:varl1} is a consequence of Lemma~\ref{lemma:Loewner_Cov}.
	\end{proof}

	\section{Auxiliary results} \label{sec:proofsaux}

	The following result is a central ingredient in the proofs of Theorem~\ref{theo:weakh} and Theorem~\ref{theo:pwm1}, and may be of independent interest. Its proof is similar to the proof of Theorem 3.6 in \cite{BucSeg18a} (disjoint blocks) and Theorem~2.6 in \cite{BucSeg18b} (sliding blocks), which is therefore postponed to the supplement.
	
	\begin{theorem}\label{theo:normallbl} 
		Assume that one of the sampling schemes from Condition~\ref{cond:obs} holds with $\gamma<1/2$. Let  
		\begin{align}\label{eq:defGc}
		\Gc = \{ g:\R \to \R  \text{ continuous }  \mid \exists\, c,d \text{ such that } |g(x)| \le c |x| + d \text{ for all } x \in \R\}.
		\end{align}
		If Conditions  \ref{cond:rl}  
		and 
		\ref{cond:unifint} hold, then, for arbitrary $g_1, \dots, g_p \in \Gc$, $p\in\N$, we have
		\begin{align*} 
		\left(\GG_n^{(\mb)} g_k\right)_{k = 1, \dots, p} 
		\xrightarrow{d} 
		\mathcal{N}_p\bigg( \bm 0, \Big(\bm \Lambda^{(\mb)}_{k,k'} \Big)_{k, k' = 1, \ldots, p}\bigg), 
		\end{align*}
		where $\GG_n^{(\mb)}$ is defined in \eqref{eq:ggdjb} and \eqref{eq:ggslb} and where, with $Z \sim G_\gamma$ and $(Z_{1\xi}, Z_{2\xi})	 \sim G_{\gamma, \xi}$,
		\begin{align*}
		\bm \Lambda_{k, k'}^{(\djb)} = \Cov(g_k(Z), g_{k'}(Z)), \qquad \bm \Lambda_{k, k'}^{(\slb)} = 2\int_0^1 \Cov(g_k(Z_{1\xi}), g_{k'}(Z_{2\xi})) \diff \xi.
		\end{align*}
		The same result holds with $\mathcal G$ replaced by  $\mathcal G'=\{ \bm 1_{(-\infty,t]}: t \in \R\}$; in that case, one may dispense with Condition~\ref{cond:unifint}.
	\end{theorem}

The next result serves the purpose of proving Equation~\eqref{eq:exp} in the proof of Theorem~\ref{theo:pwm1}.
	
	\begin{proposition} \label{prop:expallbl}
		Suppose one of the sampling schemes from Condition~\ref{cond:obs} holds with $\gamma<1/2$. If Conditions \ref{cond:rl} and \ref{cond:unifint} are met,  then, for $k\in\{0,1,2\}$, 
		\begin{align} \label{eq:expallbl}
		\sqrt{\frac{n}{r}}  \Big( \tilde \beta_{r,k}^{(\mb)} - \beta_{\gamma, k} \Big)
		=
		\GG_n^{(\mb)} f_{k} + B_{n,k}^{(\mb, \samp)} + o_\Prob(1),
		\end{align}
		with $\tilde \beta_{r,k}^{(\mb)}$, $\GG_n^{(\mb)}$, $f_k$ and $B_{n,k}^{(\mb,\samp)}$ as defined in \eqref{eq:tildebeta}, \eqref{eq:ggdjb} and \eqref{eq:ggslb}, \eqref{eq:fk} and Condition~\ref{cond:bias2}, respectively.     
	\end{proposition}
	
	\begin{proof}[Proof of Proposition~\ref{prop:expallbl}]
		We start by getting rid of the order statistics and claim that, for $k\in\{0,1,2\}$, 
		\begin{align} \label{eq:order_all}
		\tilde \beta_{r,k}^{(\mb)}= \bar \beta_{r,k}^{(\mb)} + O_\Prob(r/n),
		\end{align} 
		where
		\begin{align*} 
		\bar{\beta}_{r,k}^{(\djb)} 
		&= 
		\frac{1}{\ndb} \sum_{j=1}^{\ndb} Z_{r,j}^{(\djb) } \hat H_r^k(Z_{r,j}^{(\djb)}), \quad 
		&\bar{\beta}_{r,k}^{(\slb)} 
		&= 
		\frac{1}{\nsb} \sum_{j=1}^{\nsb} Z_{r,j}^{(\slb)} \hat H_r^k(Z_{r,j}^{(\slb)}),  \\
		\hat H_r^{(\djb)}(z) 
		&=
		\frac{1}{\ndb} \sum_{j=1}^{\ndb} \ind(Z_{r,j}^{(\djb)} \le z), \quad
		&\hat H_r^{(\slb)}(z) 
		&=
		\frac{1}{\nsb} \sum_{j=1}^{\nsb} \ind(Z_{r,j}^{(\slb)} \le z)
		\end{align*}
		The assertion is obvious for $k=0$. For $k \in \{1,2\}$, consider the disjoint and sliding case separately.

		\begin{enumerate}[label=(\roman*), leftmargin=0.6cm]
			\item First, let $\mb = \djb$, and omit the upper index $\djb$ for the ease of notation. Due to the no-tie assumption in Condition~\ref{cond:obs}, we have	 $
			(\hat H_r(Z_{r,(1)}), \dots, \hat H_r(Z_{r,(\ndb)})) = (1/\ndb,  \dots, \ndb/\ndb),
			$
			and therefore
			\begin{align*}
			\tilde \beta_{r,1} &= \frac{1}{\ndb} \sum_{j=1}^{\ndb} Z_{r,(j)} \frac{\ndb \hat H_r(Z_{r,(j)}) - 1}{\ndb-1}, \\
			\tilde \beta_{r,2} &= \frac{1}{\ndb} \sum_{j=1}^{\ndb} Z_{r,(j)} \frac{\ndb \hat H_r(Z_{r,(j)}) - 1}{\ndb-1} \frac{\ndb \hat H_r(Z_{r,(j)}) - 2}{\ndb-2}. 
			\end{align*}
			As a consequence,
			\[
			\tilde \beta_{r,1} -  \bar \beta_{r,1}
			=
			\frac1{\ndb(\ndb-1)} \sum_{j=1}^\ndb Z_{r,j} \{ \hat H_r(Z_{r,j})-1 \}
			=
			\frac{1}{\ndb-1} (\bar \beta_{r,1}-  \bar \beta_{r,0}).
			\]
			The arguments to follow imply that the expression on the right-hand side is of the order $O_\Prob(\ndb^{-1})= O_\Prob(r/n)$. The case $k=2$ can be treated similarly,  and \eqref{eq:order_all} is shown.
			
			\item Now let $\mb = \slb$ and again, suppress the upper index $\slb$.
			Write
			\begin{align*}
			\tilde \beta_{r,1} - \bar \beta_{r,1} 
			&= \frac{1}{(\nsb)(n-r)}\sum_{j=1}^{\nsb} Z_{r,(j)} \left\{ \hat{ H}_r(Z_{r, (j)}) -1\right\} + R_{n,1} \\
			&= \frac{1}{n-r} \left(\bar \beta_{r,1 } - \bar{\beta}_{r,0} \right) + R_{n,1},
			\end{align*}
			where 
			\begin{align} \label{eq:rn1}
			R_{n,1} 
			= 
			\frac{1}{\nsb} \sum_{j=1}^{\nsb} Z_{r,(j)} \frac{ j - (\nsb)\hat{ H}_r\left(Z_{r, (j)}\right)}{n-r}.
			\end{align}
			Again, the arguments to follow imply that $(\bar \beta_{r,1}-\bar \beta_{r,0})/(n-r)$ is of order $O_\PP(1/n)$. 
			For the treatment of $R_{n,1}$,  denote the $T$ different and ordered values of the scaled sliding block maxima by $ \tilde{Z}_{r,(1)} < \tilde{Z}_{r,(2)} < \ldots < \tilde{Z}_{r,(T)}$. Because of the no-tie assumption, we have $ T \geq n/r$, which can easily been seen from the fact that the $n/r$ pairwise different disjoint block maxima appear in the sequence of sliding block maxima as well. Now set 
			\[ 
			V_t := \left\{ j \in \{1, \ldots, \nsb \}: Z_{r,j} = \tilde{Z}_{r, (t)} \right\}, 
			\quad 
			t \in\{ 1, \ldots,T\},
			\]
			which defines a partition of $\{ 1, \ldots, n-r+1\}$.
			We have $\alpha_t=|V_t| \leq r$, because otherwise the no-tie assumption would be violated. 
			The empirical c.d.f.\ $\hat H_r$ is a step function that jumps up by $\alpha_t/(n-r+1)$ in the points $\tilde Z_{r, (t)} $, so we have $\hat{H}_r(\tilde Z_{r, (t)} ) = \sum_{ s= 1 }^{t} \alpha_s/ (\nsb)$. 
			Further, for each element $Z_{r,(j)}$ of the ordered sample $Z_{r,(1)}, \dots, Z_{r,(n-r+1)}$, we can find an index $t_j$ such that $Z_{r,(j)}=  \tilde{Z}_{r,(t_j)}$. As a consequence,
			$
			\sum_{s = 1}^{t_j -1} \alpha_s < j \leq \sum_{ s= 1 }^{t_j} \alpha_s,
			$
			which in turn implies
			\begin{align*}
			(n-r+1)\hat{ H}_r\left(Z_{r,(j)}\right) - \alpha_{t_j} < j \leq  (n-r+1)\hat{ H}_r\left(Z_{r,(j)}\right).
			\end{align*}
			Hence, by the definition of $R_{n,1}$  in \eqref{eq:rn1}, we have
			\begin{align*}
			\abs{R_{n,1}}  < \frac{1}{\nsb} \sum_{j=1}^{\nsb} \abs{Z_{r,(j)}} \frac{\alpha_{t_j}}{n-r} \leq \frac{r}{n-r} O_\PP(1) = O_\PP(r/n),
			\end{align*}
			where the $O_\PP(1)$-term follows from $\Exp[|Z_{r,j}|]= \Exp[|Z|]+o(1)$ by Lemma~\ref{lem:momconv}  and $\Var\{(n-r+1)^{-1} \sum_{j=1}^{n-r+1} |Z_{r,j}|\}=o(1)$ by Lemma~\ref{lem:asymcovsl} and \ref{lem:asymcovsl_S2} in the supplement. This proves  \eqref{eq:order_all} for $k=1$, and  the case $k=2$ can be treated similarly with slightly more effort.
		\end{enumerate}
		As a consequence, \eqref{eq:order_all} is shown,  and hence, for proving the proposition, it suffices to show \eqref{eq:expallbl} with $\tilde \beta_{r,k}$ replaced by $\bar \beta_{r,k}$.  The assertion is immediate for $k=0$. For $k\in\{1,2\}$, decompose
		\begin{align*} 
		\sqrt{n/r}  \Big( \bar \beta_{r,k}^{(\mb)} - \beta_{\gamma, k} \Big)
		= \XX_{n,k}^{(\mb)}  + \YY_{n,k}^{(\mb)} + B_{n,k}^{(\mb, \samp)}
		\end{align*}
		where
		\begin{align*}
		\XX_{n,k}^{(\djb)}  
		&=   
		\sqrt{\frac{n}{r}} \frac{1}{\ndb }\sum_{j=1}^{\ndb} Z_{r,j}^{(\djb)} \Big\{ (\hat{H}_r^{(\djb)}(Z_{r,j}^{(\djb)}))^k -  {H}_r^k(Z_{r,j}^{(\djb)}) \Big \}, \\
		\YY_{n,k}^{(\djb)}   
		&= 
		\sqrt{\frac{n}{r}} \Big\{ \frac{1}{\ndb}\sum_{j=1}^{\ndb } Z_{r,j}^{(\djb)} H_r^k(Z_{r,j}^{(\djb)}) - \E\Big[Z_{r,j}^{(\djb)}H_r^k(Z_{r,j}^{(\djb)})\Big]\Big \},\\
		\XX_{n,k}^{(\slb)}  
		&=   
		\sqrt{\frac{n}{r}} \frac{1}{\nsb }\sum_{j=1}^{\nsb} Z_{r,j}^{(\slb)} \Big\{ (\hat{H}_r^{(\slb)}(Z_{r,j}^{(\slb)}))^k -  \bar{H}_r^k(Z_{r,j}^{(\slb)}) \Big \}, \\
		\YY_{n,k}^{(\slb)}   
		&= 
		\sqrt{\frac{n}{r}} \Big\{ \frac{1}{\nsb}\sum_{j=1}^{\nsb } Z_{r,j}^{(\slb)}\bar H_r^k(Z_{r,j}^{(\slb)}) - \E\left[Z_{r,j}^{(\slb)}\bar H_r^k(Z_{r,j}^{(\slb)})\right]\Big\},
		\end{align*}
		Recall that $\bar H_r=H_r$ for $(\mb, \samp) \ne (\slb, \sz)$, which we will occasionally use to rewrite the above expressions.
		Observing that, for $k\in\{1,2\}$,    $f_k$ from \eqref{eq:fk} may be written as
		$f_k = f_{k,1}+ f_{k,2}$ with
		\begin{align} \label{eq:fkdec}
		f_{k,1}(x) = xG_\gamma^k(x), \qquad f_{k,2}(x) = \int\limits_{x}^\infty y\, \nu_k'(G_\gamma(y)) \, \mathrm{d}G_\gamma(y)
		\end{align}
		and $\nu_k(x) =  x^k$,
		the proposition is shown once we show that, for $\mb \in \{\djb, \slb\}$,
		\begin{align}\label{eq:xxyy}
		\XX_{n,k}^{(\mb)}   =  \GG_n^{(\mb)}f_{k,2} + o_\Prob(1), \qquad  \YY_{n,k}^{(\mb)} = \GG^{(\mb)}_nf_{k,1} + o_\Prob(1).
		\end{align}

		For the second assertion,  it is sufficient to show that $\Var(\YY_{n,k}^{(\mb)} - \GG_n^{(\mb)}f_{k,1})=o(1)$, by centeredness. For that purpose,  write
		\[ 
		Y_{n,j}^{(\mb)} = Z_{r,j}^{(\mb)} \{ \bar{H}_r^k(Z_{r,j}^{(\mb)}) - G_\gamma^k(Z_{r,j}^{(\mb)}) \},
		\] 
		and consider $\mb\in\{\djb, \slb\}$ separately.

		First, let $\mb = \djb$, and omit the upper index $\djb$ for notational convenience.   		Then, by stationarity and assuming $\ndb = n/r \in \N$ for simplicity (otherwise, a negligible remainder shows up), 
		\begin{align} \label{eq:ynkvar} \nonumber
		\Var(\YY_{n,k} - \GG_nf_{k,1})  &=
		\Var(Y_{n,1}) + 2\sum_{h=1}^{\ndb -1} \frac{\ndb -h}{\ndb } \Cov(Y_{n,1}, Y_{n,1+h})  \\
		&\leq 3\Var(Y_{n,1}) + 20 \|Y_{n,1} \|^{2}_{2+\delta} \sum_{h=2}^{\ndb -1} \alpha(\sigma(Y_{n,1}), \sigma(Y_{n,1+h}))^{\frac\delta{2+\delta}},
		\end{align}
	where $\|Y \|_{p} =\Exp[|Y|^p]^{1/p}$ and where the last inequality follows from the Cauchy-Schwarz inequality and Lemma 3.11 in \cite{DehPhi02}, with $\frac{1}{p}= \frac{1}{q}=\frac{1}{2+\delta}$ and $\delta\in[2/\omega, \nu)$ and $\nu$ from Condition~\ref{cond:unifint}.
		Since the sum starts at $h=2$, the variables generating the sigma fields depend on observations which are separated by a time lag of at least $r$, so each summand is smaller than or equal to $\alpha(r)^{\delta/(2+\delta)}$. Further, noting that $\bar{H}_r = H_{r}$,
		\[
		\Var(Y_{n,1})  \le \|Y_{n,1}\|_{2+\delta}^2 
		\le 
		\| H_r^k - G^k_\gamma\|_\infty^2 \Exp[|Z_{r,1}|^{2+\delta}]^{2/(2+\delta)} =o(1)
		\]
		by Conditions~\ref{cond:mda} and  \ref{cond:unifint}, where $\|F\|_\infty=\sup_{x\in\R}|F(x)|$. As a consequence,
		\[
		\hspace{0.6cm}
		\Var(\YY_{n,k} - \GG_nf_{k,1})  \le o(1) \{ 3+20 \cdot \ndb \alpha(r)^{\frac\delta{2+\delta}} \}
		\]
		which converges to zero by Condition~\ref{cond:rl}(ii), observing that $\delta(1+\omega)\ge 2+\delta$ by the choice of $\delta$. Hence, the second assertion in \eqref{eq:xxyy} is shown.
		
		Next, consider $\mb = \slb$ and, again assuming $n/r \in \N$, let
		\begin{align*}
		I_h := \{ (h-1)r+1 , \ldots, hr\}, \quad h \in \{1, \ldots, n/r \},
		\end{align*}
		denote the set of indices making up the $h$-th disjoint block of observations. Then 
		\begin{align}\label{eq:sumAh}
		\frac{\sqrt{n/r}}{n-r+1}\sum_{j=1}^{n-r+1} Y_{n,j} 
		= 
		(1+o(1)) \frac{1}{\sqrt{nr}}   \Big\{\sum_{h=1}^{n/r-1} A_h + Y_{n, n-r+1} \Big\},
		\end{align}
		where $A_h := \sum_{s\in I_h}Y_{n,s}$. It is sufficient to show that $\Var((nr)^{-1/2}\sum_{h=1}^{n/r -1} A_h ) = o(1)$, since  the last summand in \eqref{eq:sumAh} is asymptotically negligible.
		By stationarity of $(A_h)_h$, we get    		
		\begin{multline}\label{eq:varsummeah}
		v_n  = \Var\Big(\frac{1}{\sqrt{nr}}   \sum_{h=1}^{n/r -1} A_h \Big) 
		=
		\frac1{nr} \Bigl\{ \left(\frac{ n}{r} -1\right) \Var(A_1)   
		+ 2\left(\frac n r -2\right)\Cov(A_1, A_2) \Bigr.  \\
		\Bigl. + 2\left(\frac n r -3\right)\Cov(A_1, A_3) +  2 \sum_{h=3}^{n/r-2}\left( \frac n r -1 - h \right)\Cov(A_1, A_{1+h}) \Bigr\}.
		\end{multline} 
		Now, by the Cauchy-Schwarz and Minkowski inequality,
		\begin{align}
		| \Cov(A_1, A_{1+h})  |   
		&\leq \norm{A_1}_2^2 
		= 
		\Big\| \sum_{s=1}^r Y_{n,s} \Big\|_2^2 
		\leq r^2\Big(\| \bar{H}_{r}^k -G_\gamma^k  \|_\infty \frac{1}{r}\sum_{s=1}^r \norm{Z_{r,s}}_2 \Big)^2,
		\label{eq:AbschVarA1}
		\end{align}
		where the right-hand side is equal to
		$r^2 \| H_{r}^k -G_\gamma^k  \|_\infty^2 \norm{Z_{r,1}}_2^2$ under sampling scheme (S1).
		Likewise, for $\delta \in [2/\omega, \nu)$ with $\omega$ and $\nu$ as in Conditions~\ref{cond:rl} and \ref{cond:unifint}, we have
		\begin{align}\label{eq:AbschCovA1}
		\abs{\Cov(A_1, A_{1+h})}  
		&\leq  
		10 \norm{A_1}_{2+\delta}^2 \alpha(\sigma(A_1), \sigma(A_{1+h}))^{\frac{\delta}{2+\delta}}  
		\leq 
		10 \norm{A_1}_{2+\delta}^2 \alpha((h-2)r)^{\frac{\delta}{2+\delta}}
		\end{align}
		by Lemma 3.11 in \cite{DehPhi02}.
		Combining  \eqref{eq:AbschVarA1}, \eqref{eq:AbschCovA1} and the fact that the sum on the right-hand side of \eqref{eq:varsummeah} starts at $h = 3$, we get
		\begin{align*}
		\hspace{.6cm} 
		v_n 
		&\lesssim 
		\frac{1}{nr}  r^2 \| \bar H_r^k - G_\gamma^k  \|_\infty^2 \Big\{5  \frac n r  \Big(\frac{1}{r}\sum_{s=1}^r \norm{Z_{r,s}}_2\Big)^2 + 20 \left(\frac{n}{r}\right)^2 \Big(\frac{1}{r}\sum_{s=1}^r \norm{Z_{r,s}}_{2+\delta}\Big)^2  \alpha(r)^{\frac{\delta}{2+\delta}}\Big\} \\
		&= 
		\| \bar H_r^k - G_\gamma^k \|_\infty^2 
		\Big\{  5 \Big(\frac{1}{r}\sum_{s=1}^r \norm{Z_{r,s}}_2\Big)^2 + 20  \Big(\frac{1}{r}\sum_{s=1}^r \norm{Z_{r,s}}_{2+\delta}\Big)^2 \frac{n}{r}\alpha(r)^{\frac{\delta}{2+\delta}}\Big\}\\
		&= o(1) \{O(1) + O(1)o(1) \}      
		= o(1)
		\end{align*}
		where the orders of the terms in brackets follow from Lemma \ref{lem:unicdf_s2}  (sampling scheme (S2)) or $\bar H_r=H_r$ and Condition~\ref{cond:mda} (sampling scheme (S1)), and 
		from Conditions \ref{cond:rl}(ii) and \ref{cond:unifint} in combination with Lemma~\ref{lem:momconv}.

		Having treated the cases $\mb \in \{\djb, \slb\}$, the second assertion in \eqref{eq:xxyy} is shown, and it remains to treat the first one. Its proof will be split into two parts:
		\begin{align}\label{eq:xf2}
		\XX_{n,k}^{(\mb)} = \XXs^{(\mb)}_{n,k} + o_\PP(1), \qquad  
		\XXs^{(\mb)}_{n,k}= \GG_n^{(\mb)}f_{k,2} + o_\PP(1),    		\end{align}
		where
		\[ 
		\XXs^{(\mb)}_{n,k} := \sqrt{\frac{n}{r}}   \int_\R y \nu_k'(\bar H_r(y)) \left\{ \hat{H}^{(\mb)}_r(y) - \bar H_r(y)\right\} \, \mathrm{d}\hat{H}^{(\mb)}_r(y).
		\]

		The first assertion in \eqref{eq:xf2} is immediate for $k=1$; in that case,  even $\XX^{(\mb)}_{n,1} = \XXs^{(\mb)}_{n,1}$. 
	
		Treating the case $k=2$ is more difficult, and for that purpose, let $\mathcal{P}(\R)$ denote the set of all probability measures on $\R$ and let
		\begin{align}
		A= \{ f: \R \to \R: \norm{f}_\infty < \infty \ \text{and} \ f \ \text{is Borel-measurable} \}  \subset \ell^\infty(\R), \label{def:setA}
		\end{align}
		equipped with the uniform metric.
		Further, let
		\begin{align}
		\label{eq:defw1}
		W_1 = W_1(\R) = \{ \mu \in \mathcal{P}(\R):\int \abs{x}\diff \mu(x) < \infty   \}
		\end{align}
		denote the Wasserstein space of order 1, equipped with the Wasserstein metric
		\begin{align*}
		d_{W_1}(\mu,\nu) = \sup_{h \in \mathrm{Lip}_1} \abs{ \int_\R h(x) \diff \mu(x) - \int_\R h(x)\diff \nu(x)}, 
		\end{align*}
		where $\mathrm{Lip}_1$ is the set of all Lipschitz-continuous functions with Lipschitz-constant 1.   Recall that a sequence $(\mu_n)_n$ of probability measures in 
		$W_1$ is said to converge weakly in $W_1$ to another probability measure $\mu \in W_1$, if 
		\begin{align}\label{def:W1weakconv}
		\mu_n \to \mu \ \text{weakly} \qquad \text{and} \qquad \int_\R \abs{x}\diff \mu_n(x) \to \int_\R \abs{x}\diff\mu(x),
		\end{align}
		for $n\to \infty$. 
		The Wasserstein metric metrizes weak convergence in $W_1$  (Theorem~6.8 in \citealp{villani2008optimal}):
		\begin{align}\label{equiv:W1weak}
		\mu_n \to \mu \ \text{weakly in} \ W_1 \ \Leftrightarrow \ d_{W_1}(\mu_n, \mu) \to 0,
		\end{align} 
		Another equivalent property for weak convergence in $W_1$ is as follows (Definition 6.7 in \citealp{villani2008optimal}): 
		for all continuous functions $\varphi$ such that $\abs{\varphi(x)} \leq C(1+\abs{x})$ for all $x$ and some $C=C_\varphi$, it holds that
		\begin{align}\label{def:W1weakconv_alt}
		\int_\R \varphi(x)\diff \mu_n(x) \xrightarrow{n\to\infty} \int_\R \varphi(x) \diff\mu(x).  
		\end{align}
		Now, consider the first assertion in \eqref{eq:xf2} with $k=2$. Observing that $H_r=\bar H_r$ for $\mb=\djb$, we have
		\begin{align} \label{eq:xxprime}
		\XX^{(\mb)}_{n,2} 
		&= \nonumber
		\sqrt{\frac{n}r} \int_\R y \left( (\hat{H}^{(\mb)}_r(y))^2 - \bar H_r^2(y) \right) \, \mathrm{d}\hat{H}^{(\mb)}_r(y) \\
		&=
		\sqrt{\frac{n}r} \int_\R y \left( \hat{H}^{(\mb)}_r(y) - \bar H_r(y) \right) \left( \hat{H}^{(\mb)}_r(y) + \bar H_r(y) \right) \, \mathrm{d}\hat{H}^{(\mb)}_r(y)  
		= 
		\XXs^{(\mb)}_{n,2} + R_{n,2}^{(\mb)} ,
		\end{align}
		where
		\[
		R_{n,2}^{(\mb)} = \sqrt{\frac{n}r} \int_\R y \left( \hat{H}^{(\mb)}_r(y) - \bar H_r(y) \right)^2 \, \mathrm{d}\hat{H}^{(\mb)}_r(y).
		\]
		Now, 
		$
		|R_{n,2}^{(\mb)}| \le \| \hat H ^{(\mb)}_r - \bar H_r\|_\infty \psi( \HH_r^{(\mb)},1, \hat H_r^{(\mb)}), 
		$
		where   $\psi(a,g, \mu)= \int_\R |y|g(y) a(y) \diff \mu(y)$
		and where
		\[
		\HH_r^{(\mb)}(y) = \sqrt{\frac{n}r} \left( \hat{H}^{(\mb)}_r(y) - \bar H_r(y) \right), \qquad y \in \R.
		\]
		It then follows from continuity of $\psi$ (Lemma~\ref{lem:psiphi}), $d_{W_1}(\hat H_r^{(\mb)}, G_\gamma)=o_\Prob(1)$ (Lemma~\ref{lem:dw1h}) and weak convergence of $\HH_r^{\scs (\mb)}$ in $\ell^\infty(\R)$ to a process with continuous and bounded sample paths almost surely (Theorem~\ref{theo:weakh})  that $R_{n,2}^{\scs (\mb)}=O_\Prob((r/n)^{-1/2})=o_\Prob(1)$, which implies the assertion by \eqref{eq:xxprime}.
		
		It remains to show the second assertion in \eqref{eq:xf2}, for $k\in\{1,2\}$. For that purpose, write
		\begin{align*}
		\XXs^{(\mb)}_{n,k}
		=
		\phi(\HH^{(\mb)}_r, \nu_k' \circ \bar H_r,\hat H^{(\mb)}_r),
		\end{align*}
		where
		\[
		\phi: A \times C_b(\R) \times W_1  \to \R, 
		\qquad (a,g,\mu) \mapsto \int_\R y g(y) a(y) \diff \mu(y).
		\]
		Likewise, a simple calculation shows that we may write
		\begin{align*} 
		\GG^{(\mb)}_nf_{k,2} 
		&= \nonumber
		\sqrt{\frac{n}r} \int_{-\infty}^{\infty} y\nu_k'(G_\gamma(y)) \left\{ \hat{H}^{(\mb)}_r(y) - \bar H_r(y)\right\} \, \mathrm{d}G_\gamma(y) 
		= 
		\phi(\HH^{(\mb)}_r, \nu_k' \circ G_\gamma, G_\gamma).
		\end{align*}
		The second assertion in \eqref{eq:xf2} is then again an immediate consequence of continuity of~$\phi$ (Lemma~\ref{lem:psiphi}), since $d_{W_1}(\hat H_r^{(\mb)}, G_\gamma)=o_\Prob(1)$ (Lemma~\ref{lem:dw1h}),
		$\|\bar H_r - G_\gamma\|_\infty=o(1)$ (Condition~\ref{cond:mda} or, for sampling scheme (S2), Lemma \ref{lem:unicdf_s2}) and since $\HH^{(\mb)}_r$ converges weakly in $\ell^\infty(\R)$ to a continuous limit (Theorem~\ref{theo:weakh}). 
		The proof of Proposition~\ref{prop:expallbl} is finished.
	\end{proof}

\section*{Acknowledgements}
This work has been supported by the integrated project ``Climate Change and Extreme Events - ClimXtreme Module B - Statistics (subproject B3.3)'' funded by the German Federal Ministry of Education and Research (BMBF), which is gratefully acknowledged. Computational infrastructure and support were provided by the Centre for Information and Media Technology at Heinrich Heine University Düsseldorf.

The authors are grateful to three unknown referees and an associate editor for their constructive comments that helped to improve the presentation substantially. The authors also appreciate valuable comments by  Petra Friederichs, Marco Oesting and multiple participants of the Extreme Value Analysis (EVA) Conference 2021 in Edinburgh.


\bibliographystyle{imsart-number}
\bibliography{biblio}       


\newpage

\thispagestyle{empty}

\numberwithin{equation}{section}


\begin{center}
	
	{\bfseries SUPPLEMENT TO THE PAPER:  \\  ``ON THE DISJOINT AND SLIDING BLOCK MAXIMA METHOD \\ FOR PIECEWISE STATIONARY TIME SERIES'' }
	\vspace{.5cm}
	
	{\textsc{By Axel Bücher and Leandra Zanger}}
	
	\vspace{.28cm}
	
	{\textit{Heinrich-Heine-Universit\"at D\"usseldorf}}
	
	\vspace{.28cm}

	\begin{center}
		\begin{minipage}{.6\textwidth}
			{\small \hspace{.5cm}
					Missing proofs for the results of the main paper are presented in Appendix~\ref{sec:proofs-add}, with a couple of further theoretical results postponed to Appendix~\ref{sec:proofs-add2}.
					Appendix~\ref{sec:var} contains explicit formulas for asymptotic covariance matrices appearing in the main paper. Appendix~\ref{app:sim} contains additional simulation results. Finally, Appendix~\ref{supp:sec:cs} contains a figure supporting the case study.
					References like Lemma 1.9, Figure 0, or Equation (4) always refer to the main paper. 
				 }
		\end{minipage}
	\end{center}

\end{center}

\vspace{.5cm}

	The theoretical results are organized as follows: in Appendix~\ref{sec:proofs-add}, we provide the missing proofs of Corollaries~\ref{cor:pwm} and \ref{cor:rl} and Theorem \ref{theo:normallbl} from the main paper. Appendix~\ref{sec:proofs-add2} contains further theoretical results used throughout the proofs, and is decomposed into three sections: 
	\begin{itemize}
		\item Section~\ref{subsec:wms} is about weak convergence and moment convergence of sliding block maxima.  Joint weak convergence of sliding block maxima  is considered in Lemma~\ref{lem:jointsl}  (sampling scheme (S1)) and Lemma~\ref{lem:jointsl2} (sampling scheme (S2)); the results may be considered as bivariate extensions of Lemma~\ref{lem:weakdf_S2} from the main paper and are later used for calculating asymptotic covariances. Lemma~\ref{lem:unicdf_s2} is about (uniform) convergence of the average cdf $\bar H_r$ from Equation~(\ref{eq:barhr}) under (S2); it is needed in the proofs of Theorem~\ref{theo:weakh} and Proposition~\ref{prop:expallbl}. Moment convergence of block maxima is the content of Lemma~\ref{lem:momconv}, which is deduced from weak convergence and uniform integrability, the latter being part of Lemma~\ref{lem:unifint_s2}.
		\item Section~\ref{subsec:asycov} is about  asymptotic covariances for empirical moments of block maxima, as required in the proof of the general asymptotic normality result in Theorem~\ref{theo:normallbl}. Sampling scheme (S1) is treated in Lemma~\ref{lem:asymcovsl}, while sampling scheme (S2) is treated in Lemma~\ref{lem:asymcovsl_S2}. Finally, Lemma~\ref{lemma:Loewner_Cov} states that  the sliding blocks limiting covariance in Theorem~\ref{theo:normallbl} is never larger than its disjoint blocks counterpart.
		\item Section~\ref{subsec:faux} contains further auxiliary results. First, Lemma~\ref{lem:psiphi} provides consistency of some abstract functionals which were employed in the proof of Proposition~\ref{prop:expallbl}.
		Next,  Lemma~\ref{lem:dw1h} provides Wasserstein consistency of $\hat H_r^{\scs(\mb)}$ for $G_\gamma$, a technical result needed in the proof of Proposition~\ref{prop:expallbl} that eventually allows to dispense with arguments involving weighted weak convergence as used for deriving PWM asymptotics in \cite{KojNav17}.
		Its proof may partly be generalized to a more abstract setting, which has been formulated in a separate
		Lemma~\ref{lem:rcw1}.
		Finally, Lemma~\ref{lem:shortblock} and Lemma~\ref{lem:maxinbig} are simple adaptations  of Lemma~A.7  and A.8 in \cite{BucSeg18a} which are needed for the blocking technique.
	\end{itemize}
	Section~\ref{sec:var} contains two lemmas: Lemma~\ref{lem:formelcov} provides formulas for the asymptotic covariance in Theorem~\ref{theo:pwm1}, while Lemma~\ref{lem:JacobiC} provides formulas for the Jacobi matrix in Corollary~\ref{cor:pwm}.
	
	Last but not least, Section~\ref{app:sim} contains additional simulation results, collected in a sequence of subsections, i.e. additional simulation results for fixed blocksize (Section \ref{supp:sec:fixr}), additional simulation results for fixed samplesize (Section \ref{supp:sec:fixn}), results for comparing sampling schemes (S1) and (S2) (Section \ref{supp:sec:s1s2}), additional simulation results for a different marginal distributions (Section \ref{supp:sec:hw}) and results for comparing ML and PWM estimation (Section \ref{supp:sec:ml}).

\begin{appendix}
	
	\section{Missing proofs for  results from the main paper}
	\label{sec:proofs-add}

	\begin{proof}[Proof of Corollary~\ref{cor:pwm}]
		For the ease of reading, we omit the upper index $\mb$.
		Recall $\phi$ defined in (\ref{eq:phi}). Clearly, for $\bm \beta=(\beta_0,\beta_1,\beta_2)' \in \Dc_\phi$,
		\begin{align*}
		\renewcommand{\arraystretch}{1.3}
		\left( \begin{array}{c}
		\phi_1(\bm \beta)  \\
		\frac{1}{a_r} \phi_2(\bm \beta)  \\ 
		\frac{1}{a_r}\phi_3(\bm \beta) - \frac{b_r}{a_r}
		\end{array}\right)     
		=
		\phi \left( \begin{array}{c}
		\frac{\beta_0 -b_r}{a_r}  \\
		\frac{\beta_1 -{b_r}/{2} }{a_r}  \\ 
		\frac{\beta_2 -{b_r}/{3}}{a_r} 
		\end{array}\right).
		\end{align*}
		As a consequence, by (\ref{eq:hatinv}) and (\ref{eq:betainv}),
		\begin{align*}
		\renewcommand{\arraystretch}{1.3}
		\sqrt{\frac{n}{r}} 
		\left(\begin{array}{c}  \hat \gamma_{r}- \gamma \\ (\hat a_{r} -a_r)/{a_r}  \\  (\hat b_{r} - b_r)/{a_r}\end{array}  \right)     
		&=  
		\sqrt{\frac{n}{r}}  \left\{  \left( \begin{array}{c}
		\phi_1(\hat {\bm \beta_r})  \\
		\frac{1}{a_r} \phi_2(\hat {\bm \beta_r})   \\ 
		\frac{1}{a_r}\phi_3(\hat {\bm \beta_r}) - \frac{b_r}{a_r}
		\end{array}\right) 
		- 
		\left( \begin{array}{c}
		\phi_1(\bm \beta_{\theta_r})  \\
		\frac{1}{a_r} \phi_2(\bm \beta_{\theta_r})   \\ 
		\frac{1}{a_r}\phi_3(\bm \beta_{\theta_r}) - \frac{b_r}{a_r}
		\end{array}\right) 
		\right\} \\
		&=
		\sqrt{\frac{n}{r}}  \left\{ \phi(\tilde{\bm \beta}_r) -  \phi({\bm \beta}_\gamma)  \right\},
		\end{align*}
		where $\bm \beta_{\gamma}=(\beta_{\gamma,0}, \beta_{\gamma,1}, \beta_{\gamma,2})'$. The assertion in (\ref{eq:norm3}) is now a consequence of (\ref{eq:exp}), (\ref{eq:norm2}), Condition~\ref{cond:bias2} and the delta method.
		Finally, the assertion in (\ref{eq:varl2}) is an immediate consequence of (\ref{eq:varl1}). 
	\end{proof}

	\begin{proof}[Proof of Corollary~\ref{cor:rl}] We omit the upper index $\mb$. For $a >0 $ and $\gamma\in\R$ let $f(\gamma, a) =  a \frac{ c_T^{-\gamma}-1}{\gamma}$. Note that $a_r^{-1}f(\gamma, a)= f(\gamma, a/a_r)$ and that $\nabla f(\gamma, 1)$ is equal to the first two coordinates of $q_T$. As a consequence, by the delta method, 
		\begin{align*}
		\sqrt{n/r} \left( \frac{\widehat \RL(T,r) - \RL(T,r)}{a_r} \right)
		&=
		\sqrt{n/r} \left( f\left(\hat{\gamma}_r, \tfrac{\hat{a}_r}{a}\right) -  f(\gamma, 1)  + \frac{\hat b_r - b_r}{a_r} \right), \\
		&=
		q_T' \sqrt{n/r} \left(\begin{array}{c}  \hat \gamma_{r}- \gamma \\ (\hat a_{r} -a_r)/{a_r}  \\  (\hat b_{r} - b_r)/{a_r}\end{array}  \right) + o_\PP(1).
		\end{align*}
		The assertion then follows from Corollary \ref{cor:pwm}.
	\end{proof}
	
	\begin{proof}[Proof of Theorem~\ref{theo:normallbl}] We start by considering the function class $\mathcal G$.
		The disjoint blocks case is a straightforward adaptation of the proof of Theorem 3.6 in \cite{BucSeg18a} and is therefore omitted.
		For the sliding blocks case, we may follow the proof of Theorem 2.6 in \cite{BucSeg18b}, with substantial modifications for sampling scheme (S2). The basic idea consists of successively merging blocks of size $r$ into a `big block of blocks' followed by a `small block of blocks' followed by a `big block of blocks' and so on in such a way that the `small blocks of blocks' are small enough to be asymptotically negligible for the asymptotics and at the same are large enough to make the `big blocks of blocks' asymptotically independent, whence standard central limit theorems become available. We omit the upper index $\slb$. 
		Since $\Gc$ is a vector space and by the Cramér-Wold-device, it suffices to show that, for any fixed $g\in \Gc$,
		\begin{align}\label{eq:ggv}
		\GG_{n}g \dto \Nc(0, \sigma^2), \qquad \sigma^2 = \int_0^1 \Cov(g(Z_{1\xi}, g(Z_{2\xi})))\diff\xi.
		\end{align}
		For that purpose, let $I_j := \{ (j-1)r+1, \ldots, jr\}, j\in\{ 1, \ldots, \ndb-1\}$, denote the set of indices making up the $j$-th disjoint block of observations.
		Let $\ndb^\star= \ndb^\star_n $ be an integer sequence with $3 \leq \ndb^\star \leq \ndb -1$ that converges to infinity and satisfies $\ndb^\star = o(\ndb^{\delta/\{ 2(1+\delta)\} })$ for  some $\delta \in (\frac{2}{\omega}, 2+\nu)$.
		For simplicity, assume that $q = (\ndb-1)/\ndb^\star \in \N$.
		For $j \in \{1, \ldots,q \}$, let 
		\[ 
		J_j^+ := I_{(j-1)\ndb^\star +1} \cup \ldots \cup I_{j\ndb^\star -2}, \qquad J_j^- := I_{j\ndb^\star-1} \cup I_{j\ndb^\star} ,
		\]
		such that $|J_j^+| = (\ndb^\star - 2)r$ and $|J_j^-| = 2r$. Then, by (\ref{eq:ggslb}), 
		\begin{align} 
		\GG_n g 
		&= \nonumber
		\sqrt{\frac{n}r} \left(\frac{1}{\dnsl} \sum_{j=1}^\dnsl \left( g(Z_{r,j}) - \E\left[ g(Z_{r,j})\right]\right)\right)  \\
		&=\nonumber
		(1+o(1)) \frac{1}{\sqrt{nr}}
		\sum_{j=1}^{q} \Big\{ \sum_{s\in J_j^+} \left( g(Z_{r,s}) - \E\left[ g(Z_{r,s})\right]\right) +  \sum_{s\in J_j^-} \left( g(Z_{r,s}) - \E\left[ g(Z_{r,s})\right]\right)\Big\} \\
		&\hspace{5cm} \nonumber
		+(1+o(1))\frac{1}{\sqrt{nr}} \big( g(Z_{r,\nsb}) - \E\left[ g(Z_{r,\nsb})\right] \big) \\
		&=  \label{eq:gndec}
		(1+o(1))  \Big\{ \frac{1}{\sqrt{q}} \sum_{j=1}^q S_{n,j}^+ +  \frac{1}{\sqrt{q}} \sum_{j=1}^q S_{n,j}^- \Big\} + o_{L_2}(1), 
		\end{align}
		where $S_{n,j}^{\pm} := \sqrt {q/(nr)} \sum_{s\in J_j^{\pm}} \{ g(Z_{r,s}) - \E\left[ g(Z_{r,s})\right]\}$. Note that $(S_{n,j}^{\pm})_j$ is stationary for both of the sampling schemes (S1) and (S2).
		
		We will next argue that the contribution of the `small blocks' is  negligible. 
		Since $\E[S_{n,j}^-] = 0$, this follows if the variance is shown to converge to 0. 
		We have
		\begin{align} 
		\Var\Big( \frac{1}{\sqrt{q}} \sum_{j=1}^q S_{n,j}^- \Big)
		&= \nonumber
		\Var(S_{n,1}^-) + \frac{2}{q} \sum_{h=1}^{q-1} (q-h) \Cov(S_{n,1}^-, S_{n, 1+h}^-) \\
		&\leq \label{eq:vsnm}
		3 \Var(S_{n,1}^-) + 2 \sum_{h=2}^{q-1} \Big(1-\frac{h}{q}\Big) | \Cov(S_{n,1}^-, S_{n, 1+h}^-) |
		\end{align}
		by Cauchy-Schwarz.
		By stationarity across blocks, we may write
		\begin{align} \label{eq:vsn1} \nonumber
		\Var(S_{n,1}^-) &= \Big\| \sqrt{\frac{q}{nr}} \sum_{s\in J_1^-}  \left( g(Z_{r,s}) - \E\left[ g(Z_{r,s})\right]\right)\Big\|_2^2 \\
		&\leq \frac{q}{nr} \Big(\sum_{s \in J_1^-} \| g(Z_{r,s}) - \E\left[ g(Z_{r,s})\right] \|_2 \Big)^2 
		\le 4  \frac{qr}{n} \left( \frac{1}{r}\sum_{s=1}^r \Var(g(Z_{r,s})) \right)^2.
		\end{align}
		Since $qr/n = q/\ndb = O(1/\ndb^\star)=o(1)$, we obtain that  $\Var(S_{n,1}^-) = o(1)$ by Condition~\ref{cond:unifint} (sampling scheme (S1)) or Condition~\ref{cond:unifint} and Lemma \ref{lem:unifint_s2} (sampling scheme (S2)). It remains to consider the sum on the right-hand side of \eqref{eq:vsnm}, which is equal to zero under sampling scheme (S2). For sampling scheme (S1), we may apply  Lemma 3.11 in \cite{DehPhi02} with $1/p =1/q= 1/(2+\nu)$ to obtain
		\begin{align*}
		\big|\Cov(S_{n,1}^-, S_{n,h}^-)\big| &\leq 10 \| S_{n,1}^- \|_{2+\nu}^2 \alpha(\sigma(S_{n,1}^-), \sigma(S_{n,h}^-))^{\frac{\nu}{2+\nu}} 
		\leq 10 \| S_{n,1}^- \|_{2+\nu}^2 \alpha(r)^{\frac{\nu}{2+\nu}}
		\end{align*}
		for $h \geq 3$. Therefore
		\begin{align*}
		\sum_{h=2}^{q-1}| \Cov(S_{n,1}^-, S_{n, 1+h}^-) |
		&\lesssim  q \| S_{n,1}^- \|_{2+\nu}^2\alpha(r)^{\frac{\nu}{2+\nu}} 
		\lesssim
		q \frac{qr}{n} \alpha(r)^{\frac{\nu}{2+\nu}} 
		=
		\frac{\ndb}{(\ndb^\star)^2} \alpha(r)^{\frac{\nu}{2+\nu}},
		\end{align*}
		which converges to zero since $\ndb\alpha(r)^{\frac{\nu}{2+\nu}} = o(1)$ by Condition~\ref{cond:rl}(ii) and the choice of $\nu$ in Condition~\ref{cond:unifint}.

		The sum over the small blocks being negligible, it remains to show that $ q^{-1/2} \sum_{j=1}^q S_{n,j}^+$ converges in distribution to a centered normal distribution with variance $\sigma^2$ as in \eqref{eq:ggv}.  For sampling scheme (S2), $(S_{n,j}^{+})_j$ is a rowwise independent triangular array, and a standard argument based on characteristic functions shows that we may assume  the same for sampling scheme (S1). As a  consequence, we may apply Ljapunov's central limit theorem, for which we need to check Lyapunov's Condition: 
		\begin{align}\label{eq:lyap}
		\exists\ \delta>0: \quad  \lim_{n\to\infty} \frac{\sum_{j=1}^q\E[ |S_{n,j}^+|^{2+\delta}]}{\big\{ \sum_{j=1}^q\E[ |S_{n,j}^+|^2] \big\}^{1+\frac{\delta}{2}}} = 0.
		\end{align}
		Now,  by Condition~\ref{cond:unifint} (sampling scheme (S1)) or Condition~\ref{cond:unifint} and Lemma \ref{lem:unifint_s2} (sampling scheme (S2)),
		\begin{align*}
		\|S_{n,j}^+\|_{2+\delta} 
		&\leq \sqrt{\frac{q}{nr}} (\ndb^\star -2) \sum_{s\in I_1}\| g(Z_{r,s}) - \E\left[g(Z_{r,s})\right] \|_{2+\delta} \\
		&\lesssim \sqrt{\ndb^\star} \frac{1}{r}\sum_{s=1}^r\| g(Z_{r,s}) - \E\left[g(Z_{r,s})\right] \|_{2+\delta} = O(\sqrt{\ndb^\star}).
		\end{align*}
		As a consequence, provided that $\E[ |S_{n,j}^+|^2] $ is converging to a non-zero constant, the fraction in \eqref{eq:lyap} is of the order $O(q^{-\delta/2} (\ndb^{\star})^{1+\delta/2}) = O(\ndb^{-\delta/2} (\ndb^\star)^{1+\delta})=o(1)$ by the choice of $\delta$ and $\ndb^\star$ in the paragraph below \eqref{eq:ggv}.
		
		Finally, $\E[ |S_{n,j}^+|^2]=\Var(S_{n,j}^+)=\Var(q^{-1/2} \sum_{j=1}^q S_{n,j}^+)$ converges to $\sigma^2$ since
		\[  
		\lim_{n\to\infty} \Var\Big(q^{-1/2} \sum_{j=1}^q S_{n,j}^+\Big)= \lim_{n\to\infty} \Var(\GG^{(\slb)}_ng)
		\]
		by \eqref{eq:gndec} and since we have shown that $\|q^{-1/2} \sum_{j=1}^q S_{n,j}^-\|_2 = o(1)$. The right-hand side of the previous display is equal to $\sigma^2$ by Lemma~\ref{lem:asymcovsl} (sampling scheme (S1)) and Lemma~\ref{lem:asymcovsl_S2} (sampling scheme (S2)).
		
		Finally, for the function class $\mathcal G'$ of indicator functions, the previous proof remains valid with only minor modifications:  the right hand-side of \eqref{eq:vsn1} converges to zero since finite linear combinations of indicators are bounded. For the arguments that follow, one may apply (the simpler) Lemma~3.9 rather than  Lemma 3.11 in \cite{DehPhi02}.
	\end{proof}

	\section{Additional theoretical results used throughout the proofs} \label{sec:proofs-add2}

	\subsection{Weak convergence and  moment convergence of sliding block maxima}
	\label{subsec:wms}
	
	\begin{lemma}[Joint weak convergence of sliding block maxima under (S1)]\label{lem:jointsl} \phantom{a}
		Consider sampling scheme (S1) from Condition~\ref{cond:obs}, let $r_n\to\infty$ with $r_n=o(n)$ and
		suppose there exists a sequence of integers $(\ell_n)_n$ such that $\ell_n = o(r_n), \alpha(\ell_n) = o(\ell_n/r_n)$ for $n \to \infty$. 
		Then, for any $\xi \geq 0$ and $x,y \in \R$,
		\begin{align*}
		\lim\limits_{n\to\infty} \PP(Z_{r, 1}^{(\slb)} \leq x, Z_{r, 1+\ip{r\xi}}^{(\slb)} \leq y) = G_{ \gamma,\xi}(x,y), 
		\end{align*}
		with $G_{\gamma, \xi } $ as in (\ref{def:Ggamxi}) for $\xi\in[0,1]$ and $G_{\gamma, \xi}(x,y) = G_{\gamma}(x)G_\gamma(y)$ for $\xi>1$.
	\end{lemma}
	
	\begin{proof}
		We omit the upper index  $\slb$.
		The case $\xi = 0$ is trivial. For $j,k \in \N$ with $j \leq k$, let $M_{j:k} := \max(X_j, \ldots, X_k)$.  By similar arguments as in the  proof of Lemma 5.1 in \cite{BucSeg18a} (see below for details), we have, for $\xi \in(0,1)$,
		\begin{align}{}
		&\phantom{{}={}} \nonumber 
		\PP(Z_{r, 1} \leq x, Z_{r, 1+\ip{r\xi}} \leq y) \\
		&= \nonumber 
		\PP(M_{1:\ip{r\xi}} \leq a_{r}x + b_{r}, M_{\ip{r\xi} +1 :r} \leq a_{r} (x \wedge y) + b_{r}, 
		M_{r+1: r + \ip{r\xi}} \leq a_{r}y + b_{r})  \\
		&= \nonumber 
		\PP(M_{1:\ip{r\xi} - \ell} \leq a_{r}x + b_{r}, M_{\ip{r\xi} +1 :r-\ell} \leq a_{r} (x \wedge y) + b_{r}, 
		M_{r+1: r + \ip{r\xi}} \leq a_{r}y + b_{r}) + o(1) \\
		&=   \nonumber
		\PP(M_{1:\ip{r\xi} - \ell} \leq a_{r}x + b_{r}) \PP(M_{\ip{r\xi} +1 :r-\ell} \leq a_{r} (x \wedge y) + b_{r}) \\
		& \hspace{7.5cm} \nonumber 
		\PP( M_{r+1: r + \ip{r\xi}} \leq a_{r}y + b_{r})+ o(1) \\
		&=  \nonumber
		\PP(M_{1:\ip{r\xi}} \leq a_{r}x + b_{r}) \PP(M_{\ip{r\xi} +1 :r} \leq a_{r} (x \wedge y) + b_{r}) \\
		& \hspace{7.5cm}   
		\PP( M_{r+1: r + \ip{r\xi}} \leq a_{r}y + b_{r}) + o(1). \label{eq:jointdis}
		\end{align}
		From the last expression we can then follow the claimed limit, since Condition~\ref{cond:mda} implies
		\begin{align*}
		\lim_{n\to\infty} \PP(M_{1:\ip{r\xi}} \leq a_{r}x + b_{r}) 
		&= 
		\PP\left(Z_{\ip{r\xi},1} \leq \frac{a_{r}}{a_{\ip{{r}\xi}}} x + \frac{b_{r} - b_{\ip{r\xi}}}{a_{\ip{r\xi}}} \right)\\
		&= 
		G_\gamma\left(\xi^{-\gamma} x + \frac{\xi^{-\gamma} -1}{\gamma}\right), 
		\end{align*} 
		and analogously 
		\begin{align*}
		\lim_{n \to \infty}\PP( M_{r+1: r + \ip{r\xi}} \leq a_{r}y + b_{r}) 
		&= 
		G_\gamma\left(\xi^{-\gamma} y + \frac{\xi^{-\gamma} -1}{\gamma}\right), \\
		\lim_{n\to\infty}\PP(M_{\ip{r\xi} +1 :r} \leq a_{r} (x \wedge y) + b_{r})  
		&= 
		G_\gamma\left((1-\xi)^{-\gamma} (x\wedge y) + \frac{(1-\xi)^{-\gamma} -1}{\gamma}\right).
		\end{align*}
		Multiply the latter three limits to arrive at $G_{\gamma, \xi}$. 
		
		Explanation of \eqref{eq:jointdis}: the first equality  is obvious. 
		For the second equality, note that $\PP(A_n \cap B_n) = \PP(A_n) + o(1)$ provided that   
		$\lim_{n \to \infty}\PP(A_n \cap B_n^\mathrm{c}) = 0$. Therefore, 
		\begin{align*}
		\PP( \{M_{1:\ip{r\xi}} \leq a_{r}x + b_{r} \}) 
		&=  
		\PP( \{M_{1:\ip{r\xi} -\ell} \leq a_{r}x + b_{r} \}\cap \{ M_{\ip{r\xi} -\ell+1:\ip{r\xi}} \leq a_{r}x + b_{r} \}) \\
		&=  \PP( \{M_{1:\ip{r\xi} -\ell} \leq a_{r}x + b_{r} \} + o(1),  
		\end{align*}
		in view of
		\begin{align*}
		&\phantom{{}={}} \PP\left( \{M_{1:\ip{r\xi} -\ell} \leq a_{r}x + b_{r} \}\cap \{ M_{\ip{r\xi} -\ell+1:\ip{r\xi}} \leq a_{r}x + b_{r} \}^\mathrm{c}\right) \\
		&\leq \PP\left(M_{1:\ip{r\xi} - \ell} < M_{\ip{r\xi } - \ell + 1: \ip{r\xi}}\right) \\
		&=  \PP\left(M_{1:\ip{r\xi} - \ell} < M_{1: \ip{r\xi}}\right),
		\end{align*}
		and the last expression is of order $o(1)$ by Lemma \ref{lem:maxinbig}.  	With the same argument we can cut off the last $\ell$ observations in $ M_{\ip{r\xi} +1 :r}$ to treat $\PP( M_{\ip{r\xi} +1 :r} \leq a_{r} (x \wedge y) + b_{r})$. Combining this we get the second equality. The third equality follows because $\alpha(\ell) =o(1)$, and the observations from one considered set to another consist of observations which are at least $\ell$ apart. The last equality can be proven in the manner of the second one, just reversely. 
		
		Finally, the case $\xi >1$ can be proven in a similar way and is even easier, since in this case the blocks under consideration do not overlap.
	\end{proof}

	\begin{lemma}[Joint weak convergence of sliding block maxima under (S2)] \label{lem:jointsl2}\phantom{a}
		Consider sampling scheme (S2) from Condition~\ref{cond:obs},  let $r_n\to\infty$ with $r_n=o(n)$ and
		suppose there exists a sequence of integers $(\ell_n)_n$ such that $\ell_n = o(r_n), \alpha(\ell_n) = o(\ell_n/r_n)$ for $n \to \infty$.
		Then, for any $\xi, \xi' \geq 0$ and any $x,y \in S_\gamma$
		\begin{align*}
		\lim\limits_{n\to\infty} \PP(Z_{r, 1+\ip{r\xi}}^{(\slb)} \leq x, Z_{r, 1+\ip{r\xi'}}^{(\slb)} \leq y) = G_{ \gamma,\abs{\xi - \xi'}}(x,y), 
		\end{align*}
		with $G_{\gamma, \xi } $ as in (\ref{def:Ggamxi}) for $\xi\in[0,1]$ and $G_{\gamma, \xi}(x,y) = G_{\gamma}(x)G_\gamma(y)$ for $\xi>1$.
	\end{lemma}
	
	\begin{proof}
		If $\abs{\xi - \xi'} \geq 1$, the respective block maxima are independent, whence their joint c.d.f\ is the product of their marginal c.d.f.s and the result follows from Lemma~\ref{lem:weakdf_S2}. For $\abs{\xi - \xi'} < 1$ the proof is a slight adaptation of the proof of Lemma \ref{lem:jointsl}.
	\end{proof}

	\begin{lemma}[Convergence of average cdfs under (S2)]\label{lem:unicdf_s2} 
		Consider sampling scheme (S2) from Condition~\ref{cond:obs}. Then, with $\bar H_r$ as defined in Condition~\ref{cond:bias2},
		\begin{align*}
		\lim_{n \to \infty}	\sup_{x\in \R} \abs{\bar{H}_r(x) - G_\gamma(x) } =0.
		\end{align*} 
	\end{lemma}

	\begin{proof}
		Recalling $F_{r,j}$ from (\ref{eq:def:Frj}), we may write
		\begin{align} \label{eq:hff}
		H_{r,j+1}(x) 
		&= F_{r,j+1}(a_r x + b_r) \nonumber \\
		&= \PP(\max(X_{j+1}, \ldots, X_r) \leq a_r x + b_r, \max(X_{r+1}, \ldots, X_{r+j}) \leq a_r x   +    b_r) \nonumber \\
		&= F_{r-j}(a_r x + b_r)F_{j}(a_r x +b_r),
		\end{align}
		with $F_{0}\equiv 1$. We may thus write 
		\begin{align}\label{eq:hrmg}
		\Big\| \frac{1}{r} \sum_{j=1}^r H_{r,j} - G_\gamma  \Big\|_\infty \nonumber 
		&\le
		\frac{1}{r} \sum_{j=1}^r \big\| H_{r,j} - G_\gamma  \big\|_\infty \\
		&= \nonumber 
		\int_0^1 \| H_{r, \ip{r\xi}+1} - G_\gamma \|_\infty \diff \xi \\
		&= 
		\int_{0}^{1} \big\| F_{r-\ip{r\xi}}(a_r\cdot + b_r)F_{\ip{r\xi}}(a_r\cdot +b_r) - G_\gamma(\cdot) \big\|_\infty \diff\xi. 
		\end{align}
		Recalling identity (\ref{eq:ggg}) and invoking the triangular inequality after adding and subtracting $F_{r-\ip{r\xi}}(a_r\cdot + b_r)G_\gamma(\xi^{-\gamma} \cdot + \frac{\xi^{-\gamma} -1}{\gamma})$, we obtain the bound
		\[
		\big\| F_{r-\ip{r\xi}}(a_r\cdot + b_r)F_{\ip{r\xi}}(a_r\cdot +b_r) - G_\gamma(\cdot)\big\|_\infty 
		\le A_{r1}(\xi) + A_{r2}(\xi),
		\]
		where
		\begin{align*}
		A_{r1}(\xi) &= \norm{F_{r-\ip{r\xi}}(a_r\cdot + b_r) - G_\gamma((1-\xi)^{-\gamma}\cdot + ((1-\xi)^{-\gamma} -1)/{\gamma})}_\infty  \\
		A_{r2}(\xi) &= \norm{F_{\ip{r\xi}}(a_r\cdot +b_r) - G_\gamma(\xi^{-\gamma}\cdot + (\xi^{-\gamma} -1)/{\gamma})}_\infty.
		\end{align*}
		Using the fact that $F_{\ip{r\xi}}(x) = H_{\ip{r\xi}}((x - b_{\ip{r\xi}})/a_{\ip{r\xi}})$, we have the bound
		\begin{align*}
		A_{r2}(\xi) \leq \norm{ H_{\ip{r\xi}} - G_\gamma}_\infty + R_r(\xi), 
		\end{align*}
		where
		\begin{align*}
		R_r(\xi) := \left\| G_\gamma\left(\frac{a_r}{a_{\ip{r\xi}}} \cdot + \frac{ b_r - b_{\ip{r\xi}}}{a_{\ip{r\xi}}}\right) - G_\gamma\left( \xi^{-\gamma}\cdot + \frac{\xi^{-\gamma} -1}{\gamma} \right) \right\|_\infty.
		\end{align*}
		Likewise,
		\begin{align*}
		A_{r1}(\xi) \le \norm{ H_{r-\ip{r\xi}} - G_\gamma}_\infty + R_r(1-\xi), 
		\end{align*}
		We can thus conclude that the right-hand side of 	\eqref{eq:hrmg}  may be  bounded by
		\begin{multline*}
		\int_{0}^{1} \norm{H_{\ip{r\xi}} - G_\gamma}_\infty + \norm{H_{r-\ip{r\xi}} - G_\gamma}_\infty + R_r(\xi) + R_r(1-\xi)\diff \xi \\
		= 2\int_{0}^1  \norm{H_{\ip{r\xi}} - G_\gamma}_\infty \diff \xi + 2 \int_0^1 R_r(\xi)\diff \xi. 
		\end{multline*}
		The two integrals on the right-hand side converge to zero by dominated convergence and Equations~(\ref{eq:firstorder}) and~(\ref{eq:rvscale}) from Condition~\ref{cond:mda}, respectively.    
	\end{proof}

	\begin{lemma}[Moment convergence of block maxima] \label{lem:momconv}
		Consider one of the sampling schemes from Condition~\ref{cond:obs} with $\gamma<1/2$. Suppose
		there exists some $\nu >0$ such that 
		\[ 
		\limsup_{r\to\infty} \E[\abs{Z_{r}}^{2+\nu}] < \infty,
		\]
		and let $f$ be a real-valued function for which there exist constants $c,d\in[0,\infty)$ and $0 \leq \mu < 2+\nu $ with $|f(x)| \le c|x|^{\mu}+d$ for all $x\in\R$. Then,  with $Z\sim G_\gamma$,
		\begin{align*}
		\lim_{r\to \infty} \E[f(Z_{r})] = \E[f(Z)],  \qquad
		\lim\limits_{r \to \infty} \E\Big[  \frac1r \sum_{j=1}^r f(Z_{r,j}^{(\slb)}) \Big] &= \E [f(Z)].
		\end{align*}
	\end{lemma}
	
	\begin{proof}
		The first assertion is an immediate consequence of weak convergence (Condition~\ref{cond:mda}) and uniform integrability. This  readily implies the second assertion under sampling scheme (S1).
		For sampling scheme (S2), we may write
		\[
		\E\Big[  \frac1r \sum_{j=1}^r f(Z_{r,j}^{(\slb)}) \Big]  = \int_0^1 \Exp[f(Z_{r,1+\ip{r\xi}}^{(\slb)})] \diff \xi.
		\]
		The expression inside the integral converges to $\Exp[f(Z)]$ by weak convergence (Lemma~\ref{lem:weakdf_S2}) and uniform integrability (Lemma~\ref{lem:unifint_s2}). Since the upper bound in the latter lemma holds uniformly in $\xi$, the assertion follows from dominated convergence. 
	\end{proof}

	\begin{lemma}[Uniform integrability under (S2)]
		\label{lem:unifint_s2} 
		Consider sampling scheme (S2) from Condition~\ref{cond:obs} with $\gamma < 1/2$ and suppose that
		\begin{align} \label{eq:zrnu}
		\limsup_{r\to\infty} \E[\abs{Z_{r}}^{2+\nu}] < \infty,
		\end{align}
		for some $\nu>0$. Then	
		\begin{align*}
		\limsup_{r\to\infty} \sup_{\xi \in [0,1]} \E\Big[\big| Z_{r,1+\ip{r\xi}}^{(\slb)}\big|^{2+\nu} \Big] < \infty.
		\end{align*}
	\end{lemma}

	\begin{proof}
		Throughout, we omit the upper index $\slb$ and assume $r/2\in \N$ for simplicity. By \eqref{eq:hff}, the random variables $Z_{r,1+\ip{r\xi}}$ and $Z_{r,r-\ip{r\xi}+1}$ have the same distribution, whence it is sufficient to restrict  the supremum to $\xi\in[0,1/2]$.
		Next, note that
		\begin{align*}
		\max(X_{1+r/2}, \ldots, X_{r})  &\leq \max(X_{1+\ip{r\xi}}, \ldots, X_{r + \ip{r\xi}}) \leq \max(X_1, \ldots, X_{2r}),
		\end{align*}
		which may be written as $ M_{r/2, 1+ r/2} \leq M_{r, 1+\ip{r\xi}} \leq M_{2r,1} = M_{r,1} \vee M_{r,r+1}$.
		As a consequence,
		\begin{align*}
		|Z_{r, 1+\ip{r\xi}}| 
		= \Big| \frac{M_{r, 1+\ip{r\xi}} - b_r}{a_r} \Big|  
		&\leq 
		\Big| \frac{M_{r,1} -b_r}{a_r} \vee  \frac{M_{r,r+1}-b_r}{a_r} \Big| + \Big| \frac{M_{r/2, 1+r/2} - b_r}{a_r} \Big| \\
		&\leq
		| Z_{r,1}  | + | Z_{r,1+r} | + \Big| Z_{r/2,1+r/2 } \frac{a_{r/2}}{a_r} + \frac{b_{r/2}-b_r}{a_r} \Big|,
		\end{align*}
		which implies
		\begin{align*} 
		\| Z_{r,1+\ip{r\xi}} \|_{2+\nu} 
		\le 
		2 \|Z_{r,1}\|_{2+\nu} 
		+ \| Z_{r/2,1} \|_{2+\nu} \Big| \frac{a_{r/2}}{a_r} \Big| + \Big| \frac{b_{r/2}-b_r}{a_r} \Big|.
		\end{align*}
		This  implies the assertion by \eqref{eq:zrnu} and (\ref{eq:rvscale}).
	\end{proof}

	\subsection{Asymptotic covariances for empirical moments of block maxima} 
	\label{subsec:asycov}
	\medskip

	\begin{lemma}\label{lem:asymcovsl} 
		Consider sampling scheme (S1) from Condition~\ref{cond:obs} with $\gamma < 1/2$ and suppose further that Conditions \ref{cond:rl} and \ref{cond:unifint} hold.
		Then, for $g, g' \in \Gc$ with $\Gc$ as defined in (\ref{eq:defGc}),
		\[ 
		\lim_{n \to \infty} \Cov(\GG_n^{(\slb)}g, \GG_n^{(\slb)}g') = 2 \int\limits_0^1 \Cov\left(g(Z_{1\xi}), g'(Z_{2\xi})\right) \diff\xi. 
		\]
		where $\GG_n^{(\slb)}$ is defined in (\ref{eq:ggslb}) and where $(Z_{1\xi}, Z_{2\xi}) \sim G_{\gamma, \xi } $ with $G_{\gamma, \xi}$ from (\ref{def:Ggamxi}). 
		The same result holds with $\mathcal G$ replaced by  $\mathcal G'=\{ \bm 1_{(-\infty,t]}: t \in \R\}$; in that case, one may dispense with Condition~\ref{cond:unifint}.
	\end{lemma}

	\begin{proof}
	We only give a proof for $g,g'\in \mathcal G$, as the case $g,g'\in\mathcal G'$ is similar but simpler.
		Without making further assumptions, the sequence $\ell_n $ that satisfies the condition from Lemma \ref{lem:jointsl} can be chosen as $\ell_n = \max\{s_n, \ip{r_n\sqrt{\alpha(s_n)}}\}$, where $s_n = \ip{\sqrt{r_n}}$ (see \cite{BucSeg18b}), so we can apply that Lemma. We proceed similar as in \cite{BucSeg18b}: for $h\in\{1, \dots, \ndb\}$, let $ I_h = \{ (h-1)r_n+1, \ldots , hr_n\}$ denote the set of indices which make up the $h$-th disjoint block of observations. For simplicity, assume $  n/r \in \N$. Then,
		\begin{alignat*}{2}
		\sum_{j=1}^{\nsb}g(Z_{r,j}) 
		= 
		g(Z_{r, \nsb}) + \sum_{h=1}^{\ndb-1} A_h, \qquad 
		\sum_{j=1}^{\nsb} g'(Z_{r,j}) 
		=   
		g'(Z_{r, \nsb}) +  \sum_{h=1}^{\ndb-1} B_h \end{alignat*}
		where 
		\[ 
		A_h = \sum_{s\in I_h} g(Z_{r,s}), \qquad  B_h = \sum_{s\in I_h} g'(Z_{r,s}).
		\]
		Note that, by stationarity, the sequences $(A_h)_h, (B_h)_h$ are stationary as well. By uniform integrability (Condition~\ref{cond:unifint}), the contribution of $g(Z_{r, n-r+1})$ and $g'(Z_{r, n-r+1})$ to the asymptotic covariance is negligible. Further, since \[
		\sqrt{\frac{n}r} \frac1{n-r+1}=\frac{1}{\sqrt{nr}}\{1+o(1)\},
		\]
		it is sufficient to show that
		\[
		v_n = 
		\frac1{nr} \Cov\Big(\sum_{h=1}^{\ndb -1} A_h, \sum_{j=1}^{\ndb -1} B_j\Big)
		\to
		2 \int\limits_0^1 \Cov\left(g(Z_{1\xi}), g'(Z_{2\xi})\right) \diff\xi =v. 
		\]
		For that purpose, write   	
		\begin{align}
		v_n
		&= \frac1{nr} \Big\{ (\ndb -1) \Cov(A_1,B_1) +  \sum_{h=1}^{\ndb-2 } \left( \ndb-1 -h \right) \left(\Cov(A_1, B_{1+h}) + \Cov(A_{1+h}, B_1)\right) \Big\} \nonumber \\
		&= \frac1{nr} \Big\{ (\ndb -1) \Cov(A_1,B_1) + \left(\ndb -2\right) \Cov(A_2, B_1+B_3) \Big\} \nonumber \\
		& \hspace{1cm} 
		+ \frac1{nr} \sum_{h=2}^{\ndb -2 } \left( \ndb -1 -h \right) \left\{ \Cov(A_1, B_{1+h}) + \Cov(A_{1+h}, B_1)\right\} \nonumber \\
		&= v_{n1} + v_{n2} + v_{n3} , \label{eq:vcov}
		\end{align} 
		where
		\begin{align*}
		v_{n1}&= \frac1{r^2} \Cov(A_2,B_1+ B_2+B_3),  \qquad
		v_{n2} = -\frac1{nr} \Cov(A_2,2B_1+ B_2+2B_3) \\
		v_{n3} &=  \frac1{nr}\left\{  \sum_{h=2}^{\ndb -2 } \left( \ndb -1 -h \right) \left\{\Cov(A_1, B_{1+h}) + \Cov(A_{1+h}, B_1)\right\} \right\}.
		\end{align*}
		Next, for $\xi\ge 0$,
		define 
		\[ 
		g_{n1}(\xi) := \Cov\left(g(Z_{r,1}), g'(Z_{r,1+\ip{r\xi}})\right), 
		\quad
		g_{n2}(\xi) := \Cov\left(g(Z_{r,1+\ip{r\xi}}), g'(Z_{r,1})\right), 
		\]
		such that, by stationarity,
		\begin{align*}
		\frac{1}{r^2} \Cov(A_2,B_2) 
		&= 
		\frac{1}{r^2} \sum_{s=1}^{r_n}\sum_{t=1}^{r_n} \Cov(g(Z_{r,s}), g'(Z_{r,t})) \\
		&= 
		\frac{1}{r} g_{n1}(0) + \frac{1}{r} \sum_{h=1}^{r-1} \left(1- \frac{h}{r} \right) \left\{ g_{n1}\left(\tfrac{h}{r}\right)  + g_{n2}\left(\tfrac{h}{r}\right) \right\}.
		\end{align*}
		Similarly, we obtain 
		\begin{align*}
		\frac{1}{r^2} \Cov(A_2,B_3) 
		&= 
		\frac{1}{r} \sum_{h=1}^{r-1} \frac{h}{r} g_{n1}\left( \tfrac{h}{r} \right) + \tfrac{1}{r} \sum_{h=r}^{2r -1} \left( 2- \frac{h}{r}\right) g_{n1}\left(\tfrac{h}{r} \right),\\
		\frac{1}{r^2} \Cov(A_2, B_1) 
		&= 
		\frac{1}{r} \sum_{h=1}^{r-1} \frac{h}{r} g_{n2}\left( \tfrac{h}{r} \right) + \frac{1}{r} \sum_{h=r}^{2r -1} \left( 2- \frac{h}{r}\right) g_{n2}\left(\tfrac{h}{r} \right).
		\end{align*}
		Combining the previous three equations, we get
		\begin{align*}
		v_{n1} 
		&=\frac{1}{r} \sum_{h=0}^{r_n-1} \left\{ g_{n1}\left(\tfrac{h}{r}\right) + g_{n2}\left(\tfrac{h}{r}\right) \right\} -\frac{1}{r}g_{n2}(0)  +
		\frac{1}{r} \sum_{h=r}^{2r-1}\left(2-\frac{h}{r}\right) \left\{ g_{n1}\left(\tfrac{h}{r}\right) + g_{n2}\left(\tfrac{h}{r}\right) \right\} \\ 
		&=
		\int_0^1 g_{n1}(\xi) + g_{n2}(\xi)\diff\xi  + R_n,
		\end{align*}
		where 
		\begin{align*}
		\abs{R_n} &\leq \frac{1}{r_n} \abs{g_{n2}(0)} + \int_{1}^2 \abs{2-\frac{\ip{r\xi}}{r}} \abs{g_{n1}(\xi) + g_{n2}(\xi)} \diff \xi \\
		&\leq \frac{1}{r_n} \abs{g_{n2}(0)} + 2 \int_{1}^2 \abs{g_{n1}(\xi) + g_{n2}(\xi)} \diff\xi.
		\end{align*}
		Now, weak convergence (Lemma~\ref{lem:jointsl}) and uniform integrability (Condition~\ref{cond:unifint}) implies that $
		\lim_{n\to\infty} g_{nj}(\xi) = \Cov(g(Z_{1\xi}), g'(Z_{2\xi}))$
		for $j\in\{1,2\}$ and $\xi\ge 0$; in particular, the limit is zero for $\xi>1$. As a consequence, by dominated convergence, $R_n=o(1)$ and then $\lim_{n\to\infty} v_{n1}=v$.
		
		It remains to prove that $v_{n2}$ and $v_{n3}$ in Equation \eqref{eq:vcov} converge to zero. It can be shown by similar arguments as for $v_{n1}$ that $v_{n2}=O(r/n)=o(1)$. Considering $v_{n3}$, we start by treating the sum over those summands for which $h\ge 3$.  Lemma 3.11 in \cite{DehPhi02} yields 
		\begin{align*}
		\abs{ \Cov(A_1, B_{1+h}) } &\leq 10 \norm{A_1}_{2+\nu} \norm{B_1}_{2+\nu} \alpha(\sigma(A_1), \sigma(B_{1+h}) )^{\frac{\nu}{2+\nu} } \\
		&\leq 10 r^2 \norm{g(Z_{r,1})}_{2+\nu} \norm{g'(Z_{r,1})}_{2+\nu} \alpha((h-2)r)^{\frac{\nu}{2+\nu} },
		\end{align*}
		where $\nu $
		is taken from Condition~\ref{cond:unifint}, so that the norms are uniformly bounded by some constant $C$. $\Cov(A_{1+h}, B_1) $ can be bounded in the same way, whence the sum over the summands with $h\ge 3$ in $v_{n3}$ may be bounded by
		\begin{align*}
		\frac{1}{r^2}  \sum_{h=3}^{\ndb-2} \abs{ \Cov(A_1, B_{1+h})} +\abs{ \Cov(A_{1+h}, B_1)  } 
		&\leq 20  C^2  \sum_{h=1}^{\ndb-4} \alpha(h r)^{\frac{\nu}{2+\nu} },
		\end{align*}
		which converges to zero by Condition~\ref{cond:rl}(ii) and the choice of $\nu$ in Condition~\ref{cond:unifint}. 
		The summand for $h=2$ can be written as
		\begin{align*}
		\int_2^3  \left(3-\frac{\ip{r\xi}}{r}\right)(g_{n1}(\xi) + g_{n2}(\xi))\diff \xi
		\end{align*}
		and this converges to 0 by the same arguments as used in the treatment of $R_n$. 
	\end{proof}
	
	\begin{lemma}\label{lem:asymcovsl_S2}
		Consider sampling scheme (S2) from Condition~\ref{cond:obs} with $\gamma < 1/2$ and suppose further that Conditions \ref{cond:rl} and \ref{cond:unifint} hold.
		Then, for $g, g' \in \Gc$ with $\Gc$ as defined in (\ref{eq:defGc}),
		\[ 
		\lim_{n \to \infty} \Cov(\GG_n^{(\slb)}g, \GG_n^{(\slb)}g') = 2 \int\limits_0^1 \Cov\left(g(Z_{1\xi}), g'(Z_{2\xi})\right) \diff\xi. 
		\]
		where $\GG_n^{(\slb)}$ is defined in (\ref{eq:ggslb}) and where $(Z_{1\xi}, Z_{2\xi}) \sim G_{\gamma, \xi } $ with $G_{\gamma, \xi}$ from (\ref{def:Ggamxi}). The same result holds with $\mathcal G$ replaced by  $\mathcal G'=\{ \bm 1_{(-\infty,t]}: t \in \R\}$; in that case, one may dispense with Condition~\ref{cond:unifint}.
	\end{lemma} 
	
	\begin{proof}
		As in the previous proof, we only consider the case $g,g'\in \mathcal G$.
		Let $g_{n}(\xi, \xi') = \Cov(g(Z_{r, 1+\ip{r\xi}}, g'(Z_{r, 1+\ip{r\xi'}})))$. With the same arguments and  notations as in the proof of Lemma~\ref{lem:asymcovsl} for sampling scheme (S1) we obtain that the leading term in the covariance under consideration is
		\begin{multline}
		\frac1{r^2} \left\{  \Cov(A_1, B_1) +  \Cov(A_1, B_2) + \Cov(A_2, B_1) \right\} \\
		= \int_0^1 \int_0^1 g_n(\xi, \xi') \diff\xi\diff\xi' + \int_0^1\int_1^2 g_n(\xi, \xi')\diff\xi\diff\xi' + \int_1^2\int_0^1 g_n(\xi, \xi')\diff\xi\diff\xi'. \label{eq:intcov_S2}
		\end{multline}
		Weak convergence (Lemma~\ref{lem:jointsl2}) and uniform integrability (Lemma~\ref{lem:unifint_s2}) implies that $g_n(\xi, \xi')$ converges to $\Cov(g(Z_{1,|\xi-\xi'|}), g'(Z_{2,|\xi-\xi'|}))$, where $(Z_{1,|\xi-\xi'|}, Z_{2,|\xi-\xi'|}) \sim G_{\gamma, \abs{\xi-\xi'}}$. By dominated convergence, the integrals in \eqref{eq:intcov_S2} converge as well, the limit being 
		\begin{align}\label{eq:limcov:sls2}
		&\int_0^1\int_0^1 \Cov(g(Z_{1,|\xi-\xi'|}), g'(Z_{2,|\xi-\xi'|})) \diff\xi\diff\xi' \\
		& \qquad + 2\int_0^1\int_1^{1+\xi'}  \Cov(g(Z_{1,|\xi-\xi'|}), g'(Z_{2,|\xi-\xi'|})) \diff\xi\diff\xi', \nonumber
		\end{align}
		where we have used symmetry and the fact that $Z_{1,|\xi-\xi'|}$ and $Z_{2,|\xi-\xi'|}$ are independent if $\abs{\xi - \xi'} >1$.
	
		 It remains to show that the last expression is equal to $2 \int_0^1 \Cov\left(g(Z_{1\xi}), g'(Z_{2\xi})\right) \diff\xi. $ For that purpose, note that the function $\xi \mapsto u(\xi) = \Cov\left(g(Z_{1\xi}), g'(Z_{2\xi})\right)$ is bounded by some constant independent of $ \xi$, as can be seen by applying the Cauchy-Schwarz inequality. 
			Therefore, $u(\cdot)$ is integrable on every closed interval $[a,b] \subset \R$ with
			$ \int_a^b u(\xi) \diff\xi = U(b) - U(a)$, where $U$ denotes an antiderivative of $u$. We need to show that \eqref{eq:limcov:sls2} may be written as $2\left\{ U(1) - U(0) \right\}$.
			By changing variables we obtain
			\begin{align*}
			\eqref{eq:limcov:sls2} &= 
			\int_0^1 \left\{
			\int_0^1  u(\abs{\xi - \xi'}) \diff \xi+ 2 \int_1^{1+\xi'} u(\abs{\xi -\xi'}) \diff \xi
			\right\} \diff \xi' \\
			&= \int_0^1\left\{ 
			\int_0^{\xi'} u(\xi'-\xi) \diff \xi + \int_{\xi'}^1 u(\xi - \xi') \diff \xi +2 \int_1^{1+\xi' }u(\xi-\xi') \diff \xi
			\right\} \diff \xi' \\
			&= \int_0^1\left\{ 
			\int_0^{\xi'} u(v) \diff v + \int_{0}^{1-\xi'} u(v) \diff v +2 \int_{1-\xi'}^{1}u(v) \diff v
			\right\} \diff \xi' \\
			&= \int_0^1\left\{ 
			\int_0^{\xi'} u(v) \diff v + \int_{1-\xi'}^{1}u(v) \diff v
			+ \int_{0}^{1} u(v) \diff v
			\right\} \diff \xi' \\
			&= \int_0^1 \left\{ 
			U(\xi' ) - U(0)  + U(1)- U(1-\xi') 
			\right\}\diff \xi'  + U(1) - U(0) \\
			&= 2 \left\{ U(1) - U(0) \right\},
			\end{align*}
			since $\int_0^1U(\xi') \diff \xi' = \int_0^1 U(1-\xi') \diff \xi'$.
	\end{proof}

	\begin{lemma}\label{lemma:Loewner_Cov}
		Let $\bm{\Lambda}^{(\mb)}$ be defined as in Theorem~\ref{theo:normallbl}. Then $\bm\Lambda^{(\slb)} \leq_L \bm\Lambda^{(\djb)}. 
		$
	\end{lemma}
	
	\begin{proof}
		By the definition of the Loewner-order, we have  to show that
		$\Var(c' \bm{Y}^{(\djb)}) \geq  \Var(c' \bm{Y}^{(\slb)})$
		for any $c\in\R^p$,
		where $\bm{Y}^{(\mb)}$ denotes the limit variable of Theorem \ref{theo:normallbl}.  Choosing an iid\ sequence satisfying the conditions from that theorem, we have    
		\[
		\Var(c' \bm{Y}^{(\mb)}) = \lim_{n\to\infty} \Var\left(\GG_{n}^{(\mb)} g\right), \quad \mb \in \{ \djb, \slb\},
		\]
		for some function $g\in\Gc$ (note that $\Gc$ is closed under taking linear combinations), see Lemma~\ref{lem:asymcovsl} for the case $\mb = \slb$. In view of the fact that $\bigl(g(Z_{r,i}^{(\slb)})\bigr)_i$ is $r_n$-dependent, the assertion follows from Lemma A.10 in \cite{ZouVolBuc21}.
	\end{proof}

	\subsection{Further auxiliary results} \label{subsec:faux}
	\medskip

	\begin{lemma}\label{lem:psiphi} Recall the definition of $A$ and $W_1$ in (\ref{def:setA}) and (\ref{eq:defw1}), respectively,  and let
		\begin{align*}
		\psi: A \times C_b(\R) \times W_1 \to \R, &\qquad (a,g,\mu) \mapsto \int_\R |y| g(y) a(y) \diff \mu(y), \\
		\phi: A \times C_b(\R) \times W_1 \to \R, &\qquad (a,g,\mu) \mapsto \int_\R y g(y) a(y) \diff \mu(y). 
		\end{align*}
		The maps $\psi$ and $\phi$ are continuous in every $(a,g,\mu) \in C_b(\R) \times C_b(\R) \times W_1$.
	\end{lemma}
	
	\begin{proof} We only consider $\phi$.
		For $(a,g,\mu) \in C_b(\R) \times C_b(\R) \times W_1$, let $(a_n,g_n, \mu_n)_n \subset A \times C_b(\R) \times W_1$ such that $\lim_{n\to\infty}(a_n,g_n,\mu_n)= (a,g,\mu)$, i.e., 
		$
		\lim_{n \to \infty}\norm{a_n -a }_\infty  = 0$, $\lim_{n \to \infty}\norm{g_n -g }_\infty  = 0$ and $
		\lim_{n \to \infty}d_{W_1}(\mu_n, \mu)= 0. $
		Then,
		\begin{align*}
		& \phantom{{}={}} 
		|\phi(a_n,g_n,\mu_n) - \phi(a,g,\mu)| \\
		&\le
		\abs{\phi(a_n,g_n,\mu_n) - \phi(a,g_n,\mu_n)} + 	\abs{ \phi(a,g_n,\mu_n) - \phi(a,g,\mu_n)} \\
		& \hspace{3cm}+ \abs{ \phi(a,g,\mu_n) - \phi(a,g,\mu)}  \\
		&= 
		\abs{\int_\R y g_n(y) \{ a_n(y)-a(y) \} \diff \mu_n(y) } +   \abs{ \int_\R y \{g_n(y) - g(y))\} a(y) \diff \mu_n(y) } \\
		& \hspace{3cm}+ \abs{ \int_\R yg(y)a(y)\diff \{\mu_n(y) - \mu(y) \}} \\
		&\leq 
		\{ \norm{a_n - a}_\infty \norm{g_n}_\infty +  \norm{g_n - g}_\infty \norm{a}_\infty \} \int_\R \abs{y} \diff\mu_n(y) +  \abs{ \int_\R \varphi_{a,g}(y) \diff \{\mu_n(y) - \mu(y) \} },
		\end{align*}
		where $\varphi_{a,g}(y) = yg(y)a(y)$. The first term on the right-hand side of the previous display converges to 0, since
		$d_{W_1}(\mu_n, \mu) \to 0$ implies $ \int_\R \abs{y}\diff\mu_n(y) \to  \int_\R \abs{y}\diff\mu(y) < \infty$ by (\ref{def:W1weakconv}). Since $a$ and $g$ are continuous, $\varphi_{a,g}$ is continuous as well and satisfies $\abs{\varphi_{a,g}(y)} \leq \norm{g}_\infty\norm{a}_\infty \abs{y} \leq \norm{g}_\infty\norm{a}_\infty(1+\abs{y})$. Hence, the second term converges to 0 by  (\ref{def:W1weakconv_alt}).
	\end{proof}

	\begin{lemma}[Wasserstein consistency]  \label{lem:dw1h}
		Consider one of the sampling schemes from Condition~\ref{cond:obs} with $\gamma<1/2$. If Conditions~\ref{cond:rl}  and \ref{cond:unifint} are met,  then
		$
		d_{W_1}(\hat{H}_{r}^{\scs (\mb)}, G_\gamma) = o_\Prob(1)$.
	\end{lemma}
	
	\begin{proof}
		The result follows from application of  Lemma \ref{lem:rcw1}.		First of all, for every $n$, $\hat{H}_{r}^{(\mb)}$ is a discrete probability measure and hence an element of $W_1$.  		Next, note that
		\[
		\| \hat{H}_{r}^{(\mb)}  - G_\gamma \|_\infty  
		\le 
		\| \hat{H}^{(\mb)}_{r}  - \bar H_r \|_\infty 
		+
		\| \bar{H}_{r}  - G_\gamma \|_\infty  
		= O_\Prob((r/n)^{1/2}) + o(1) = o_\Prob(1)
		\]
		as a consequence of Theorem~\ref{theo:weakh}.
		It remains to be shown that 
		$
		M_n=\int |z| \diff \hat H_r^{(\scs \mb)}(z) = \E[\abs{Z}] + o_\PP(1),
		$
		where $Z \sim G_\gamma$. First,  $\Exp[M_n] \to \E[\abs{Z}]$ by Lemma~\ref{lem:momconv}. It thus suffices to show that $\Var(M_n) = o(1)$. This follows by the same arguments as for the treatment of $\Var(\YY_{n,k}^{\scs (\mb)} - \GG_n^{\scs (\mb)}f_{k,1})$ in the proof of Proposition~\ref{prop:expallbl}, see in particular (\ref{eq:ynkvar}) for disjoint blocks and (\ref{eq:varsummeah}) for sliding blocks.  
	\end{proof}

	\begin{lemma}\label{lem:rcw1}
		If, for a sequence  of random probability measures $(\hat{\mu}_n)_n, \ \hat{\mu}_n : \Omega \to W_1$ with distribution functions $\hat{F}_n$ and some $\mu \in W_1$ with continuous distribution function~$F$, the conditions
		\[ \| \hat{F}_n - F \|_\infty \xrightarrow{\PP} 0 \ \text{and} \ \int_\R \abs{x} \diff \hat{\mu}_n(x) \xrightarrow{\PP}  \int_\R \abs{x} \diff \mu(x) \]
		hold, then 
		\[ d_{W_1}(\hat{\mu}_n, \mu) \xrightarrow{\PP} 0. \] 
	\end{lemma}
	
	\begin{proof}
		Weak convergence of probability measures $\mu_n$ on the real line to a limit $\mu_0$ with continuous distribution function $F_{\mu_0}$ is well-known to be equivalent to uniform convergence of the respective distribution functions $F_{\mu_n}$. As a consequence, by (\ref{equiv:W1weak}),
		\begin{align*}
		d_{W_1}(\mu_n, \mu_0) \to 0 
		\quad \Leftrightarrow \quad 
		\| F_{\mu_n} - F_{\mu_0}\|_\infty + \abs{\int \abs{x} \diff \mu_n(x) - \int \abs{x}\diff \mu_0(x)} \to 0.
		\end{align*}
		The imposed  assumptions imply
		\[ 
		\| \hat F_{n} - F\|_\infty + \abs{\int \abs{x} \diff\hat{ \mu}_n(x) - \int \abs{x}\diff \mu(x)} \xrightarrow{\PP} 0. 
		\]
		The assertion then follows from standard arguments based on passing to almost surely convergent subsequences.    
	\end{proof}

	The following two lemmas are simple adaptations  of Lemma~A.7  and A.8 in \cite{BucSeg18a}.
	
	\begin{lemma}\label{lem:shortblock}
		Assume Condition~\ref{cond:mda} and let $M_{k}= \max(X_1, \dots, X_k)$. If $r\to\infty, r=o(n), \ell \to \infty, \ell = o(r)$ and $\alpha(\ell) = o(\ell/r)$ for $n\to\infty$, then, for all $y\in S_\gamma$, 
		\[ 
		\PP( M_{\ell} \geq a_{r}y + b_{r}) = O(\ell/r), \qquad n \to \infty. 
		\]
	\end{lemma}
	
	\begin{proof}
		From \cite{BucSeg14}, Lemma 7.1 we know that, for all $u >0$,
		\begin{align*} 
		\PP(F_{r}(M_{\ell}) > u ) = O(\ell/r) , \qquad  n\to \infty. 
		\end{align*}
		Since for all $ y\in S_\gamma $ we have 
		$ \lim_{n \to \infty} F_{r}(a_{r}y + b_{r})  = G_\gamma(y)$, we have $F_{r}(a_{r}y + b_{r}) > G_\gamma(y)/2 >0$ for sufficiently large $n$. Hence, the previous display implies
		\begin{align*}
		P(M_{\ell} \geq a_{r}y + b_{r})  &\leq \PP(F_{r}(M_{\ell}) \geq F_{r}(a_{r}y + b_{r} ))\\ 
		&\leq \PP(F_{r}(M_{\ell}) > G_\gamma(y)/2)  
		= O(\ell/r). &&   \qedhere
		\end{align*}    
	\end{proof}

	\begin{lemma}\label{lem:maxinbig}
		Under the same conditions as in Lemma \ref{lem:shortblock}, we have
		\[ 
		\lim_{n\to\infty} \PP(M_{r} > M_{r-\ell})=0.
		\]
	\end{lemma}
	\begin{proof}
		For  any $y \in S_\gamma$, we have 
		\begin{alignat*}{1}
		&\phantom{{}={}} 
		\PP(M_{r}> M_{r-\ell}) \\
		&= 
		\PP(M_{r}> M_{r-\ell}, M_{r-\ell} \leq a_{r}y+b_{r} )  
		+ \PP(M_{r}> M_{r-\ell}, M_{r-\ell}  >a_{r}y+b_{r} ) \\
		&\leq
		\PP( M_{r-\ell} \leq a_{r}y+b_{r} ) + \PP(\max\{X_{r-\ell+1}, \ldots, X_{r}\} > a_{r}y + b_{r}).
		\end{alignat*}
		The first summand converges to $G_\gamma(y)$ because of Condition~\ref{cond:mda}, invoking local uniform convergence in (\ref{eq:rvscale}). The second summand is equal to $\PP(M_{\ell} > a_{r}y+b_{r})$ by stationarity, which converges to 0 by Lemma \ref{lem:shortblock}. Now let $y \downarrow x_L$, the left endpoint of $G_\gamma$, to obtain the assertion. 
	\end{proof}

	\section{Exact formulas for the asymptotic covariance matrix}
	\label{sec:var}

	\begin{lemma}\label{lem:formelcov}
		Let $k, k' \in \{0,1,2\}$ and $\gamma < 1/2$. The asymptotic covariance from Theorem~\ref{theo:pwm1} can be written as
		\begin{align*}
		\bm{\Omega}_{k, k'}^{(\slb)} = 2 C_\gamma \int_0^{1/2} \frac{h_{\gamma ,k, k'}( w) + h_{\gamma,k', k}(w)}{ \{ w(1-w)\}^{1+\gamma}} \diff w, 
		\end{align*} 
		where 
		\begin{align*}
		C_\gamma = \begin{cases}
		\Gamma(2\abs{\gamma}), & \gamma <0, \\
		1, & \gamma =0, \\
		- \Gamma(1-2\gamma)/2\gamma, & \gamma> 0
		\end{cases}
		\end{align*}
		and, writing $c_{k,k'}(w)=kw+k'(1-w)$,
		\begin{align*}
		h_{\gamma, k, k'}(w) =  \frac{ \{c_{k,k'}(w)+1\}^{2\gamma+1} -\{c_{k,k'}(w)+1-w\}^{2\gamma +1 }}{w(2\gamma +1)}   - \{c_{k,k'}(w)+1\}^{2\gamma} 
		\end{align*}
		for $\gamma\notin \{0, -1/2\}$, 
		\begin{align*}
		h_{-\frac{1}{2}, k, k'}(w) =  \frac{ \log(c_{k,k'}(w)+1) -\log(c_{k,k'}(w)+1-w)}{w}   - \{c_{k,k'}(w)+1\}^{-1} 
		\end{align*}
		and  
		\[
		h_{0, k, k'}(w) 
		= 
		1- \frac{c_{k,k'}(w) + 1 -w}{w}\log\left( \frac{c_{k,k'}(w)+1}{c_{k,k'}(w)+ 1-w} \right).
		\]
		Moroever,
		\[
		\bm{\Omega}_{k, k'}^{(\djb)} = H_{k,k',\gamma} +  H_{k',k,\gamma} , 
		\]
		where
		\[
		H_{k,k',\gamma} = \int_0^1 u^{k-1}(1-u) (-\log u)^{-1-\gamma} \int_0^u s^{k'}(-\log s)^{-1-\gamma} \diff s \diff u . 
		\]
	\end{lemma}

	The above integrals cannot be solved explicitly, but can be approximated to an arbitrary precision based on numerical integration.

	\begin{proof} We only treat the sliding blocks case, the disjoint block case follows from a simple calculation and has, for instance, been worked out in \cite{FerDeh15}, see their formula (12).   
		
		Recall that, for fixed $k \in \N$,  we may write $f_k = f_{k,1}+ f_{k,2}$ with
		\[
		f_{k,1}(x) = xG_\gamma^k(x), \qquad f_{k,2}(x) = \int\limits_{x}^\infty y\, \nu_k'(G_\gamma(y)) \, \mathrm{d}G_\gamma(y)
		\]
		and $\nu_k(x) = x^k$. Therefore, supposing for the moment that $\gamma\ne 0$ and choosing $x\in S_\gamma$ and $k\in\N_{\ge 1}$, we may invoke  
		\begin{align*} 
		xG_\gamma^k(x)=x e^{-k(1+\gamma x)^{-\frac{1}{\gamma}}} &=  \int\limits_{-\infty}^x \frac{\mathrm{d}}{\mathrm{d}y} ye^{-k(1+\gamma y)^{-\dgam}}  \ind_{S_\gamma}(y) \, \mathrm{d}y \\
		& =  \int\limits_{-\infty}^x \{ 1+k y (1+\gamma y )^{-\dgam - 1} \} e^{-k (1+\gamma y)^{-\dgam}}  \ind_{S_\gamma}(y) \diff y,
		\end{align*}
		to obtain
		\begin{align*}
		f_k(x) &=xG_\gamma^k(x) + \int\limits_{x}^{\infty} k y e^{-(k-1)(1+ \gamma y )^{-\dgam}} (1+\gamma y)^{-\dgam -1} e^{-(1+\gamma y)^{-\dgam}} \ind_{ S_\gamma}(y) \, \mathrm{d}y \\
		&=  \int\limits_{-\infty}^x  e^{-k (1+\gamma y)^{-\dgam}}  \ind_{S_\gamma}(y) \, \mathrm{d}y
		+  \int\limits_{-\infty}^\infty k y (1+\gamma y )^{-\dgam - 1} e^{-k (1+\gamma y)^{-\dgam}} \ind_{S_\gamma}(y) \, \mathrm{d}y \\
		&= \int\limits_{(1+\gamma x ) ^{-\dgam}}^\infty e^{-kt} t^{-\gamma -1} \, \mathrm{d}t  + \ k \beta_{\gamma, k-1},
		\end{align*}
		where we made use of the substitution $ t= (1+\gamma y)^{-\dgam}, \ -t^{-\gamma -1} \mathrm{d}t = \mathrm{d}y$, and where $\beta_{\gamma, k}$ is defined in (\ref{eq:pwm}). Some thoughts reveal that the same equation is true for $\gamma=0$, i.e.,
		\begin{align*}
		f_k(x) = \int_{e^{-x}}^{\infty} e^{-kt}t^{-1} \diff t + k\beta_{0,k-1}, \quad k\in\N_{\ge 1}.
		\end{align*}
		Now, if $ Z \sim G_\gamma $, the transformation $ S= (1+\gamma Z)^{-\dgam}$ (with $S=\exp(-Z)$ for $\gamma=0$) is exponentially distributed with rate $\lambda = 1$. Indeed, for any $x>0$, we have:
		\begin{enumerate}
			\item Case $\gamma > 0$:
			\begin{align*}
			\PP(S \ge x) = \PP\left( (1+ \gamma Z)^{-\dgam} \ge x \right) = \PP\left(Z \le  \frac{x^{-\gamma} -1}{\gamma}\right) =  \exp(-x).
			\end{align*}
			\item Case $\gamma < 0$:
			\begin{align*}
			\PP( S \ge  x) = \PP\left( (1 - |\gamma| Z)^{\frac{1}{|\gamma|}} \ge x \right) 
			&= \PP\left( Z \le \frac{1-x^{|\gamma|}}{|\gamma|} \right) 
			= \exp(-x). 
			\end{align*}
			\item   Case $\gamma = 0$:     
			\begin{align*}
			\PP( S \ge  x) = \PP\left( \exp(-Z) \ge x   \right) &= \PP(Z \le -\log(x)) 
			= \exp(-x).
			\end{align*}
		\end{enumerate}  
		Therefore we may write 
		\begin{align}\label{darst:fkZ}
		f_k(Z) = \int\limits_{S}^{\infty} e^{-kt} t^{-\gamma -1 } \, \mathrm{d}t+  k \beta_{\gamma, k-1},
		\end{align}
		so that, when taking the expectation, Fubini's Theorem yields
		\begin{align}
		\E[f_k(Z)]  &= \E \left[ \int\limits_{0}^{\infty} \ind(S \leq t) e^{-kt} t^{-\gamma -1 } \, \mathrm{d}t\right] +  k \beta_{\gamma, k-1} \nonumber \\
		&= \int\limits_{0}^{\infty} \PP(S \leq t) e^{-kt} t^{-\gamma -1 } \, \mathrm{d}t +  k \beta_{\gamma, k-1} \nonumber \\
		&= \int\limits_0^\infty \left(e^{-kt} - e^{-(k+1)t}\right) t^{-\gamma -1} \, \mathrm{d}t + k \beta_{\gamma, k-1}. \label{erwwertfk}
		\end{align}
		To calculate integrals of that type, we distinguish cases according to the sign of~$\gamma$:
		\begin{enumerate} 
			\item Case  $ \gamma < 0 $. 
			For every $z > 0$, we have 
			\begin{align}
			\int_0^\infty e^{-zt} t^{- \gamma -1}\, \mathrm{d} t = z^{\gamma} \int_0^\infty e^{-s} s^{-\gamma -1} \, \mathrm{d}s = z^\gamma \Gamma( \abs{\gamma}).\label{wert:gammaintegralscaled:klnull} 
			\end{align}
			\item Case: $0 <\gamma < 1$. 
			First notice that, by partial integration and since $\gamma < 1$,
			\[ 
			\int\limits_0^\infty \left(1-e^{-s}\right)s^{-\gamma -1 } \, \mathrm{d}s = \frac{-s^{-\gamma}}{\gamma} (1-e^{-s}) \Big|_{s=0}^\infty+ \dgam \int\limits_{0}^\infty e^{-s} s^{-\gamma} \, \mathrm{d}s  = \frac{\Gamma(1-\gamma)}{\gamma} 
			\]
			and with the same substitution as in the first case we get for every $z > 0$
			\begin{align} \int\limits_0^\infty \left(1-e^{-zs}\right)s^{-\gamma -1 } \, \mathrm{d}s =  z^\gamma \frac{\Gamma(1-\gamma)}{\gamma}.\label{eq:integr:substGamma}
			\end{align}
			\item Case: $\gamma = 0$. 
			Recall the Exponential Integral 
			\begin{align*}
			\mathrm{E}_1(x) = \int\limits_{x}^{\infty} e^{-t}t^{-1} \diff t = - \gamma - \log(x) + \int_0^x \frac{1-e^{-t}}{t} \diff t.
			\end{align*}
			Then we may write, for $0<z_1< z_2 $,
			\begin{align}
			\int\limits_0^\infty \left( e^{-z_1t} - e^{-z_2t}\right)t^{-1} \diff t
			&= \lim\limits_{a\downarrow 0} \left\{ \int\limits_{az_1}^\infty e^{-t}t^{-1} \diff t - \int\limits_{az_2}^\infty e^{-t}t^{-1} \diff t \right\}   \nonumber \\
			&=  \lim\limits_{a\downarrow 0} \left\{\, \log(az_2) - \log(az_1) - \int\limits_{az_1}^{az_2} \frac{1-e^{-t}}{t}\diff t \right\} \nonumber \\
			&= \log\left(\frac{z_2}{z_1}\right). \label{eq:integr:gamma0}
			\end{align}
		\end{enumerate}  
		
		Next, for $(Z_{1\xi}, Z_{2 \xi}) \sim G_{\gamma, \xi} $ from Theorem \ref{theo:pwm1}, let
		\begin{align*}
		(S_{1 \xi}, S_{2\xi}) =  ((1+\gamma Z_{1\xi})^{-\frac{1}{\gamma}},(1+\gamma Z_{1\xi})^{-\frac{1}{\gamma}} )
		\end{align*}
		be the random vector arising from the transformation of the marginal distributions to standard exponentially distributed random variables (with $(S_{1 \xi}, S_{2\xi})=(\exp(-Z_{1\xi}), \exp(-Z_{2\xi}))$ in case $\gamma=0$). Note that, recalling $A_\xi(w) = \xi + (1- \xi) \{ w \vee (1-w) \}$, we have
		\[
		\Prob(S_{1 \xi} \le s, S_{2\xi} \le t) = 1- e^{-s}-e^{-t}+e^{-(s+t)A_\xi(\frac{t}{t+s})}, \qquad s,t>0,
		\]
		by a simple calculation.
		Invoking \eqref{darst:fkZ} and  \eqref{erwwertfk}, we get, for $k,k'\in \N_{\ge 1}$, 
		\begin{align}
		&\phantom{{}={}} \Cov (f_k(Z_{1\xi}), f_{k'}(Z_{2\xi})) \nonumber \\ 
		&= \E\left[ \int_{S_{1\xi}}^\infty \int_{S_{2\xi}}^\infty e^{-kt-k's} (ts)^{-\gamma-1} \diff s \diff t \right] 
		\nonumber \\
		&\qquad \qquad  - \left( \int_0^\infty   \left( e^{-kt} - e^{-(k+1)t} \right) t^{-\gamma-1} \diff t \right) \left(\int_0^\infty   \left( e^{-k's} - e^{-(k'+1)s} \right) s^{-\gamma-1} \diff s \right) \nonumber  \\
		&= \int\limits_{(0,\infty)^2} \PP(S_{1\xi} \leq s, S_{2\xi}\leq t) e^{-kt-k's}(ts)^{-\gamma -1} \, \mathrm{d}(s,t) \nonumber \\
		&\hspace{5cm} - \int\limits_{(0,\infty)^2} (1-e^{-t})(1-e^{-s})e^{-kt-k's} (ts)^{-\gamma-1} \, \mathrm{d}(s,t)\nonumber \\
		&= \int\limits_{(0,\infty)^2} \left(e^{-(t+s) A_\xi(\frac{t}{t+s})} - e^{-(t+s)}\right) e^{-kt-k's}(ts)^{-\gamma -1} \, \mathrm{d}(s,t) \nonumber 
		\\
		&= \int\limits_{0}^1 \int\limits_0^\infty  \frac{e^{- u \{ A_\xi(w) + kw + k'(1-w) \}} - e^{-u \{ kw + k'(1-w) +1\}} }{\{ w(1-w) \}^{\gamma +1}} u^{-2\gamma -1}  \, \mathrm{d}u \mathrm{d}w,  \label{rechn:covfkfl}
		\end{align}
		where we used the change of variables $  u = t+s, \ w = t/(t+s)$, i.e., $ uw = t, \ u(1-w) =s$ with Jacobian determinant $u$.
		As the function $f_k$ is the identity for $k = 0$, applying Hoeffding's formula for $ k = k' = 0$ yields 
		\begin{align*}
		\Cov(f_0(Z_{1,\xi}), f_0(Z_{2,\xi})) &= \Cov(Z_{1,\xi}, Z_{2,\xi}) \\
		&= \int\limits_{\R} \left(\PP(Z_{1\xi} \geq x, Z_{2\xi} \geq y) - \PP(Z_{1\xi} \geq x)  \PP(Z_{2\xi} \geq y) \right) \ind_{ S_\gamma^2}(x,y) \diff x(x,y) \\
		&= \int\limits_{0}^{\infty}\int\limits_{0}^{\infty}  \left(e^{-(t+s) A_\xi(\frac{t}{t+s})} - e^{-(t+s)}\right) (ts)^{-\gamma -1} \diff s \diff t,
		\end{align*}
		which implies \eqref{rechn:covfkfl} with $k = k' = 0$. Finally,
		for the case $k \in  \N_{\ge 1} 1$ and $k' = 0$, we may apply a generalized Hoeffding formula (\citealp{lo_hoeffding}, Theorem 3.1), which yields, with 
		$$f_k'(x)
		= \frac{d}{dx}\left\{ \int\limits_{(1+\gamma x)^{-\frac{1}{\gamma}}}^{\infty}e^{-kt}t^{-\gamma -1}\diff t + k\beta_{\gamma, k-1}\right\}= e^{-k(1+\gamma x)^{-\frac{1}{\gamma}}}, $$
		(defined as $e^{-e^{-x}}$ for $\gamma = 0$) that \eqref{rechn:covfkfl} is also valid if only one of $k,k'$ equals 0. As a summary, the equation holds  for all $k,k' \in \N_{\ge 0}$.
		
		We proceed by first restricting attention to the case $\gamma <0$. By \eqref{wert:gammaintegralscaled:klnull}, for every $z >0$,
		\begin{align*} \int\limits_{0}^\infty e^{-uz} u^{-2\gamma -1} \, \mathrm{d}u = \Gamma(2\vert\gamma\vert) z^{2\gamma}, 
		\end{align*}
		so that \eqref{rechn:covfkfl} equals
		\begin{align*} 
		\Gamma(2\vert\gamma\vert) \int\limits_{0}^1 \frac{\left\{ kw + k'(1-w) + A_\xi(w)\right \}^{2 \gamma} - \{ kw+k'(1-w)+1 \}^{2\gamma}}{\{w(1-w)\}^{\gamma +1} } \, \mathrm{d}w.
		\end{align*}
		By symmetry and the definition of $A_\xi$, this expression may be written as
		\begin{align*} 
		\Gamma(2\vert\gamma\vert) \{ J_{\gamma,k,k'}(\xi)  + J_{\gamma,k',k}(\xi)  \}
		\end{align*}
		where 
		\[ 
		J_{\gamma,k,k'}(\xi) = \int\limits_{0}^{1/2}  \frac{ \left\{ (k + \xi)w +(k'+1)(1-w)\right\}^{2\gamma} - \left\{kw + k'(1-w)+1\right\}^{2\gamma}} {\{w(1-w)\}^{\gamma +1}}\, \mathrm{d}w,
		\]     
		As a summary,
		\begin{align*}
		\bm{\Omega}_{k, k'}^{(\slb)} 
		&=  
		2 \int_0^1 \Cov\left(f_k(Z_{1\xi}), f_{k'}(Z_{2\xi})\right) \diff\xi 
		= 2\Gamma(2\vert\gamma\vert)  \int_0^1J_{\gamma,k,k'}(\xi)  + J_{\gamma,k',k}(\xi) \diff \xi.
		\end{align*}
		Finally, note that, for $\beta \neq -1$ and $a \neq 0$,
		\begin{align*}
		\int_0^1 (a\xi +b)^\beta \, \mathrm{d}\xi = \frac{(a+b)^{\beta +1} -b^{\beta+1}}{(\beta +1)a},
		\end{align*}
		while for $\beta = -1$ and $a\neq 0$,
		\begin{align*}
		\int_0^1 (a\xi +b)^{-1} \, \mathrm{d}\xi = \frac{\log(a+b)- \log(b)}{a},
		\end{align*}
		which readily implies
		\begin{align*}
		&\int_0^1 J_{\gamma,k,k'}(\xi) \, \mathrm{d}\xi = \int_0^{1/2}
		\frac{h_{\gamma ,k, k'}( w) }{ \{ w(1-w)\}^{1+\gamma}} \diff w
		\end{align*}
		after changing the order of integration,
		and hence yields the asserted formula after assembling terms.

		For the case $\gamma >0 $, we add $\pm 1$ in the numerator in \eqref{rechn:covfkfl} and with \eqref{eq:integr:substGamma} we see that, 
		\[
		\Cov (f_k(Z_{1\xi}), f_{k'}(Z_{2\xi}))=  -\frac{\Gamma(1-2\gamma)}{2\gamma} \{ J_{\gamma,k,k'}(\xi)  + J_{\gamma,k',k}(\xi)  \}.
		\]
		The calculations for the case $\gamma<0$ imply the asserted formula.

		Finally, for the case $\gamma = 0$, \eqref{eq:integr:gamma0} 
		yields that the expression in \eqref{rechn:covfkfl} may be rewritten as
		\[
		\int\limits_0^1 \log\left( \frac{c_{k, k'}(w)+1}{A_\xi(w) + c_{k, k'}(w)} \right) \{ w(1-w)\}^{-1} \diff w = J_{0, k, k'}(\xi) + J_{0, k', k}(\xi), 
		\]
		where 
		\[
		J_{0, k, k'}(\xi) = \int\limits_0^{1/2}  \log\left(  \frac{c_{k, k'}(w)+1}{\xi w + c_{k, k'}(w) + 1-w}   \right)\{w(1-w)\}^{-1} \diff w.
		\]
		Since
		$$  \int\limits_0^1 \log \left(\frac{c}{\xi a + b}\right)\diff \xi
		=   1 + \frac{1}{a}\left( (a+b)\log\left(\frac{c}{a+b}\right) - b \log\left( \frac{c}{b}\right) \right) $$
		for $a,b,c>0 $, we get 
		$$ \int\limits_0^1 J_{0, k', k}(\xi) \diff \xi 
		= \int\limits_0^{1/2} \frac{h_{0, k, k'}(w)}{w(1-w)}\diff w,$$
		which implies the final formula. 
	\end{proof}
	
	\begin{lemma}\label{lem:JacobiC}
		The entries $(c_{j,k})_{j,k} $ of the Jacobian matrix $\bm{C}$ of $\phi$ from Corollary \ref{cor:pwm} are given by 
		\begin{alignat*}{3}
		&c_{11} = \left( \frac{3^\gamma-1}{2^\gamma -1} - 1\right)\tilde{c}_{\gamma,1} ,\qquad 
		&& c_{12} = -2\frac{3^\gamma-1}{2^\gamma -1}\tilde{c}_{\gamma,1}, \qquad
		&& c_{13} = 3\tilde{c}_{\gamma,1}, \\
		&c_{21}= -  \tilde{c}_{\gamma,2} + c_{11} \tilde{c}_{\gamma,3}, \qquad 
		&&c_{22} = 2  \tilde{c}_{\gamma,2} + c_{12} \tilde{c}_{\gamma,3},\qquad
		&& c_{23} = c_{1,3} \tilde{c}_{\gamma,3}, \\
		&c_{31} = 1+  \tilde{c}_{\gamma,4}c_{11} +  \tilde{c}_{\gamma,5}c_{21}, \qquad 
		&&c_{32} = \tilde{c}_{\gamma,4} c_{12} +  \tilde{c}_{\gamma,5} c_{22}, \qquad
		&& c_{33} =  \tilde{c}_{\gamma,4} c_{13} +  \tilde{c}_{\gamma,5} c_{23},
		\end{alignat*}
		where 
		\begin{align*}
		\tilde{c}_{\gamma,1}  &= \frac{\gamma}{\Gamma(1-\gamma)(2^\gamma -1)}
		\left\{ \frac{3^\gamma\log(3)}{2^\gamma -1} - \frac{2^\gamma (3^\gamma -1)\log(2)}{(2^\gamma -1)^2} \right\}^{-1}, \\
		\tilde{c}_{\gamma,2}  &= \frac{\gamma}{\Gamma(1-\gamma)(2^\gamma -1)}, \\
		\tilde{c}_{\gamma,3} &= \frac{1}{\gamma} - \frac{2^\gamma \log(2)}{2^\gamma -1} +  \frac{\Gamma'(1-\gamma)}{\Gamma(1-\gamma)},\\
		\tilde{c}_{\gamma,4} &= \frac{1}{\gamma^2}\left\{ \gamma \Gamma'(1-\gamma) - 1+ \Gamma(1-\gamma) \right\} , \\
		\tilde{c}_{\gamma,5} &= \frac{1}{\gamma}\left(1-\Gamma(1-\gamma)\right).
		\end{align*}
		For $\gamma = 0$, the expressions are interpreted as the limit for $\gamma \to 0$, yielding 
		\begin{alignat*}{3}
		&c_{11} = \left( \frac{1}{\log(2)} - \frac{1}{\log(3)}\right)\tilde{c}_{0,1} ,\quad
		&&c_{12} = -\frac{2}{\log(2)}\tilde{c}_{0,1}, \quad
		&& c_{13} = \frac{3}{\log(3)} \tilde{c}_{0,1}, 
		\\
		&c_{21}= -\frac{1}{\log(2)} + c_{11} \tilde{c}_{0,3}, \quad
		&&c_{22} = \frac{2}{\log(2)} + c_{12} \tilde{c}_{0,3},\quad
		&& c_{33} = c_{1,3} \tilde{c}_{0,3}, \\
		&c_{31} = 1+   \Gamma''(1) c_{11} +  \Gamma'(1)c_{21}, \quad
		&&c_{32} =  \Gamma''(1)  c_{12} +  \Gamma'(1) c_{22}, \quad
		&& c_{33} =   \Gamma''(1) c_{13} + \Gamma'(1) c_{23},
		\end{alignat*}
		where 
		\begin{align*}
		\tilde{c}_{0,1}  &= \left\{ \frac{\log(3)}{2} - \frac{\log(2)}{2}\right\}^{-1},  \quad
		\tilde{c}_{0,3} = \frac{\log(2)}{2} + \Gamma'(1).
		\end{align*}
	\end{lemma}

	\begin{proof}
		This follows from straightforward calculations.
	\end{proof}

	\section{Additional simulation results}~\label{app:sim}
	
	\subsection{Additional results for  fixed block size}\label{supp:sec:fixr}
	Additional results comparable to those in Figures~\ref{fig:sim_releff_ar} and \ref{fig:sim_rl} in a situation where $r=90$ is fixed can be found in Figure~\ref{fig:sim_rl_AR_S2} (estimation of $\RL(T,90)$ in the AR-GPD-model under sampling scheme (S2)),
	Figure~\ref{fig:sim_relMSE_CAR} (shape estimation in the CAR-GPD-model for both sampling schemes) as well as Figures~\ref{fig:sim_rl_CAR_S1} ($\RL(T,90)$-estimation in the CAR-GPD-model under sampling scheme (S1)) and \ref{fig:sim_rl_CAR_S2} ($\RL(T,90)$-estimation in the CAR-GPD-model under sampling scheme (S2)). 
	Results for the ARMAX-GPD-model are shown in Figure~\ref{fig:sim_ARMAX_S1S2} (shape estimation for both sampling schemes) and Figures~\ref{fig:sim_rl_ARMAX_S1}  and \ref{fig:sim_rl_ARMAX_S2} (estimation of $\RL(T,90)$ under sampling  scheme (S1) and (S2), respectively).
	
	Remarkably, for almost all dependence structures, when considering return level estimation in contrast to shape estimation, it is only the $\gamma=0.4$ case for which the sliding blocks version does not provide an improvement over the disjoint blocks counterpart. 
	\begin{figure}[t]	
		\centering
		\makebox{\includegraphics[width=0.9\textwidth]{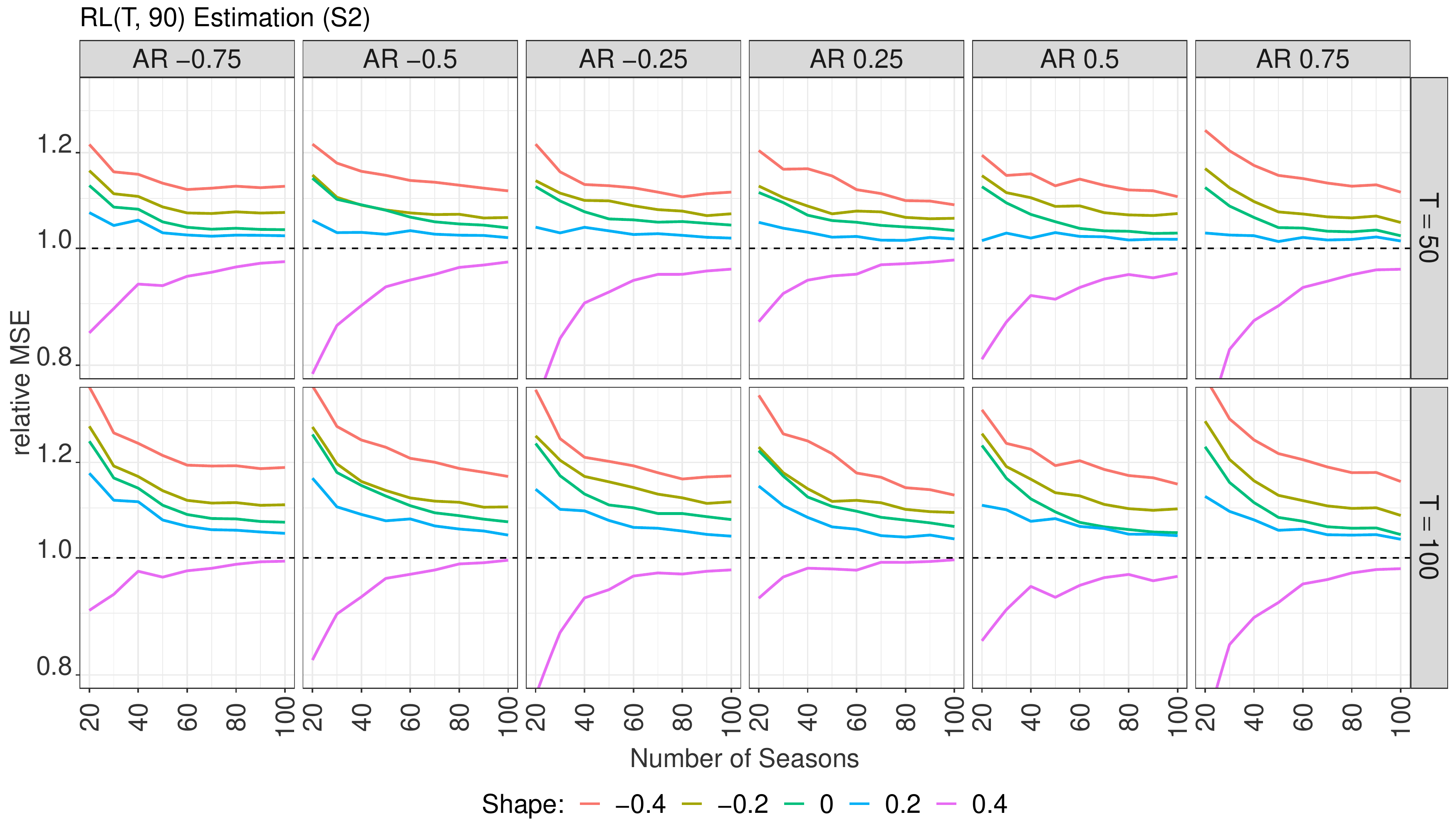}}
		\caption{ \label{fig:sim_rl_AR_S2}  
			Relative Efficiency (MSE of disjoint blocks estimator divided by MSE of sliding blocks estimator) of $\RL(T, r)$-estimation in a transformed AR(1) model with GPD-margins  under sampling scheme (S2)  for fixed block size $r=90$.}
	\end{figure}
	
	\begin{figure}[t]	
		\includegraphics[width=0.7\textwidth]{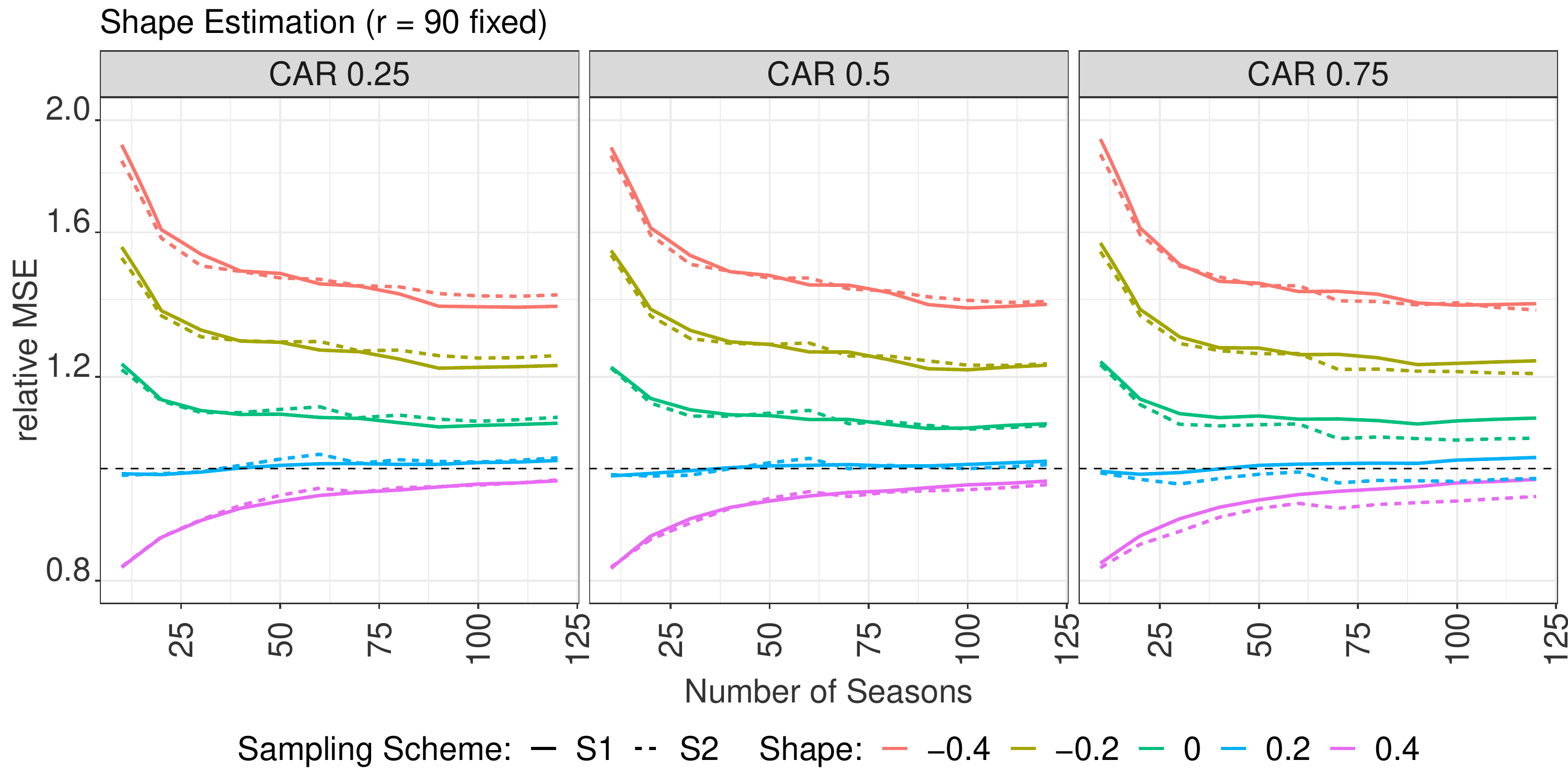}
		\caption{ \label{fig:sim_relMSE_CAR}  
				Relative Efficiency (MSE of disjoint blocks shape Estimation divided by MSE of sliding blocks shape Estimation) as a function of the number of seasons under sampling schemes (S1) and (S2) in CAR-GPD-models. }
	\end{figure}
	\begin{figure}[t]	
		\centering
		\includegraphics[width=0.7\textwidth]{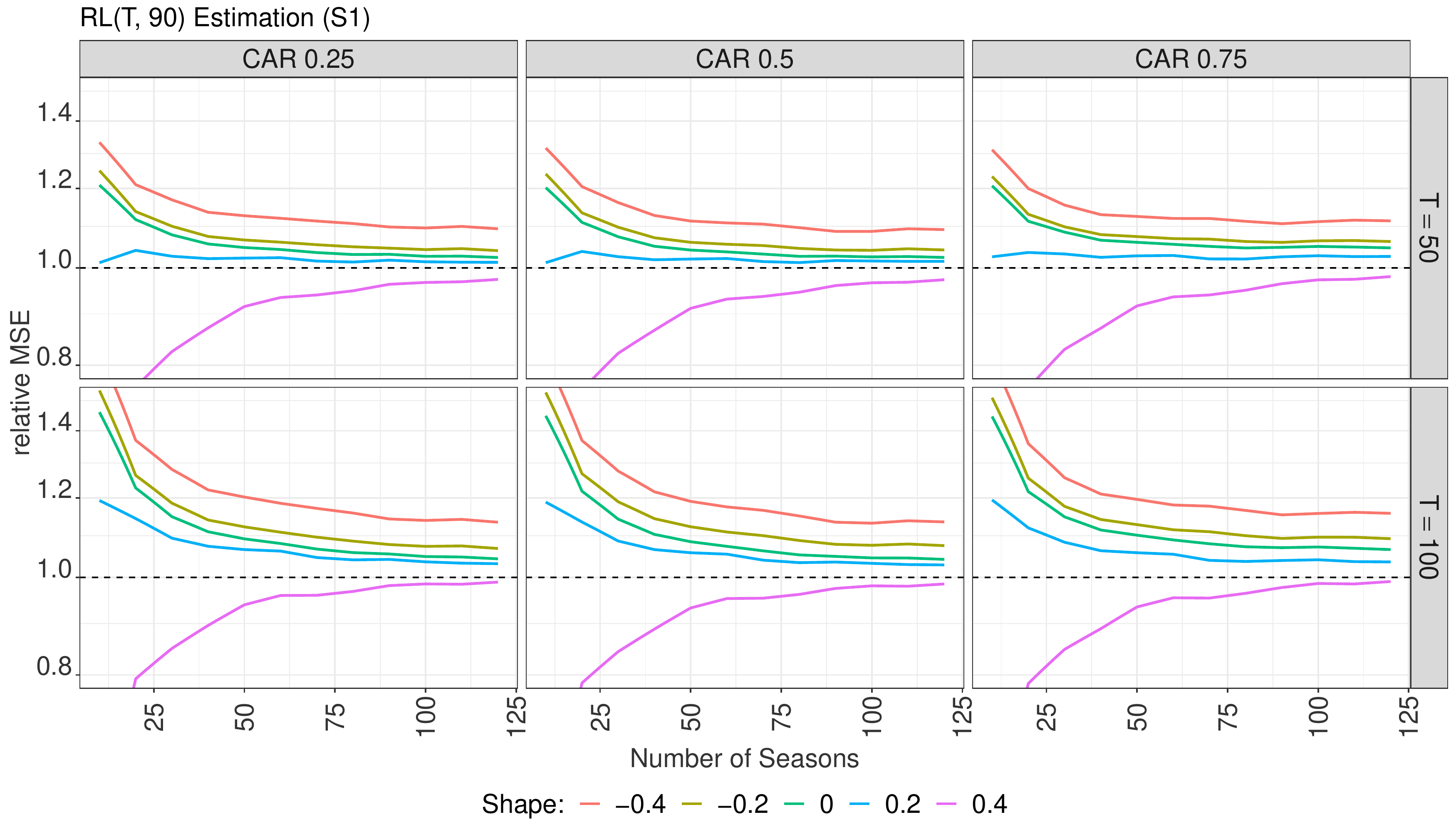}
		\caption{ \label{fig:sim_rl_CAR_S1}  
				Relative Efficiency (MSE of disjoint blocks estimator divided by MSE of sliding blocks estimator) for $\RL(T, 90)$-estimation as a function of the number of seasons under sampling scheme (S1) in CAR-GPD models.}
	\end{figure}
	
	\begin{figure}[t]	
		\includegraphics[width=0.7\textwidth]{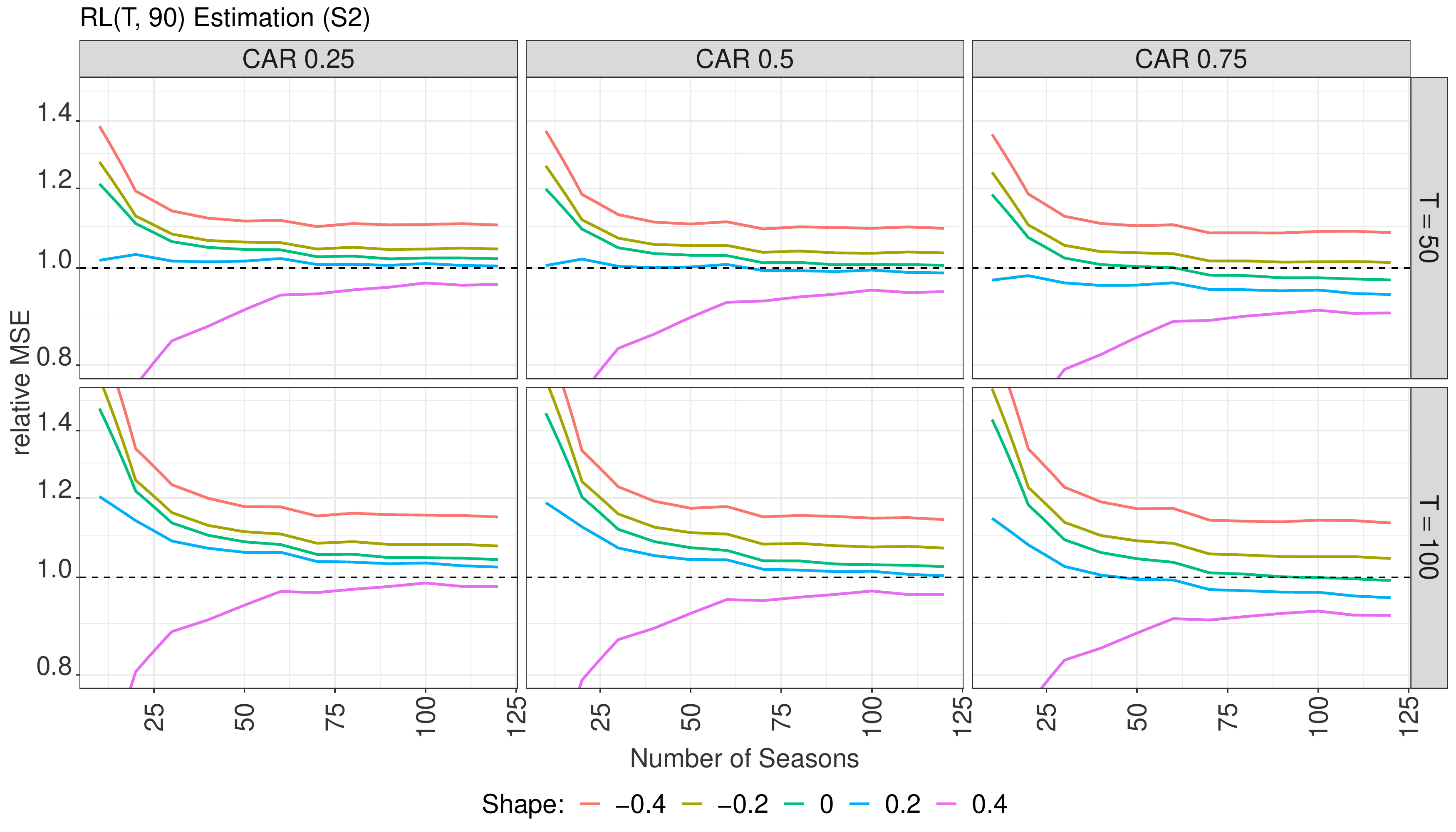}
		\caption{ \label{fig:sim_rl_CAR_S2}  
			Relative Efficiency (MSE of disjoint blocks estimator divided by MSE of sliding blocks estimator) for $\RL(T,90)$-estimation as a function of the number of seasons under sampling scheme  (S2) in CAR-GPD models. }
	\end{figure}
	
	\begin{figure}[t]	
		\centering
		\makebox{\includegraphics[width=0.7\textwidth]{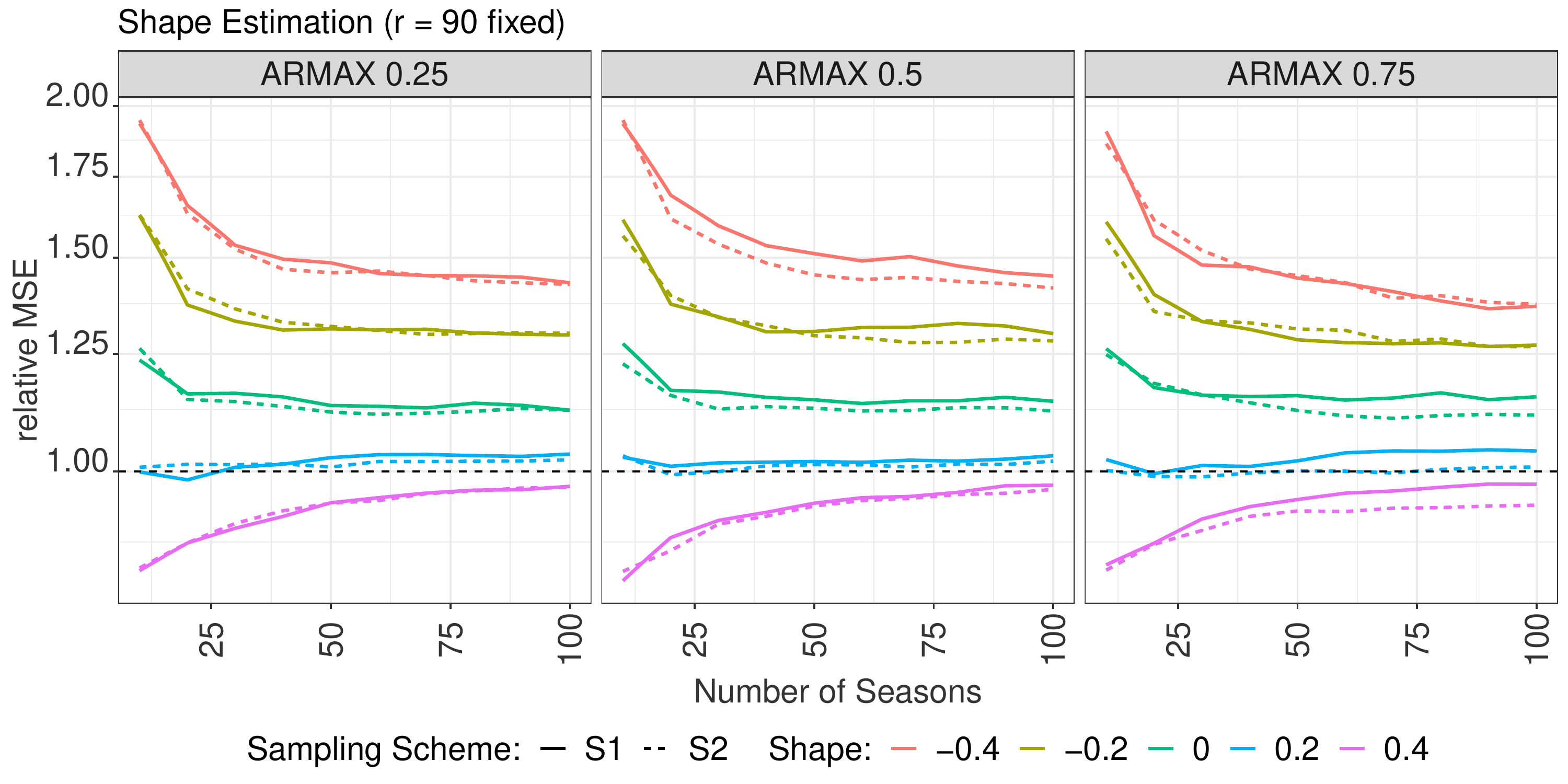}}
		\caption{ \label{fig:sim_ARMAX_S1S2}  
			Relative Efficiency (MSE of disjoint blocks estimator divided by MSE of sliding blocks estimator)  in a transformed ARMAX model with GPD-margins  under sampling schemes (S1)  and (S2) and for fixed block size $r=90$.}
	\end{figure}
	
	\begin{figure}[t]	
		\centering
		\includegraphics[width=0.7\textwidth]{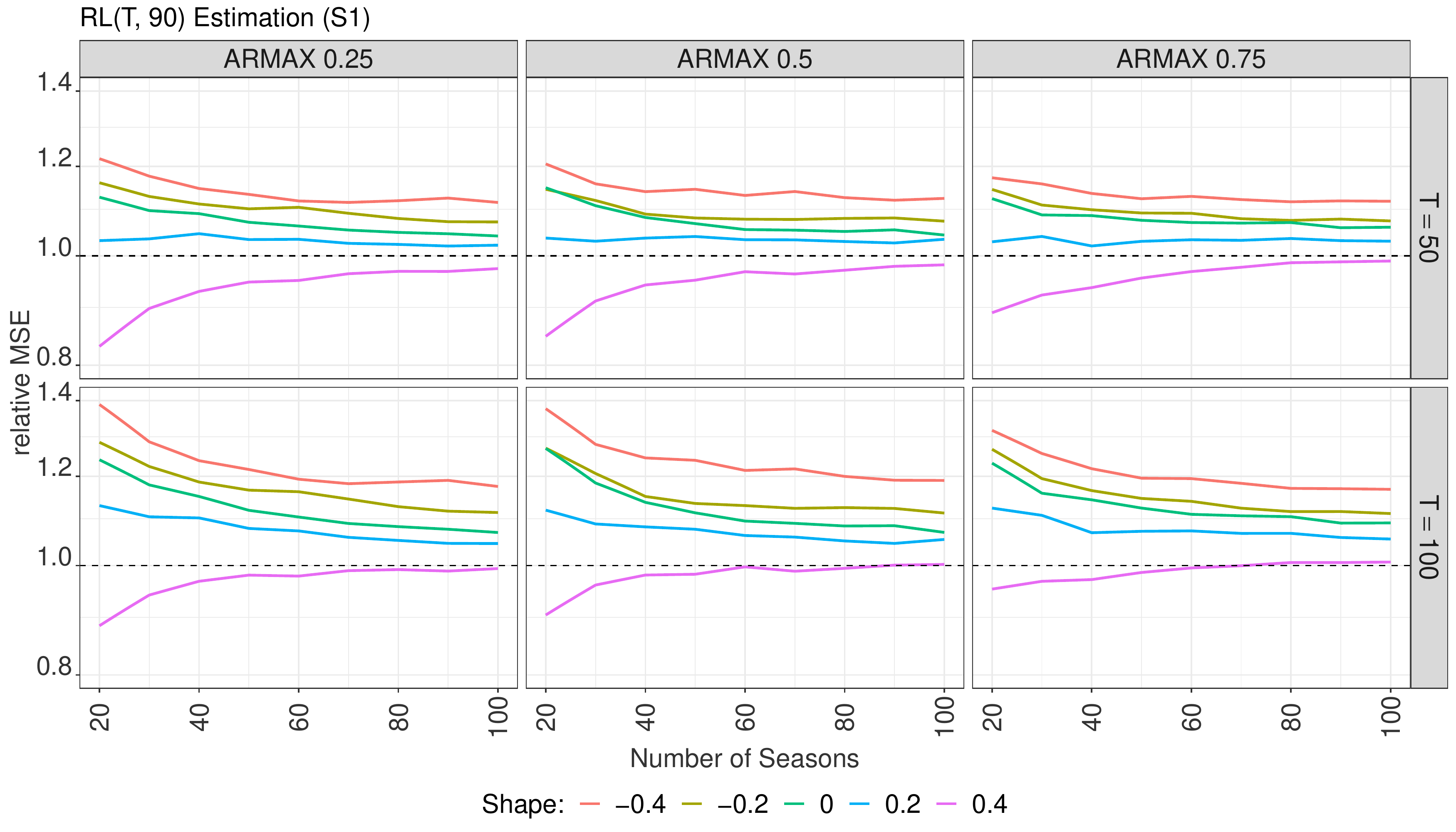}
		\caption{ \label{fig:sim_rl_ARMAX_S1}  
			Relative Efficiency (MSE of disjoint blocks estimator divided by MSE of sliding blocks estimator) for $\RL(T,90)$-estimation as a function of the number of seasons under sampling scheme (S1) in various ARMAX models.}
	\end{figure}
	\begin{figure}[t]	
		\includegraphics[width=0.7\textwidth]{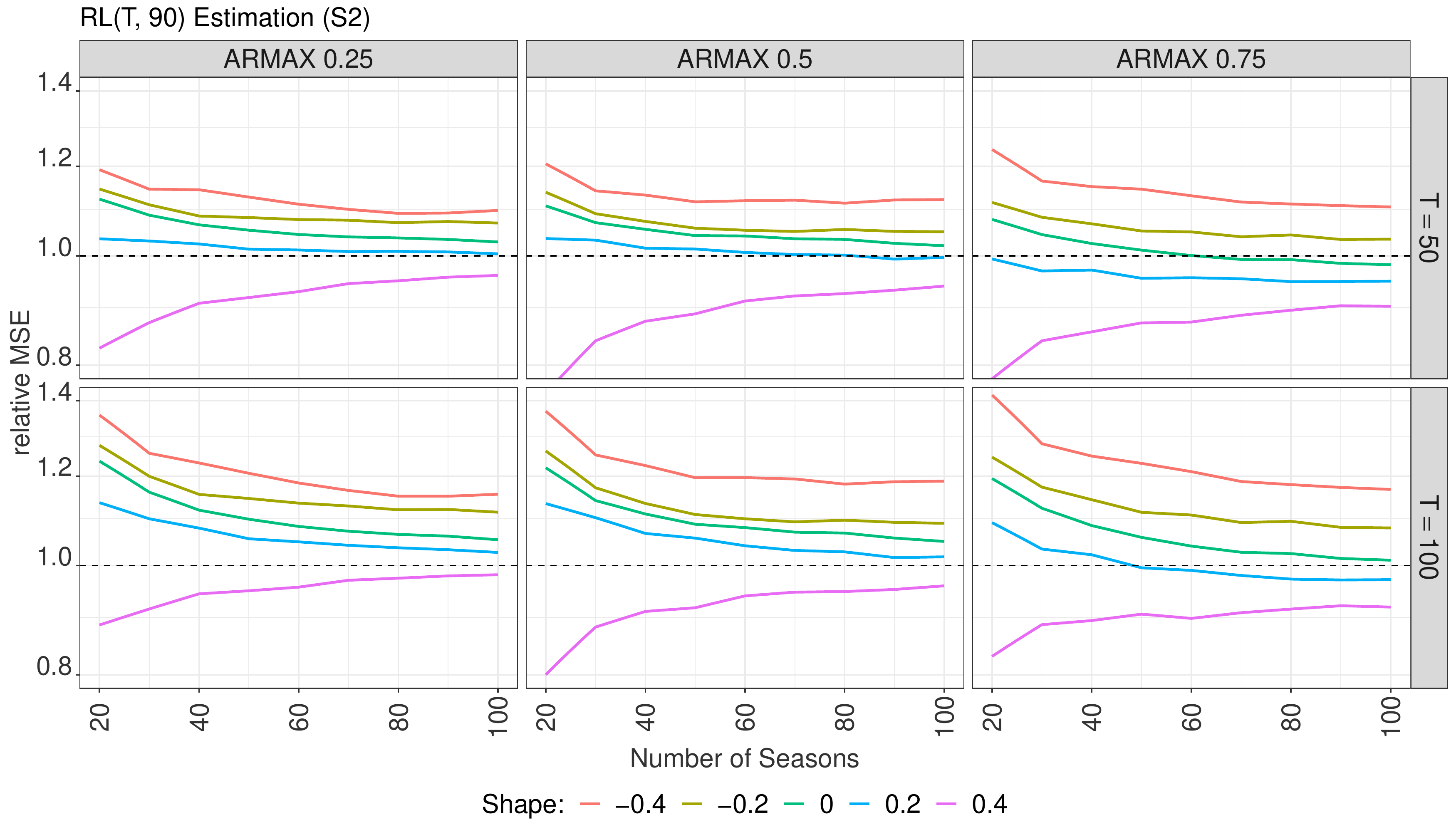}
		\caption{ \label{fig:sim_rl_ARMAX_S2}  
			Relative Efficiency (MSE of disjoint blocks estimator divided by MSE of sliding blocks estimator) for $\RL(T,90)$-estimation as a function of the number of seasons under sampling scheme (S2) in ARMAX-GPD-models.}
	\end{figure}
	
	\subsection{Additional results for  fixed sample size}\label{supp:sec:fixn}
	
	\begin{figure}[tbh!]	
		\centering
		\includegraphics[width=0.7\textwidth]{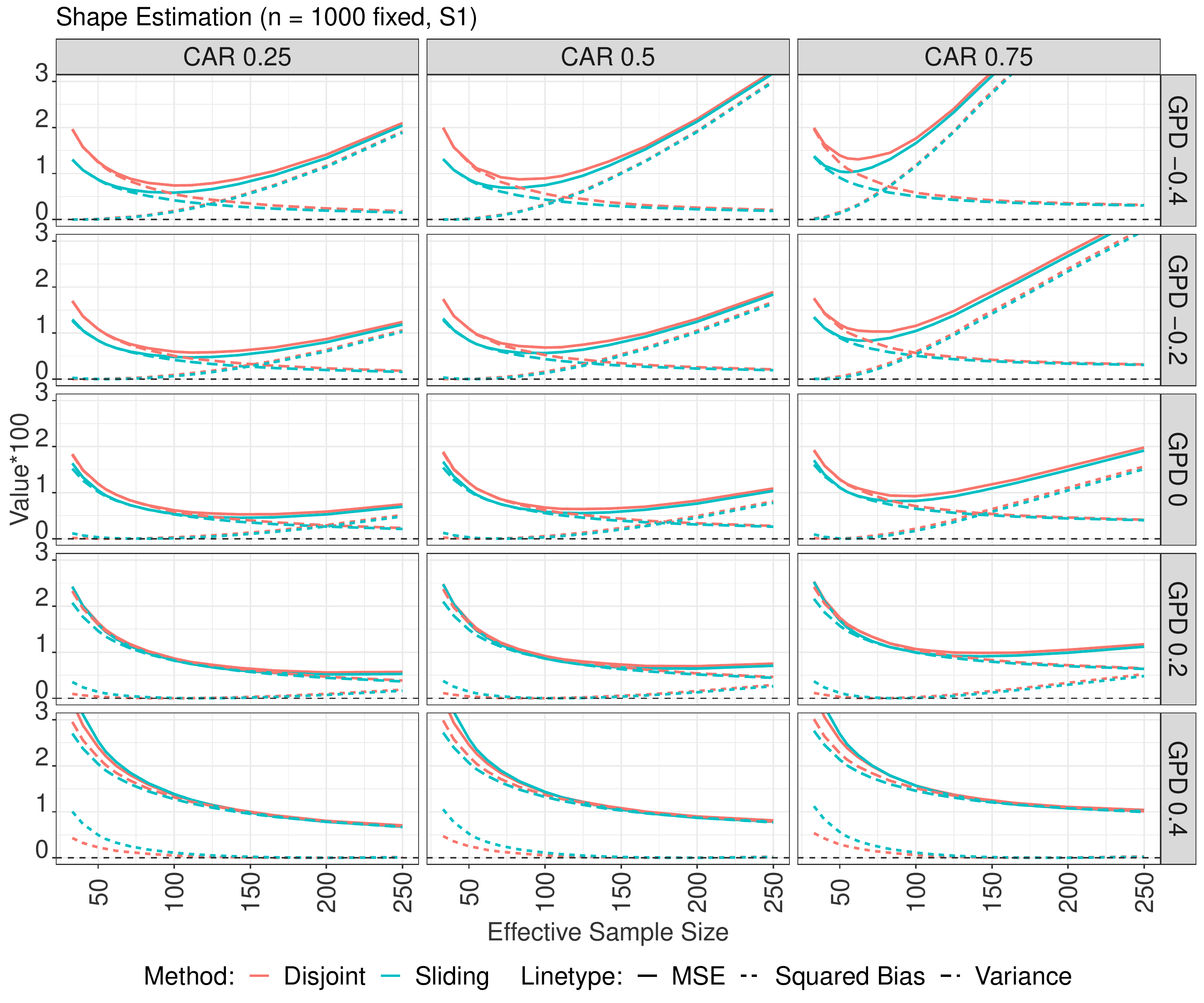}
		\caption{\label{fig:car_all_s1} 
			MSE, squared bias and variance for the estimation of the shape parameter $\gamma$ in a transformed Cauchy AR(1) model with GPD-margins under sampling scheme (S1) for fixed sample size $n=1000$.
		}
	\end{figure}
	\begin{figure}[t]	
		\includegraphics[width=0.7\textwidth]{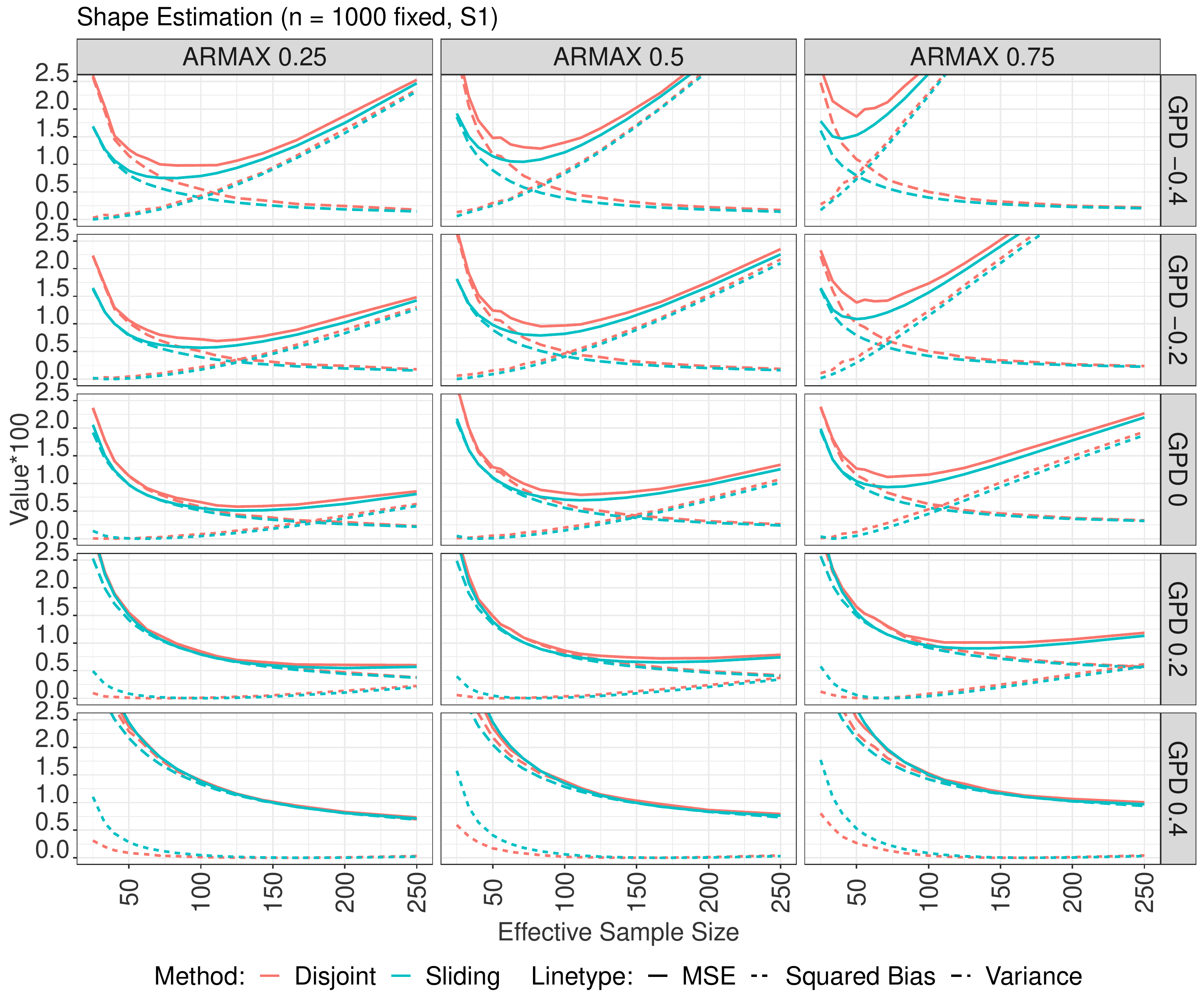}
		\caption{ \label{fig:sim_armax_all} 
			MSE, squared bias and variance for the estimation of the shape parameter $\gamma$ in a transformed ARMAX(1) model with GPD-margins under sampling scheme (S1) for fixed sample size $n=1000$.}
	\end{figure}

	Results for fixed sample size $n=1000$ for shape estimation in the CAR-GPD-model under sampling scheme (S1) can be found in Figure~\ref{fig:car_all_s1}, those for the
		ARMAX-GPD-model in Figure~\ref{fig:sim_armax_all} .
	The overall findings are similar as for the AR(1)-model depicted in Figure~\ref{fig:sim_ar_all}.
	
	\subsection{Results for comparing the plain and bias-reduced sliding blocks estimator}\label{supp:sec:slbr}
	As mentioned in the main paper, the bias-reduced sliding blocks estimator from Remark~\ref{rem:biasred} is computationally costly for situations involving overall sample sizes of up to $n=9000$. Therefore, when comparing results for fixed blocksize $r = 90$, we restrict attention to sampling scheme (S2) and a selection of 20 models that are made up of 4 different time series models (i.i.d., AR 0.5, CAR 0.5 and ARMAX 0.5) and the 5 different GPD-margins (GPD$(\gamma)$ with $\gamma \in \{ -0.4, -0.2,0, 0.2, 0.4\}$) .
		The bias and MSE of shape estimation as obtained for the disjoint, sliding and bias-reduced sliding blocks methods are shown in Figures \ref{fig:mixedts_bias_rfix} and \ref{fig:mixedts_mse_rfix}, respectively. The bias of the bias-reduced sliding blocks estimator matches the bias of the disjoint blocks estimator almost perfectly. For positive shape parameters that results in equal performance in terms of MSE (with a tiny advantage for the sliding version and small sample sizes) for those two estimators. For negative shapes, the plain sliding version still has the smallest MSE, which can be explained by its smaller variance.  
	
	For fixed samplesize, the sliding blocks estimator is compared with its bias-reduced version in Figure~\ref{fig:sim_ar_compare_sl_slbi} (shape estimation, AR-GPD-model, sampling scheme (S1), $n=1000$), Figure~\ref{fig:sim_car_compare_sl_slbi} (the same for the CAR-GPD-model) and  Figure~\ref{fig:sim_armax_compare_sl_slbi} (the same for the ARMAX-GPD-model). Considering only the squared bias, it can be seen that the bias-reduced version may outperform its counterpart for small block sizes, in particular in scenarios involving non-negative shapes and positive AR parameters. However, in CAR and ARMAX scenarios as well as AR scenarios with negative parameters, the bias-reduced estimator may also exhibit a uniformly larger squared bias. In terms of variance, the plain estimator mostly has a slight edge. Summarizing the findings is rather difficult, whence we tend to recommend the use of the plain version merely for computational reasons (in particular for non-negative shapes).

	\begin{figure}[t]	
		\centering
		\includegraphics[width=0.7\textwidth]{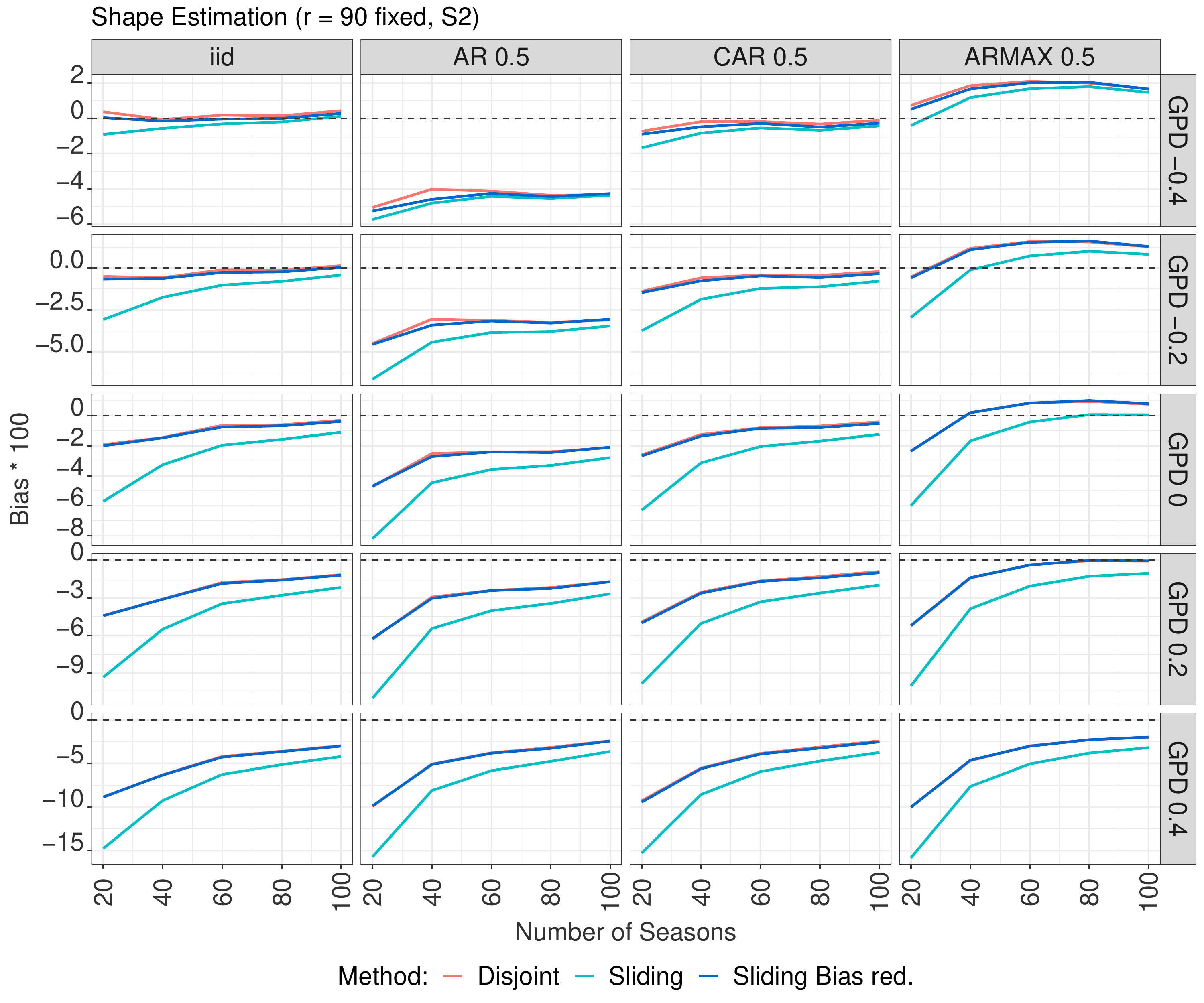}
		\caption{\label{fig:mixedts_bias_rfix} Bias of shape estimation in a selection of transformed time series models with GPD margins for fixed $r = 90$ and  growing number of seasons under sampling scheme (S2).}
	\end{figure}
	
	\begin{figure}[t]	
		\centering
		\includegraphics[width=0.7\textwidth]{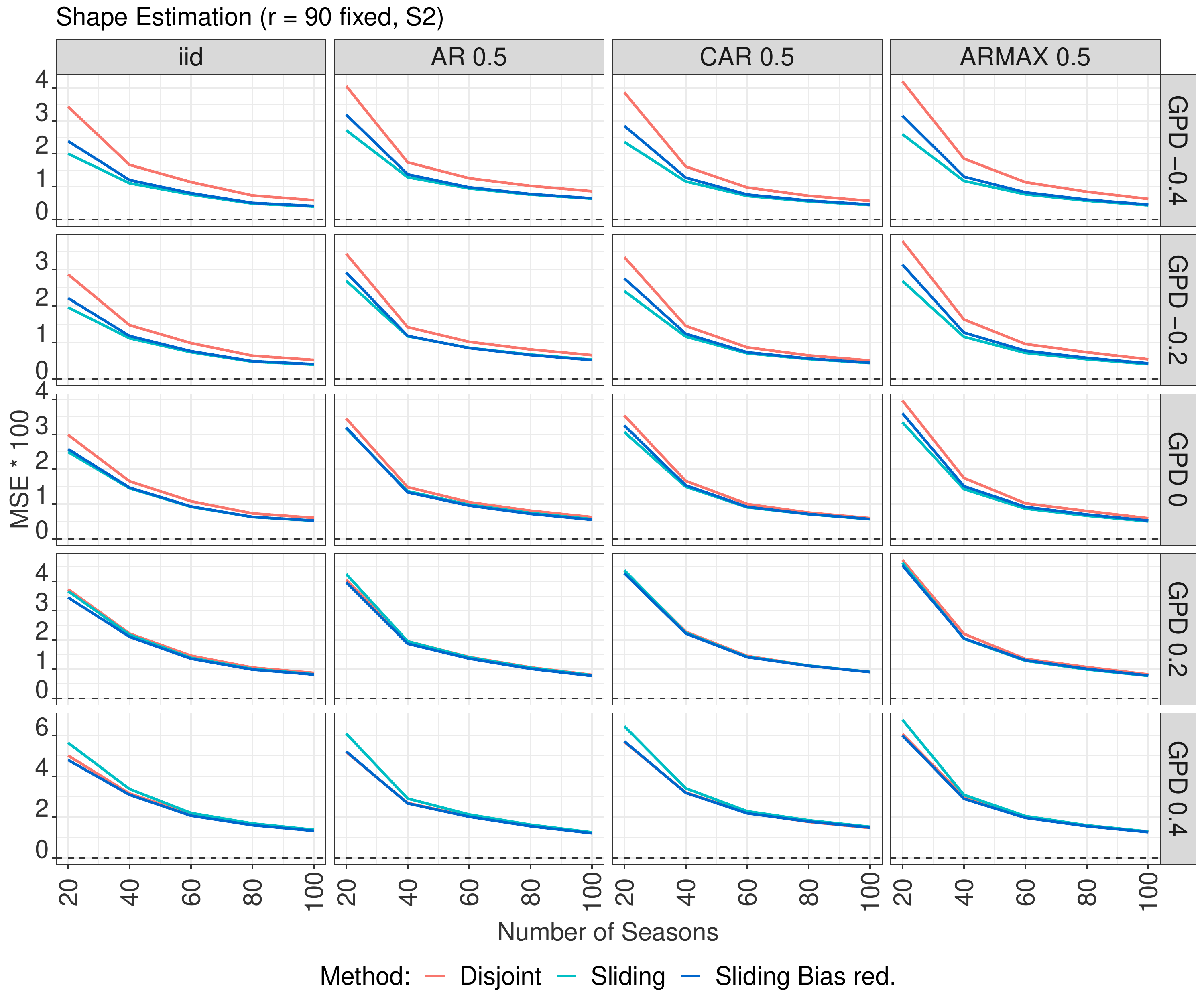}
		\caption{\label{fig:mixedts_mse_rfix} MSE of shape estimation  in a selection of transformed time series models with GPD margins for fixed $r = 90$ and  growing number of seasons under sampling scheme (S2).}
	\end{figure}
	\begin{figure}[t]	
		\centering
		\makebox{	\includegraphics[width=0.7\textwidth]{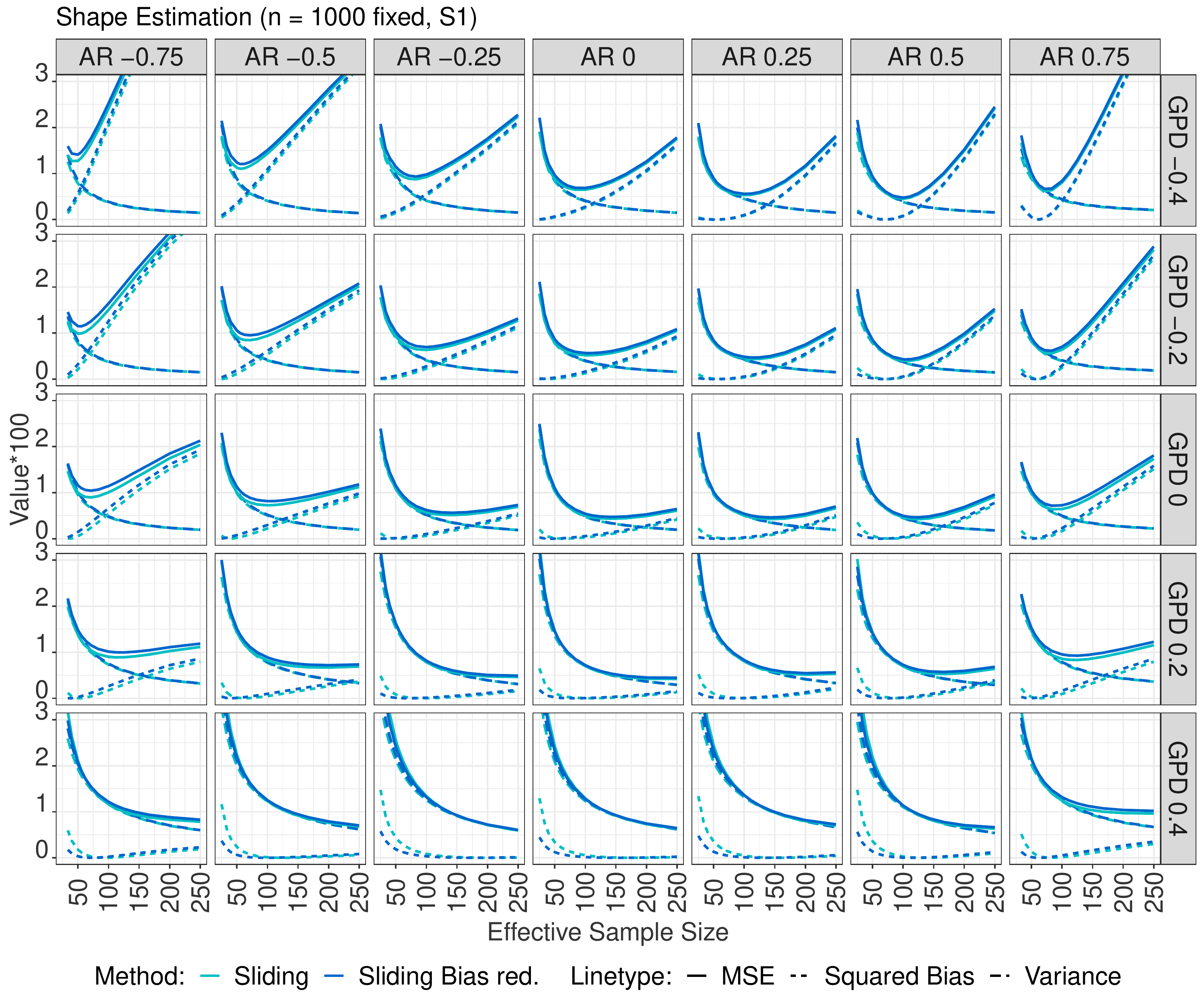}}
		\caption{ \label{fig:sim_ar_compare_sl_slbi} 
			Comparison of MSE, squared Bias and Variance for the plain and bias reduced sliding blocks estimators in the AR-GPD-models under sampling scheme (S1). }
	\end{figure}
	
	\begin{figure}[t]	
		\centering
		\makebox{\includegraphics[width=0.7\textwidth]{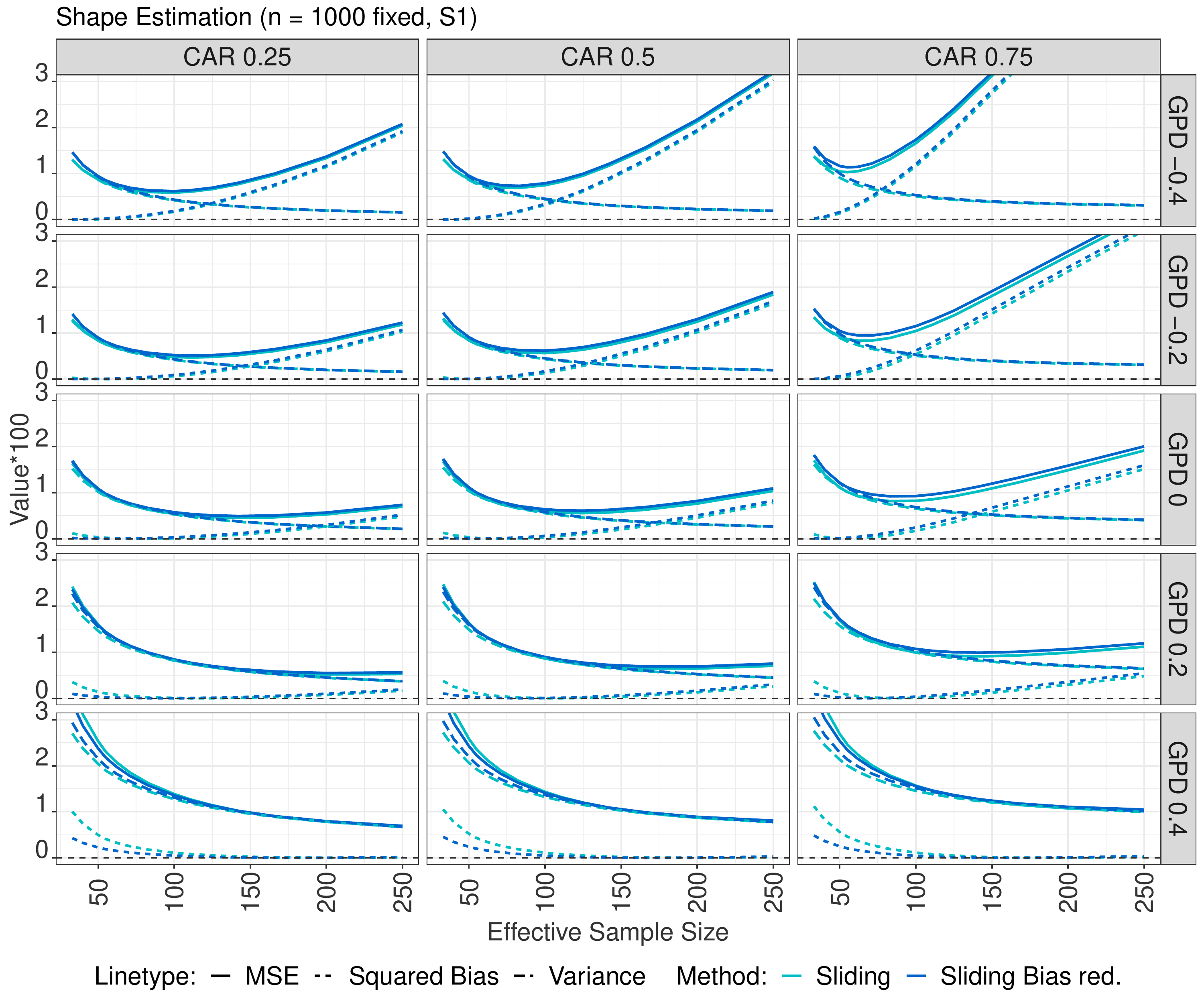}}
		\caption{ \label{fig:sim_car_compare_sl_slbi} 
			Comparison of MSE, squared Bias and Variance for the plain and bias reduced sliding blocks estimators in the CAR-GPD-models under sampling scheme (S1).}
	\end{figure}
	
	\begin{figure}[t]	
		\centering
		\makebox{\includegraphics[width=0.7\textwidth]{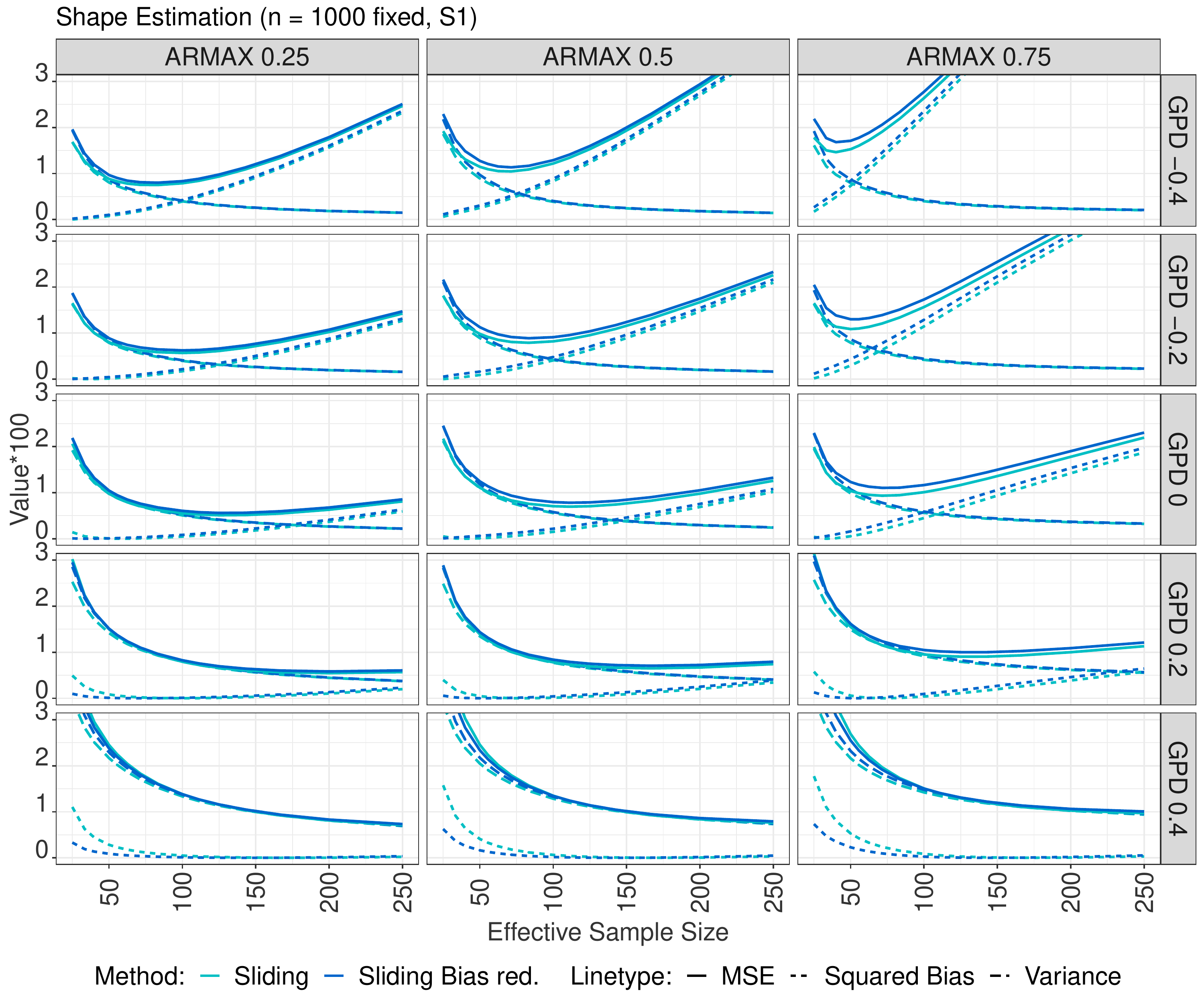}}
		\caption{ \label{fig:sim_armax_compare_sl_slbi} 
			Comparison of MSE, squared Bias and Variance for the plain and bias reduced sliding blocks estimators in the ARMAX-GPD-models under sampling scheme (S1).}
	\end{figure}

	\subsection{Results for comparing sampling schemes (S1) and (S2) for fixed sample size $n$}\label{supp:sec:s1s2} 
	Results for the comparison of sampling schemes (S1) and (S2) in situations of fixed sample size $n=1000$ can be found in  Figure~\ref{fig:sim_comp_s1_s2} (shape estimation within the AR-GPD-model), Figure~\ref{fig:sim_comp_s1_s2_car} (shape estimation within the CAR-GPD-model) and Figure~\ref{fig:sim_comp_s1_s2_armax} (shape estimation within the ARMAX-GPD-model).
	In most cases the behavior between the two sampling schemes is similar, as was to be expected from the theoretic results. A notable exception concerns high level of serial dependence, non-positive shape parameters and small block sizes, where the MSE for the sliding blocks estimator is slightly smaller in scenario (S2) than in (S1).  This difference may be explained by the fact that, heuristically, the non-constancy of $j\mapsto H_{r,j}$ (see Condition~\ref{cond:bias2}) is increasing in the strength of serial dependence and decreasing in the block size. As a consequence, 
	the bias $B_{n,k}^{\scs (\mb, \samp)}$ in Condition~\ref{cond:bias2} shows a similar behavior, eventually impacting the MSE in the observed way.

	\begin{figure}[t]	
		\centering
		\includegraphics[width=0.8\textwidth]{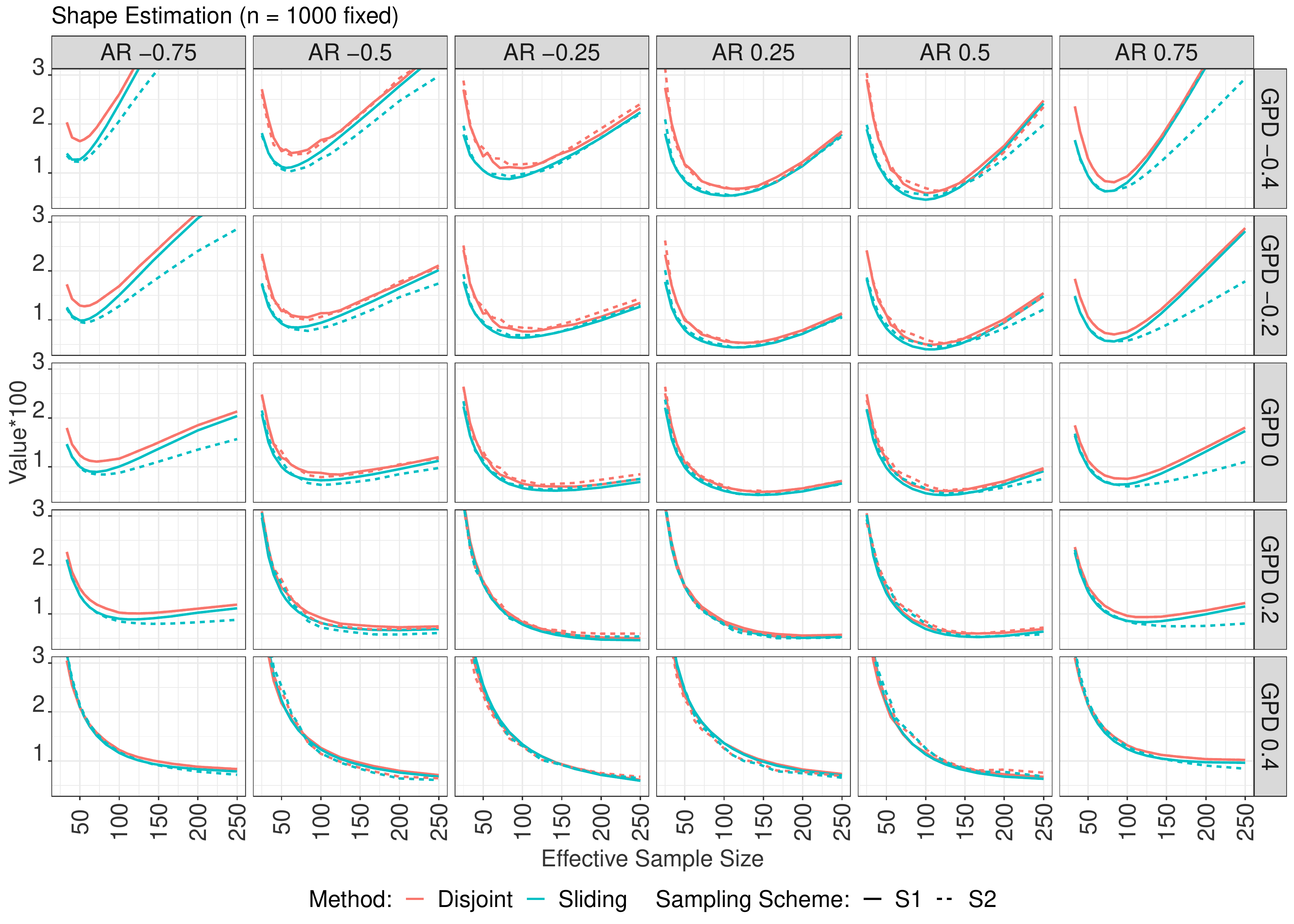}
		\caption{ \label{fig:sim_comp_s1_s2}  MSE of shape estimation for observations from sampling scheme (S1)  and (S2) based on a transformed AR(1) model with GPD-margins for fixed sample size $n=1000$.
		}
	\end{figure} 
	
	\begin{figure}[t]	
		\centering
		\includegraphics[width=0.7\textwidth]{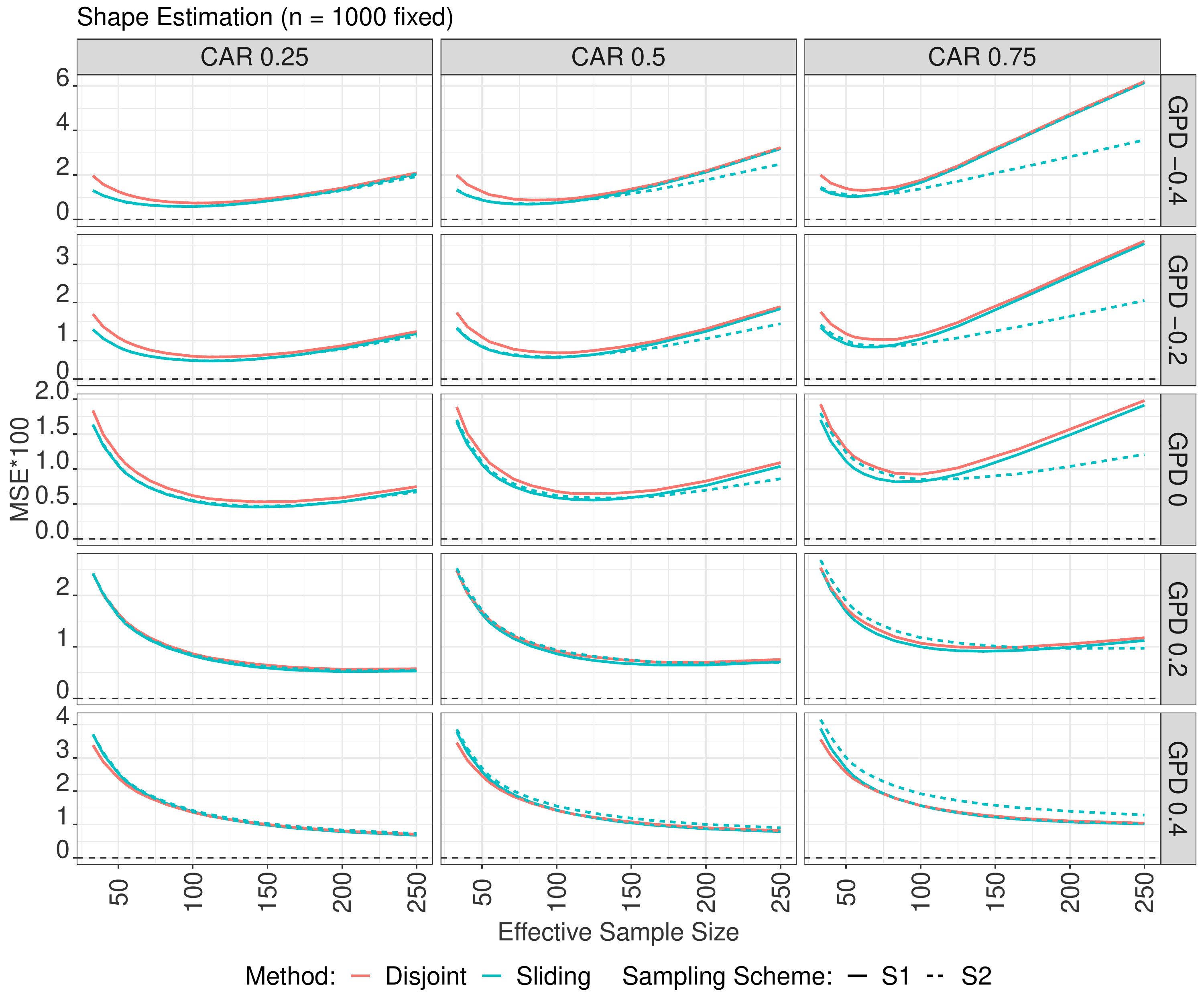}
		\caption{ \label{fig:sim_comp_s1_s2_car}   MSE of  shape estimation for observations from sampling scheme (S1)  and (S2) based on a transformed Cauchy AR(1) model with GPD-margins for fixed sample size $n=1000$.
		}
	\end{figure} 
	\begin{figure}[t]	
		\centering
		\includegraphics[width=0.7\textwidth]{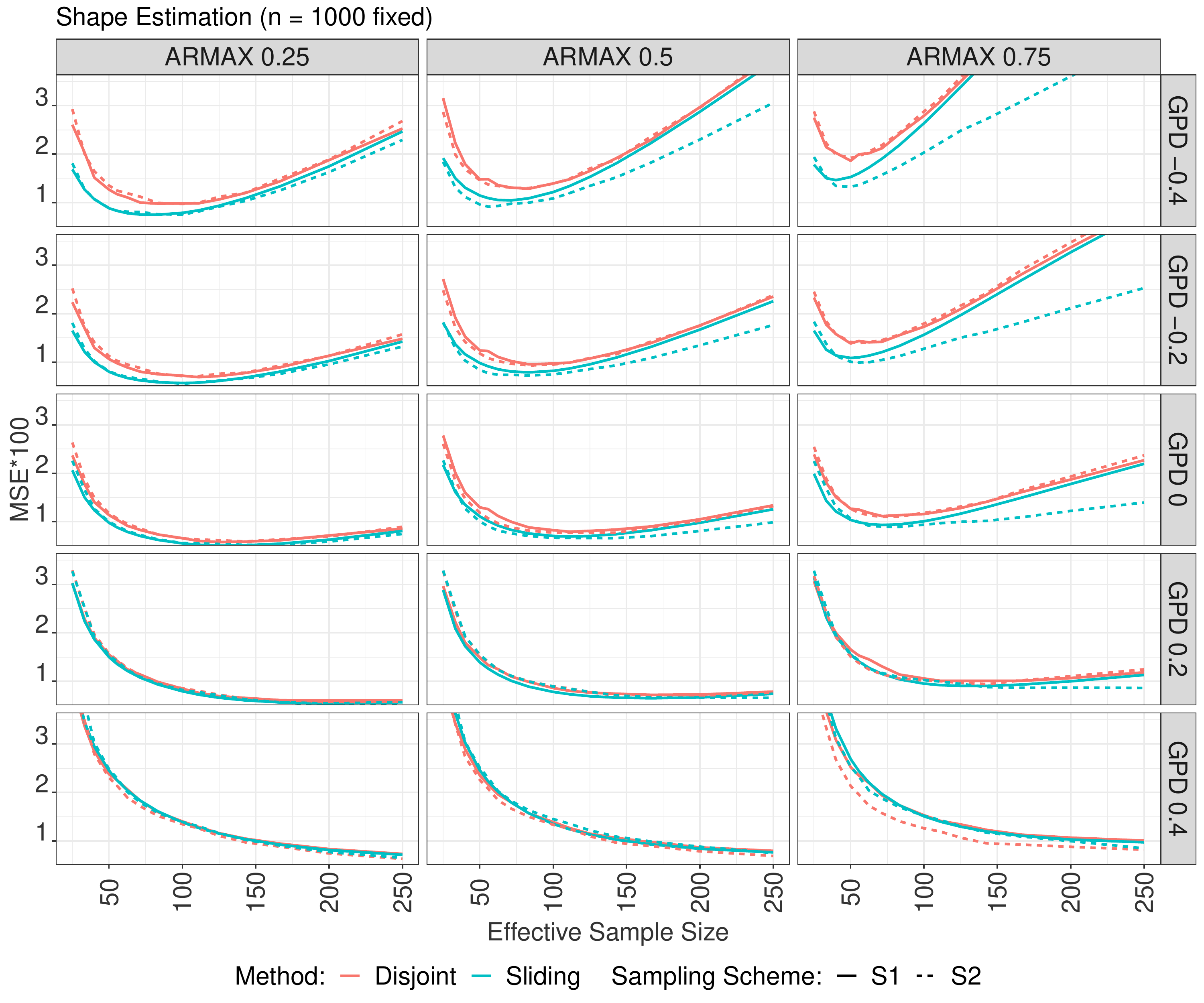}
		\caption{ \label{fig:sim_comp_s1_s2_armax}   MSE of  shape estimation for observations from sampling scheme (S1)  and (S2) based on a transformed ARMAX(1) model with GPD-margins for fixed sample size $n=1000$.
		}
	\end{figure}

		\subsection{Additional results for different marginal distributions}\label{supp:sec:hw}
		
		Block maxima obtained from an i.i.d. GPD sample are known to converge comparably fast to their limiting GEV distribution. The speed of convergence may be measured with the second order parameter, say $\rho = \rho_{\mathrm{BM}}$, which takes its values in $[-\infty, 0]$. The smaller $\rho$ is, the higher is the speed of convergence.   
		For the GPD distribution, we have $ \rho = -1$, see Section 2 in \cite{buecher2020horse}. Slower convergence is thus obtained if the second order parameter is larger than $-1$, whence we chose to (partially) repeat our simulation study for distributions such that $\rho=-1/2$.
		
		More precisely, for positive $\gamma$, we chose to consider a member from the Hall-and-Welsh (HW) distribution family, defined by its cumulative distribution function $F_\gamma(x)  = 1- x^{-1/\gamma}(1+x^{-1/(2\gamma)})/2, x \geq 1 $.  It can be shown that this distribution is in the maximum domain of attraction of $G_{\gamma}$ with  second order parameter  $\rho=-1/2$ (see Table 1 in \citealp{buecher2020horse}). For negative $\gamma$, we chose the distribution of the random variable $-1/Z$ where $Z \sim F_{|\gamma|}$, whose second order parameter is again $-1/2$ (model RHW in Table 3 in \citealp{buecher2020horse}). Finally, for $\gamma=0$, we chose to consider the distribution defined by its inverse $F^{-1}(p) = \log\{ 1/(1-p) \} \times \{ 1 + (1-p)^{1/2} \}$. Further, in order to avoid division by values close to zero when evaluating the (relative) performance of return level estimators, all distributions were shifted by adding $1$ to the simulated values.

		Simulations were carried out for all dependence structures of Section \ref{sec:sim} in the main paper and marginal distributions as described above, with $\gamma \in \{-0.4, -0.2, 0, 0.2, 0.4\}$. Again, the quantile transformation method was applied for sampling from the respective models.
		For the ease of presentation and because findings were similar, we restrict attention to models with medium (AR 0.5, CAR 0.5, ARMAX 0.5) or no (i.i.d.) dependence under sampling scheme (S2).
		
		The resulting MSE curves  are shown in Figure~\ref{fig:MSE_HW_mixedts_S2} (shape estimation) and Figure~\ref{fig:MSE_HW_mixedts_S2_RL} ($\RL(100,90)$ estimation). Regarding the latter, the MSE is computed from 
		$\{ \widehat{\RL}(T,r) - \RL(T,r) \} / \RL(T,r)$, with $T = 100$, $r = 90$ and $\RL(T,r)$ computed from a preliminary simulation imvolving $N = 10^6$ blockmaxima of independent blocks of size $r$.

		For estimation of the shape, the MSE curves for the two second order parameters are nearly identical, while some differences are visible for return level estimation. For the latter however, a direct comparison is not quite sensible, as the true values deviate from each other. Overall, the qualitative behaviour is not significantly influenced by the second order parameter, in particular when comparing disjoint and sliding blocks.

		\begin{figure}[t]	
			\centering
			\includegraphics[width=0.7\textwidth]{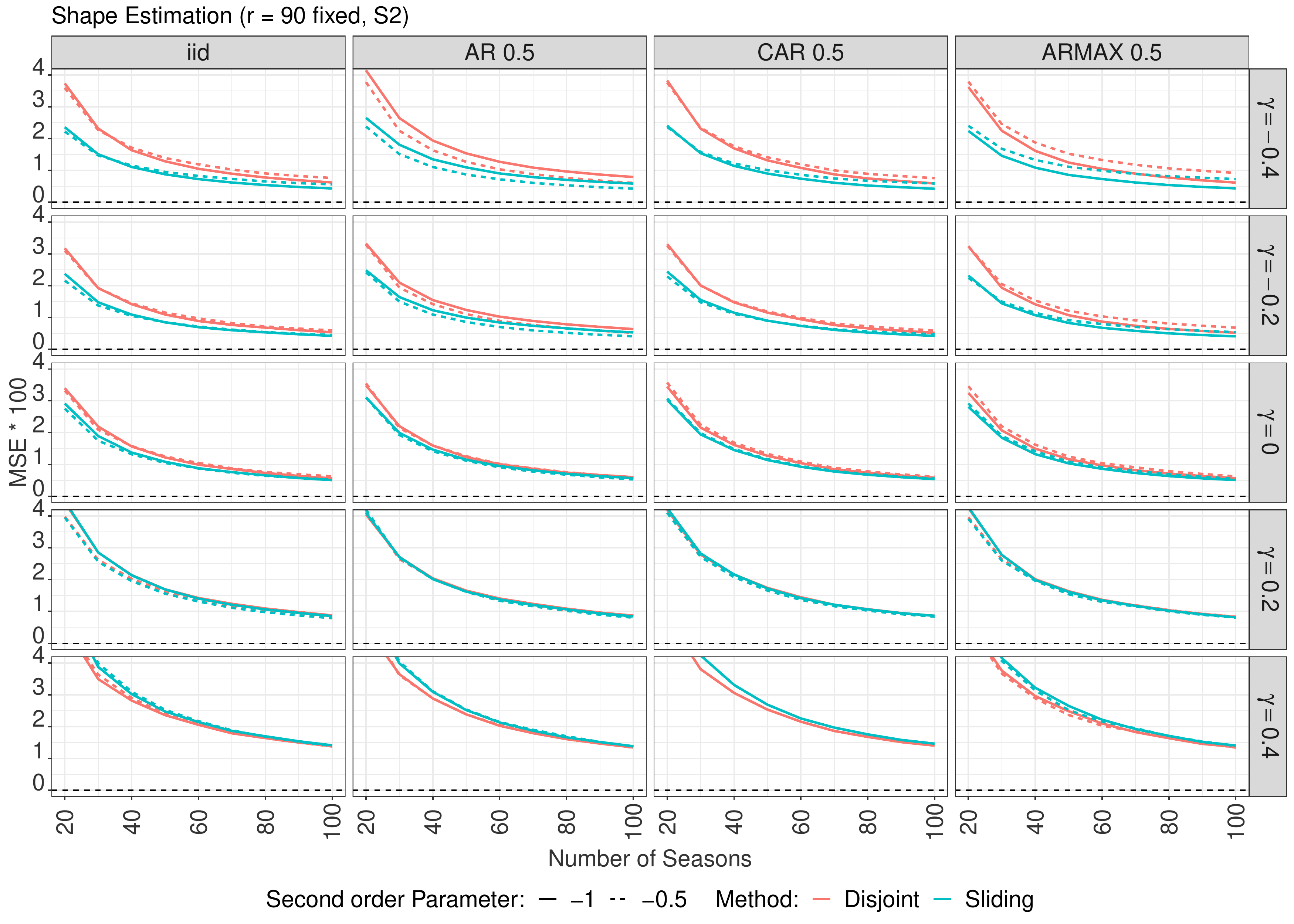}
			\caption{ \label{fig:MSE_HW_mixedts_S2}  
				MSE of shape estimation as a function of the number of seasons (blocksize $ r= 90$ is fixed) under sampling scheme (S2) and a selection of different dependence structures. The marginal distributions are attracted to $G_\gamma$, where $\gamma$ varies across rows and the second order parameter is either -1 (solid line) or -0.5 (dashed line).}
		\end{figure}
		\begin{figure}[t]	
			\centering
			\includegraphics[width=0.7\textwidth]{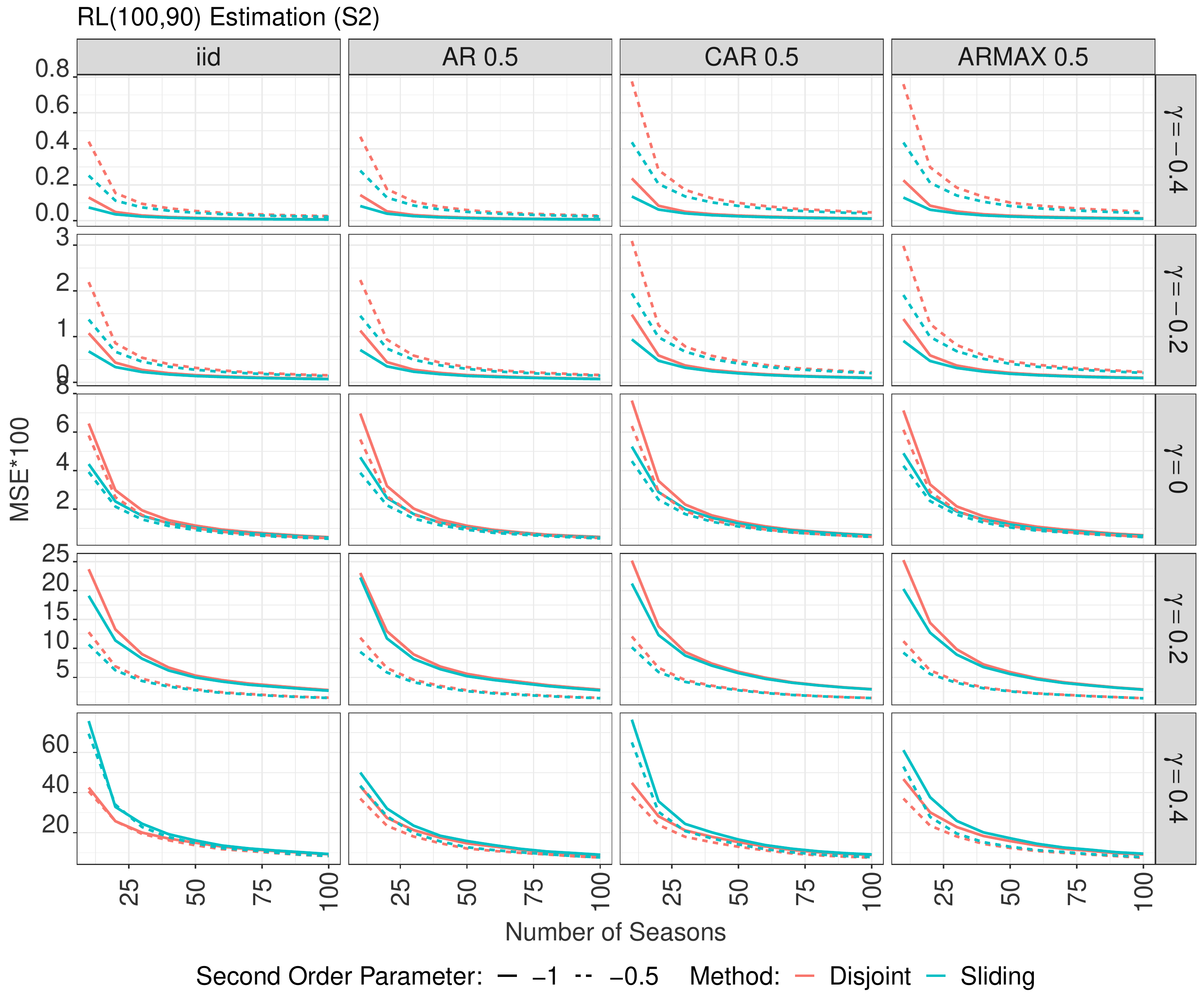}
			\caption{ \label{fig:MSE_HW_mixedts_S2_RL}  
				Rescaled MSE of $\RL(100,90)$ estimation under the same specifications as in Figure~\ref{fig:MSE_HW_mixedts_S2}. }
		\end{figure}

		\subsection{Results for comparing Maximum Likelihood and PWM Estimation}\label{supp:sec:ml}
		A (reduced) simulation study was performed to compare the PWM estimator to its most popular competitor, the (pseudo) Maximum Likelihood estimator. Attention was restricted to 20 selected models that are made up from 4 different time series models (i.i.d., AR 0.5, CAR 0.5 and ARMAX 0.5) and the 5 different GPD-margins (GPD$(\gamma)$ with $\gamma \in \{ -0.4, -0.2,0, 0.2, 0.4\}$). The sliding blocks maximum likelihood estimator was obtained by maximizing the likelihood function that results from treating the sliding blocks as independent, see \cite{BucSeg18b} for respective theoretical results in the heavy tailed case.
		
		The respective results for the estimation of $\gamma$ and $\RL(100,90)$ are summarized in Figure~\ref{fig:MLvsPWM} and Figure~\ref{fig:MLvsPWM-RL}.  Both figures are slightly manipulated in favor of the maximum likelihood estimator as all presented results are conditional on the event that $\abs{\hat\gamma_{\mathrm{ML} } - \gamma} \le 1$. The latter happens to be the case in approx.\ 95\% of the simulation runs for $n=10$ and in up to 99.5\% for $n\ge 20$; not omitting the remaining (unrealistic) cases yields quite unstable curves for the ML estimator when $\abs{\gamma} = 0.4$.
		
		The results reveal that the PWM estimator has a tendency to be superior for small sample sizes while the maximum likelihood estimator is superior for large sample sizes; to the best of our knowledge this is a  usual view of the two estimators among applied statisticians. For shape estimation, smaller shapes yield better results for the PWM estimator, while for return level estimation, the picture is almost reversed.
		This seems to be an interesting aspect that could be confirmed in an extensive simulation study in future research.

	\begin{figure}[tbh!]	
		\centering
		\includegraphics[width=0.9\textwidth]{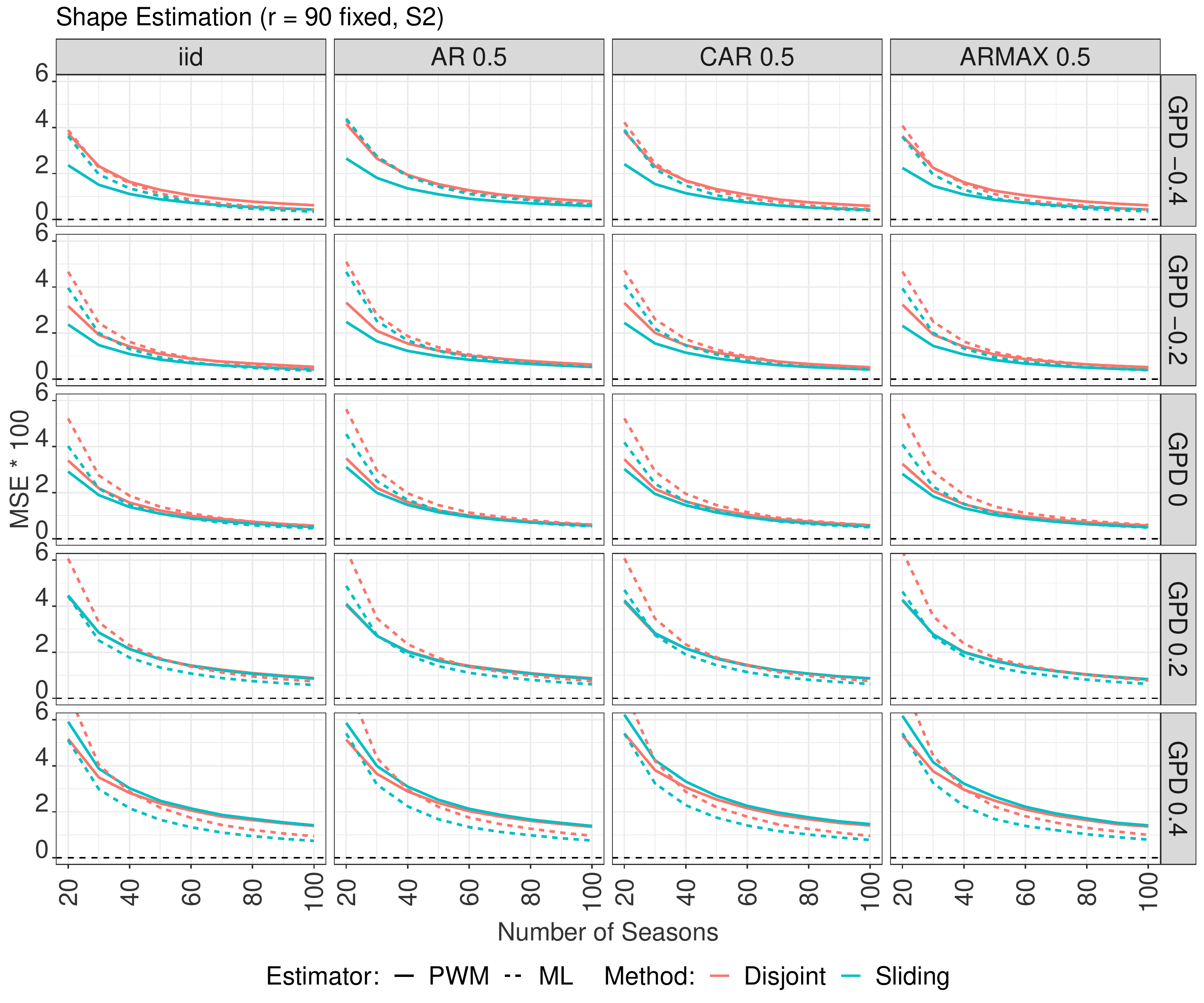}\\[5mm]
		\includegraphics[width=0.9\textwidth]{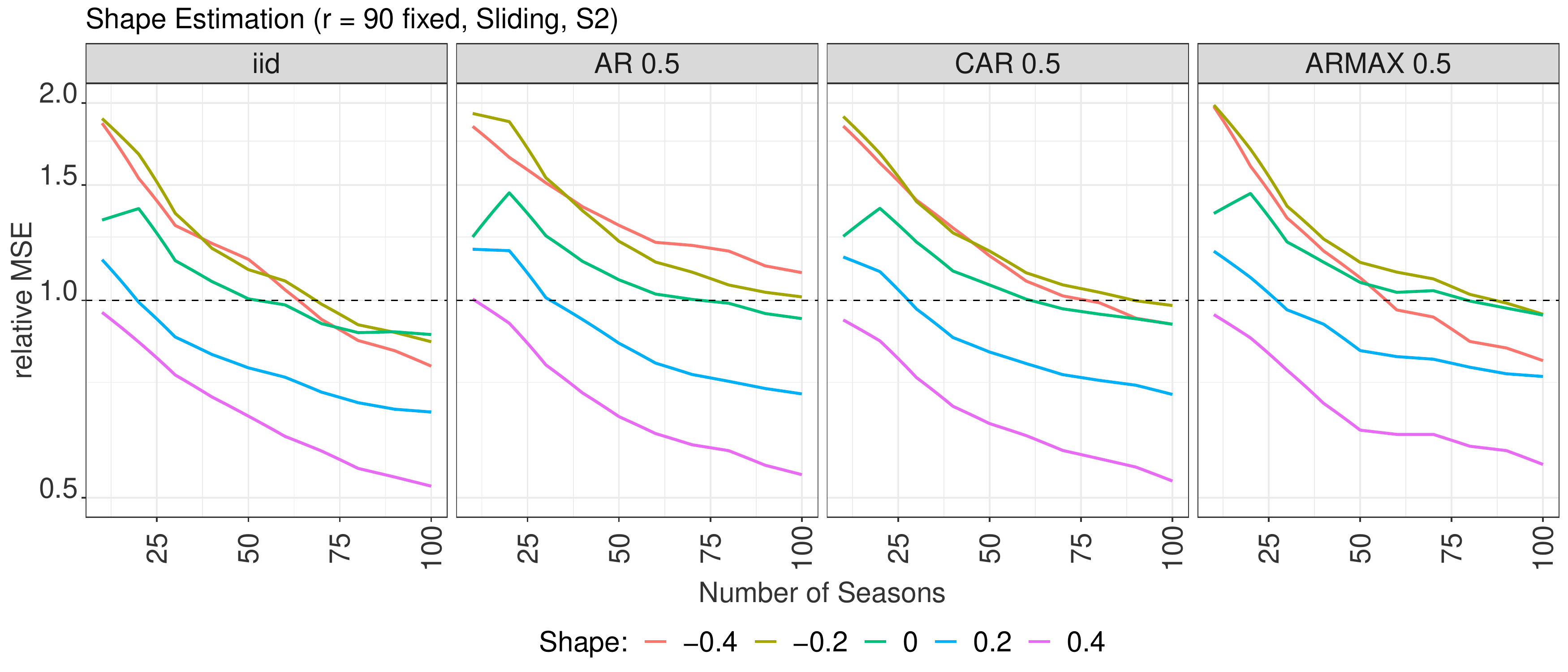}
		\caption{\label{fig:MLvsPWM} Top: MSE obtained from PWM (solid line) and ML (dashed line) shape estimation based on disjoint and sliding blocks. Bottom: Relative MSE of sliding blocks shape estimation (MSE of ML estimation divided by MSE of PWM estimation). }
	\end{figure}

	\begin{figure}[tbh!]	
		\centering
		\includegraphics[width=0.9\textwidth]{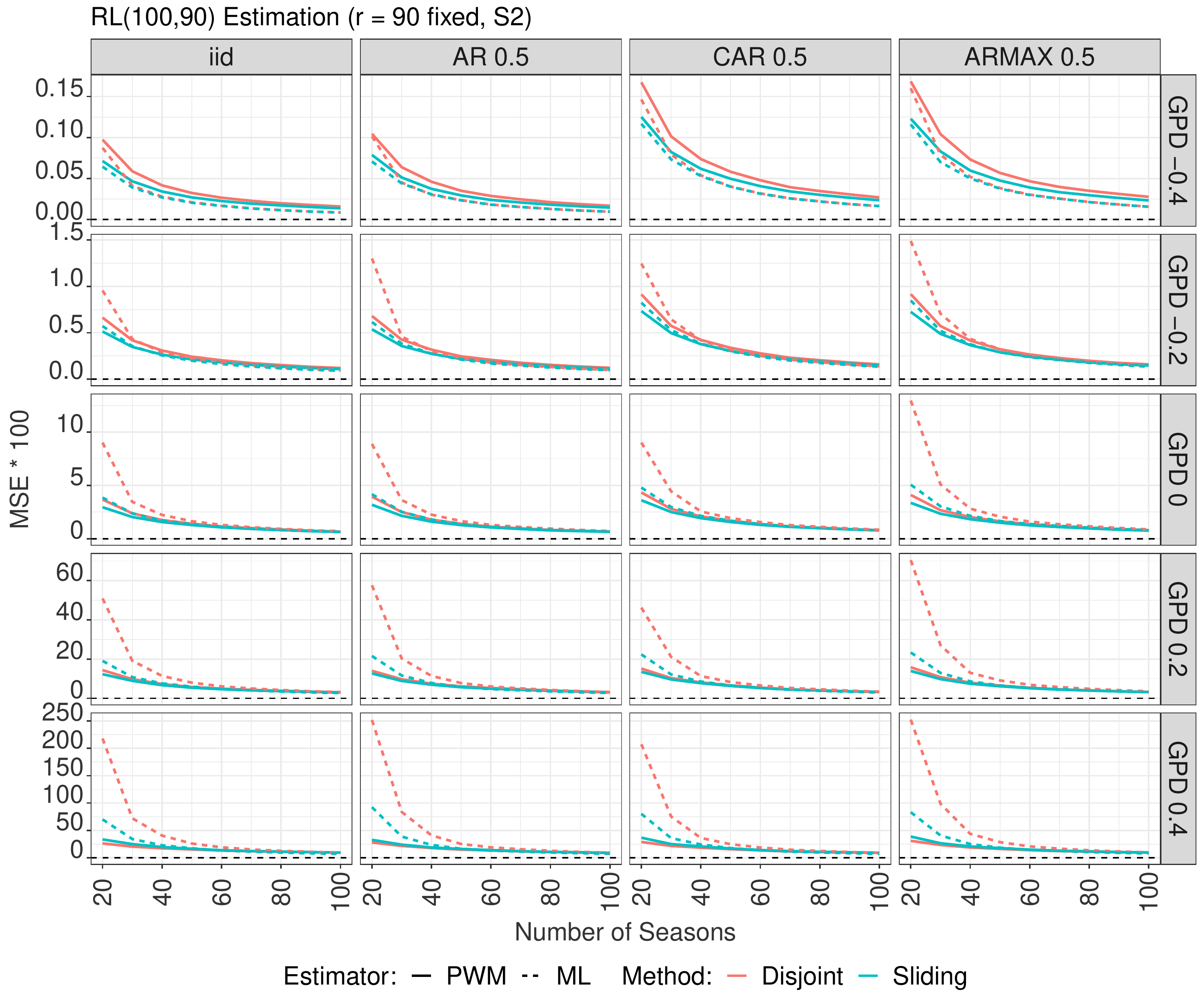}\\[5mm]
		\includegraphics[width=0.9\textwidth]{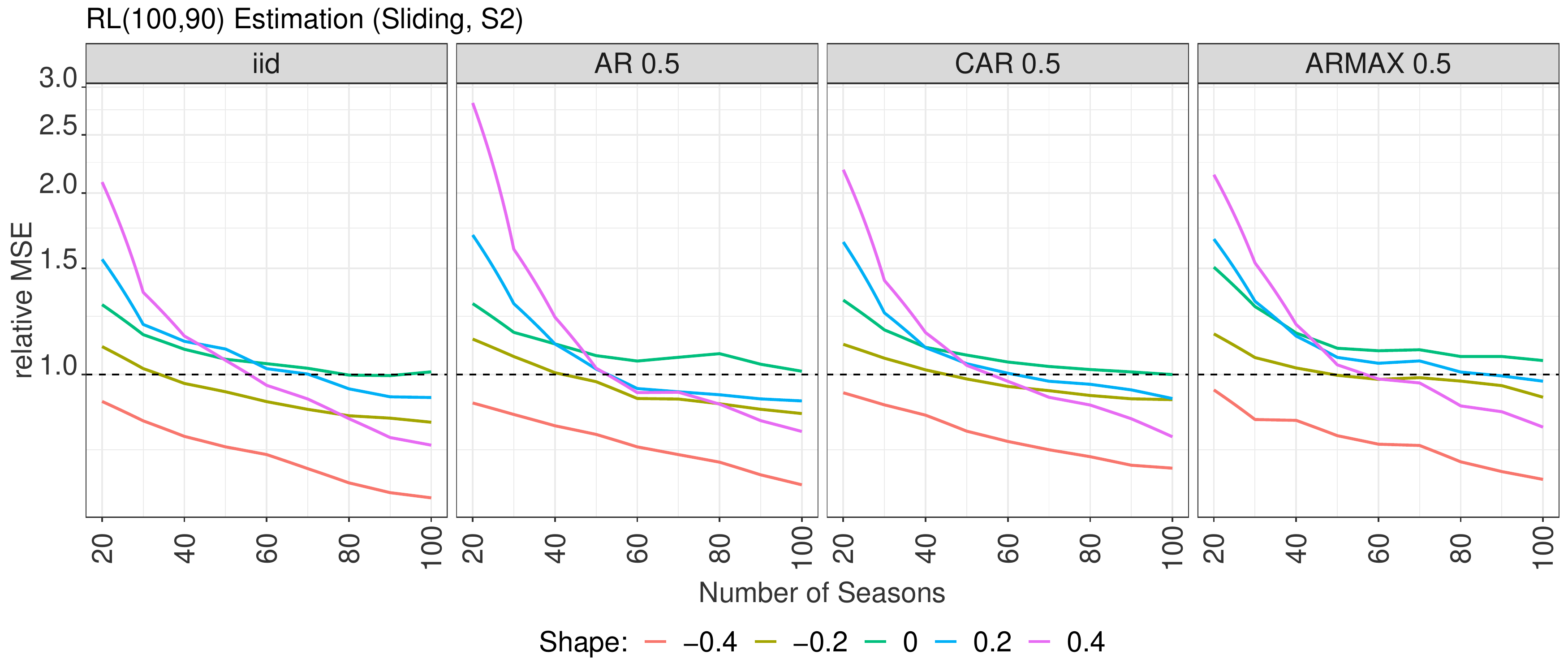}
		\caption{\label{fig:MLvsPWM-RL} Same as Figure~\ref{fig:MLvsPWM}, but for the estimation of $\RL(100,90)$.}
	\end{figure}

		\section{GEV-fit examination for the case study}\label{supp:sec:cs}
		To assess wether the fitted GEV distributions in the case study are plausible, we generated QQ-plots, which can be found in Figure~\ref{fig:cs:qqplot} and reveal a remarkably good fit.
		
		\begin{figure}[tbh!]	
			\centering
			\includegraphics[width=0.7\textwidth]{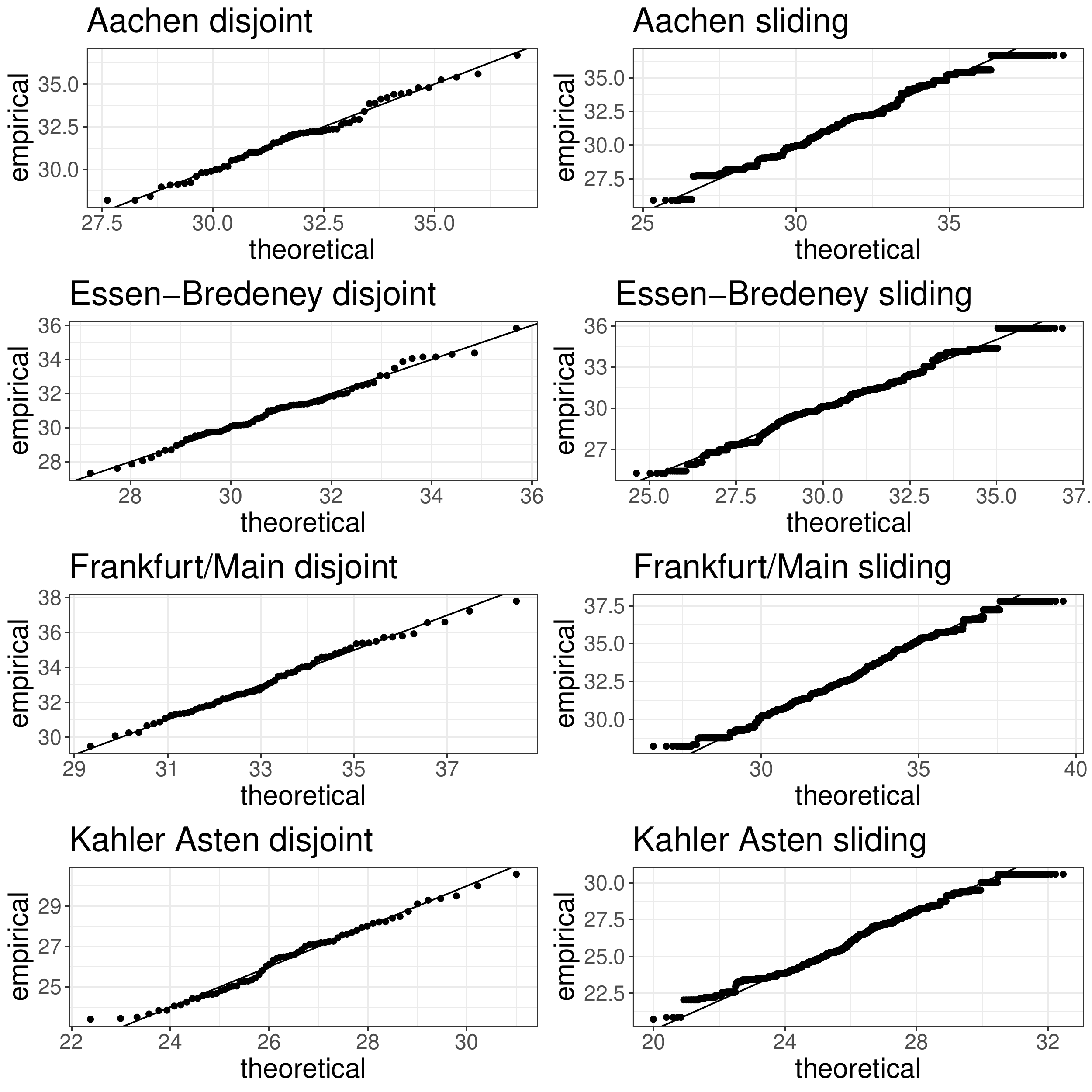}
			\caption{\label{fig:cs:qqplot} QQ-Plots for the fitted models based on disjoint BM (first column) and sliding BM (second column).
			}
		\end{figure}

\end{appendix}

\section*{Acknowledgements}
This work has been supported by the integrated project ``Climate Change and Extreme Events - ClimXtreme Module B - Statistics (subproject B3.3)'' funded by the German Bundesministerium für Bildung und Forschung, which is gratefully acknowledged. 
Computational infrastructure and support were provided by the Centre for Information and Media Technology at Heinrich Heine University Düsseldorf.

The authors are grateful to three unknown referees and an associate editor for their constructive comments that helped to improve the presentation substantially. The authors also appreciate valuable comments by  Petra Friederichs, Marco Oesting and multiple participants of the Extreme Value Analysis (EVA) Conference 2021 in Edinburgh.

\clearpage

\bibliographystyle{imsart-number} 
\bibliography{biblio}	

\end{document}